\documentclass[reqno]{amsart}
\usepackage[dvipsnames]{xcolor}
\usepackage{graphicx}
\usepackage{hyperref}

\newtheorem{theorem}{Theorem}
\newtheorem{corollary}{Corollary}
\newtheorem{lemma}{Lemma}
\newtheorem{proposition}{Proposition}
\theoremstyle{definition}
\newtheorem{definition}{Definition}
\theoremstyle{remark}
\newtheorem{remark}{Remark}
\numberwithin{equation}{section}
\numberwithin{lemma}{section}
\numberwithin{remark}{section}

\newcommand\R{\mathbb{R}}
\newcommand\Z{\mathbb{Z}}

\newcommand\N{\mathbb{N}}
\newcommand\zp{\zeta_+}
\newcommand\zm{\zeta_-}
\newcommand\zpm{\zeta_\pm}
\newcommand\zn{\zeta_0}

\newcommand\pp{\boldsymbol{\phi}}
\newcommand\alphab{\boldsymbol{\alpha}}
\newcommand\betab{\boldsymbol{\beta}}
\newcommand\Omegab{\boldsymbol{\Omega}}
\newcommand\ab{\boldsymbol{a}}
\newcommand\uu{\boldsymbol{u}}
\newcommand\ww{\boldsymbol{w}}
\newcommand\zz{\boldsymbol{z}}
\newcommand\hh{\boldsymbol{h}}
\newcommand\kk{\boldsymbol{k}}
\newcommand\jj{\boldsymbol{j}}
\newcommand\gb{\boldsymbol{g}}
\newcommand\ff{\boldsymbol{f}}
\newcommand\vv{\boldsymbol{v}}
\newcommand\AAA{\boldsymbol{A}}
\newcommand\BB{\boldsymbol{B}}
\newcommand\DD{\boldsymbol{D}}
\newcommand\EE{\boldsymbol{E}}
\newcommand\FF{\boldsymbol{F}}
\newcommand\GG{\boldsymbol{G}}
\newcommand\HH{\boldsymbol{H}}
\newcommand\JJ{\boldsymbol{J}}
\newcommand\LL{\boldsymbol{L}}
\newcommand\RR{\boldsymbol{R}}
\newcommand\TT{\boldsymbol{T}}
\newcommand\VV{\boldsymbol{V}}
\newcommand\cB{\mathcal{B}}
\newcommand\cC{\mathcal{C}}
\newcommand\cF{\mathcal{F}}
\newcommand\cA{\mathcal{A}}
\newcommand\cH{\mathcal{H}}
\newcommand\cI{\mathcal{I}}
\newcommand\cJ{\mathcal{J}}
\newcommand\cT{\mathcal{T}}
\newcommand\cG{\mathcal{G}}
\newcommand\cK{\mathcal{K}}
\newcommand\cLL{\mathcal{L}}
\newcommand\cP{\mathcal{P}}
\newcommand\cE{\mathcal{E}}
\newcommand\cM{\mathcal{M}}
\newcommand\cN{\mathcal{N}}
\newcommand\cR{\mathcal{R}}
\newcommand\cZ{\mathcal{Z}}

\newcommand\xib{\boldsymbol{\xi}}
\newcommand\ONE{\mathrm{1\kern-0.25em l}}
\newcommand\MM{M}
\newcommand\Mb{\boldsymbol{M}}
\newcommand\OO{O}
\newcommand\TTT{T_2}
\newcommand\RRR{R_2}
\newcommand\SSS{S_2}
\newcommand\SSb{\boldsymbol{S}}
\newcommand\TTb{\boldsymbol{T}}
\newcommand\NN{N}
\newcommand\DSG{{\textnormal{dsG}}}
\newcommand\SSG{{\textnormal{ssG}}}
\newcommand\SG{{\textnormal{sG}}}
\newcommand\cL{L}
\newcommand\ud{\, \textnormal{d}}
\newcommand{\inprod}[2]{\left\langle{#1},{#2}\right\rangle}
\newcommand{\norm}[1]{\left\| #1 \right\|}
\newcommand{\XX}{X_\varepsilon}
\newcommand{\PSI}{\psi_{A,B}}
\newcommand{\CHI}{\chi_A}
\DeclareMathOperator\sech{sech}

\DeclareMathOperator\Id{\boldsymbol I}

\title[Asymptotic stability of kinks for scalar fields]
{A sufficient condition for asymptotic stability of kinks in general (1+1)-scalar field models}

\author[M. Kowalczyk]{Micha{\l} Kowalczyk}
\address{Departamento de Ingenier\'{\i}a Matem\'atica and Centro
de Modelamiento Matem\'atico (UMI 2807 CNRS), Universidad de Chile, Casilla
170 Correo 3, Santiago, Chile}
\email {kowalczy@dim.uchile.cl}

\author[Y. Martel]{Yvan Martel}
\address{CMLS, \'Ecole polytechnique, CNRS, Institut Polytechnique de Paris, 91128 Palaiseau Cedex, France}
\email{yvan.martel@polytechnique.edu}

\author[C. Mu\~noz]{Claudio Mu\~noz}
\address{CNRS and Departamento de Ingenier\'{\i}a Matem\'atica and Centro
de Modelamiento Matem\'atico (UMI 2807 CNRS), Universidad de Chile, Casilla
170 Correo 3, Santiago, Chile}
\email{cmunoz@dim.uchile.cl}

\author[H. Van Den Bosch]{Hanne Van Den Bosch}
\address{Departamento de Ingenier\'{\i}a Matem\'atica and Centro
de Modelamiento Matem\'atico (UMI 2807 CNRS), Universidad de Chile, Casilla
170 Correo 3, Santiago, Chile}
\email{hannevdbosch@gmail.com}

\thanks{M.K. was partially funded by Chilean research grant FONDECYT 1170164.
C.M. was partially funded by Chilean research grant FONDECYT 1191412.
H.VDB. was partially funded by Chilean research grants FONDECYT 3180059 and REDI 170157.
M.K., H.VDB. and C.M. were partially funded by project France-Chile ECOS-Sud C18E06, MathAmSud EEQUADDII 20-MATH-04 and CMM Conicyt PIA AFB170001.
Part of this work was done while Y.M. was visiting Centro
de Modelamiento Matem\'a\-tico, Universidad de Chile, whose hospitality is acknowledged.}

\begin{document}
\begin{abstract}
We study stability properties of kinks for the (1+1)-dimensional nonlinear scalar field theory models
\begin{equation*}
\partial_t^2\phi -\partial_x^2\phi + W'(\phi) = 0, \quad (t,x)\in\R\times\R.
\end{equation*}
The orbital stability of kinks under general assumptions on the potential $W$ is a consequence of energy arguments. Our main result is the derivation of 
a simple and explicit sufficient condition on the potential $W$ for the asymptotic stability of a given kink. 
This condition applies to any static or moving kink, in particular no symmetry assumption is required.
Last, motivated by the Physics literature, we present applications of the criterion to the $P(\phi)_2$ theories and the double sine-Gordon theory.
\end{abstract}
\maketitle
\tableofcontents

\section{Introduction}
\subsection{General setting}
This article is concerned with the stability of kinks for general (1+1)-dimensional nonlinear scalar field theory models
\begin{equation}\label{eq:W}
\partial_t^2\phi -\partial_x^2\phi + W'(\phi) = 0, \quad (t,x)\in\R\times\R.
\end{equation}
This model rewrites as a first order system for $\pp=(\phi,\partial_t\phi)=(\phi_1,\phi_2)$
\begin{equation}\label{syst}
\begin{cases}
\partial_t \phi_1 = \phi_2 \\
\partial_t \phi_2 = \partial_x^2 \phi_1 -W'(\phi_1). 
\end{cases}
\end{equation}
The assumptions on the potential $W$ are standard
\begin{equation}\label{on:W}
\begin{cases} 
\mbox{$W:\R\to [0,\infty)$ is of class $\cC^3$,}\\
\mbox{$W$ has at least two zeros $\zm,\zp\in \R$ such that $\zm<\zp$,}\\
\mbox{$W'(\zpm)=0$, $W''(\zpm)>0$, and $W>0$ on $(\zm,\zp)$.}
\end{cases}
\end{equation}
Such assumptions ensure that~\eqref{eq:W} admits a \emph{kink} corresponding to the heteroclinic orbit of the equation $h''=W'(h)$ connecting the two consecutive vacua $\zm$ and $\zp$, and whose main properties are gathered in the following statement
(see \S\ref{s:2.1} for a proof and further properties).
\begin{lemma}[Static kink]\label{pr:H}
Assume~\eqref{on:W}.
There exists a solution $H$ of class $\cC^4$ of the equation
\begin{equation*}
\begin{cases}
H'' = W'(H) \mbox{ on $\R$},\\
\lim_{-\infty} H = \zm,\quad \lim_{+\infty} H=\zp,
\end{cases}
\end{equation*}
unique up to translation, which satisfies
$H'>0$ on $\R$. Moreover, there exist  $C>0$ and $\omega>0$ such that for any $k=1,2,3,4$,
\begin{equation}\label{eq:Hasym}
|H(x)-\zpm|\leq C e^{\mp \omega x},\quad |H^{(k)}(x)|\leq C e^{-\omega |x|} \quad \mbox{on $\R$}.
\end{equation}
\end{lemma}
The solution $\HH=(H,0)$ of~\eqref{syst} is called the static kink.
For $c\in (-1,1)$ we define $\gamma=\gamma(c)\in [1,\infty)$ and the functions $H_c$, $\HH_c$ by
\[
\gamma = \frac 1 {\sqrt{1-c^2}} , \quad H_c(x)=H(\gamma x),\quad 
\HH_{c}(x)= 
\begin{pmatrix}
H_{c}(x) \\[2pt]
- c H'_{c}(x)
\end{pmatrix}.
\]
The Lorentz boosted versions of the kink are defined by
$\pp(t,x)=\HH_c(x-ct)$, for any $c\in (-1,1)$.
These solutions, together with their translated versions, form the kink family of the model~\eqref{syst} related to the consecutive wells $\zeta_-$, $\zeta_+$ of the potential $W$.

Recall the conservation laws of the model
\begin{align*}
\cE[\pp] & = \frac 12 \int \left[ \phi_2^2 + (\partial_x \phi_1)^2 + 2 W(\phi_1)\right],\tag{Energy}\\
\cP[\pp] & = \int \phi_2 \partial_x \phi_1.\tag{Momentum}
\end{align*}
Since $W\geq 0$ on $\R$, the set of functions $\pp\in L^1_{\rm loc}(\R)\times L^1_{\rm loc}(\R)$ for which
the energy is finite is
\[
\EE=\left \{\pp\in L^1_{\rm loc}(\R)\times L^1_{\rm loc}(\R) :
\partial_x \phi_1 \in L^2(\R),\ \sqrt{W(\phi_1)}\in L^2(\R),\ \phi_2\in L^2(\R)\right\}.
\]
To study the stability of $\HH$, we introduce the following subset $\EE_{\HH}$ of $\EE$
\[
\EE_{\HH}=\left\{\pp\in \EE : \pp - \HH \in H^1(\R)\times L^2(\R)\right\}.
\]
By the properties of $W$ in~\eqref{on:W} and Lemma~\ref{pr:H}, we observe that 
\[
\EE_{\HH} = \left\{\pp\in L^1_{\rm loc}(\R)\times L^1_{\rm loc}(\R) :
\partial_x \phi_1 \in L^2(\R),\ \phi_1-H\in L^2(\R),\ \phi_2\in L^2(\R)\right\}.
\]
For $c\in (-1,1)$ and $\vv=(v_1,v_2)\in H^1(\R)\times L^2(\R)$, we set
\begin{equation*}
\norm{\vv}_{c}^2= \gamma^{-1} \norm{\partial_x v_1}_{L^2}^2
+\gamma \norm{v_1}_{L^2}^2+ \gamma \norm{v_2+c\partial_x v_1}_{L^2}^2,
\end{equation*}
and for $R>0$,
\begin{equation*}
\norm{\vv}_{c,R}^2= \gamma^{-1} \norm{\partial_x v_1}_{L^2(|x|<R)}^2
+\gamma \norm{v_1}_{L^2(|x|<R)}^2+ \gamma \norm{v_2+c\partial_x v_1}_{L^2(|x|<R)}^2.
\end{equation*}

\subsection{Main results}
For the sake of completeness, we state the result of orbital stability of the kink family in the general context~\eqref{on:W}.
The proof given in Sections~\ref{S:2} and~\ref{S:3} relies on standard energy arguments; see also~\cite{MR678151,Lohe}.

\begin{theorem}[Orbital stability]\label{th:1}
Assume~\eqref{on:W}.
There exist $\delta_*>0$ and $C_*>0$ such that for any $c_0\in (-1,1)$
and any $\pp^{in}\in \EE_{\HH}$ with
\[
\norm{\pp^{in} - \HH_{c_0}}_{c_0} \leq \delta_*,
\]
there exists a unique global solution $\pp\in \cC(\R,\EE_{\HH})$ of~\eqref{eq:W} 
with $\pp(0)=\pp^{in}$. Moreover, for all $t\in \R$
\begin{equation*}
\inf_{y\in \R} \norm{\pp(t,\cdot+y) - \HH_{c_0}}_{c_0}
\leq C_* \norm{\pp^{in} - \HH_{c_0}}_{c_0}.
\end{equation*}
\end{theorem}
\begin{remark}
In the statement of Theorem~\ref{th:1}, the constants $\delta_*$ and $C_*$ are independent of the speed $c_0$ of the kink. To state such an optimal stability result, one has to use the norm $\|\cdot\|_{c_0}$.
Moreover, while equivalent to $\|\cdot\|_{H^1\times L^2}$, this norm seems more natural since it measures the perturbation in the moving frame of the kink:
if $\pp$ is a Lorentz boosted version of the kink with speed $c$, then
$\phi_2+c\partial_x \phi_1=0$.
\end{remark}

We turn to the notion of asymptotic stability.

\begin{definition}[Asymptotic stability]\label{def:1} Assume \eqref{on:W}.
The kink $H$ is said to be \emph{asymptotically stable} if
for any $c_0\in (-1,1)$ there exists $\delta_0\in (0,\delta_*]$ such that
for any $\pp^{in}\in \EE_{\HH}$ with
\[
\norm{\pp^{in} - \HH_{c_0}}_{c_0}\leq \delta_0,
\]
there exist $c_-,c_+ \in (-1,1)$ and a $\cC^1$ function $y:\R\to \R$ with
$\lim_{t\to\pm\infty} \dot y(t)= c_\pm$,
such that
the global solution $\pp$ of~\eqref{eq:W} with $\pp(0)=\pp^{in}$ satisfies, for any $R>0$
\begin{equation*}
\lim_{t\to \pm \infty} \|\pp(t,\cdot +y(t)) - \HH_{c_\pm}\|_{c_\pm,R} =0.
\end{equation*}
\end{definition}

This definition gives a strong notion of asymptotic completeness of the family of kinks defined from $H$, requiring that any translated and Lorentz boosted version of the kink is asymptotically stable, under any small perturbation in the energy space. In this definition, 
the solution $\pp(t)$ approaches as $t\to\pm\infty$ the final moving kinks $\HH_{c_\pm}(x-y(t))$  
\emph{locally in space} around the time-dependent kink position $y(t)$.

The goal of this article is to exhibit a  simple and general \emph{sufficient condition} on the potential $W$ for asymptotic stability, without assuming any symmetry property for the potential nor for the kink.
To formulate the condition, we introduce the \emph{transformed potential}
$V:(\zm,\zp)\to \R$ defined by
\begin{equation}\label{def:V}
V=- W\left(\log W\right)''.
\end{equation}
The main result of this article is the following theorem, proved in Section~\ref{S:4}.
\begin{theorem}[Sufficient condition for asymptotic stability] \label{th:2}
Assume~\eqref{on:W}. If the transformed potential $V$ satisfies $V'\not\equiv 0$ on $(\zm,\zp)$ and
\begin{equation}\label{on:V}
\mbox{there exists $\zn\in [\zm,\zp]$
such that $(\phi-\zeta_0)V'(\phi)\leq 0$ for all $\phi\in (\zm,\zp)$,}
\end{equation}
then the kink $H$ is asymptotically stable.
\end{theorem}

\begin{remark}
From the orbital stability result, in the context of Theorem~\ref{th:2}, it also holds, for a constant $C>0$
\begin{equation*}
\gamma_0^2 \sup_{t\in \R} |\dot y(t)-c_0|+ \gamma_0^2 |c_\pm-c_0|\leq C \norm{\pp^{in} - \HH_{c_0}}_{c_0}.
\end{equation*}
Moreover, the proof of Theorem~\ref{th:2} yields the following information: there exists a $\cC^1$ function $c:\R\to (-1,1)$ satisfying
\[
\lim_{\pm \infty} c=c_\pm\quad\mbox{and}\quad \sup_{t\in \R}|\dot c(t)-c_0|\leq C \norm{\pp^{in} - \HH_{c_0}}_{c_0},
\]
and such that, for any $R>0$
\begin{equation}\label{eq:int0T}
\int_{-\infty}^{+\infty}
\|\pp(t,\cdot +y(t)) - \HH_{c(t)}\|_{(H^1\times L^2)(|x|<R)}^2 \ud t <\infty.
\end{equation}
\end{remark}

The sufficient condition \eqref{on:V} is related to the repulsivity of the potential $V(H)$ (see Lemma~\ref{rk:P}) that appears after transformation by factorization of the linearized equation around the kink.
This condition implies that the model does not have internal mode nor resonance for the kink $H$, which are known spectral obstructions for estimates of the form \eqref{eq:int0T}.
The factorization procedure is a key tool of the proof of Theorem~\ref{th:2}. 
 We refer to \S\ref{s.4.6} for its heuristic presentation and a more technical discussion on the condition~\eqref{on:V}.

Interestingly, one easily checks that the transformed potential is constant for the (integrable) sine-Gordon equation corresponding to the potential
\begin{equation}\label{W:sineG}
W_\SG(\phi)=1-\cos \phi.
\end{equation}
The converse statement is established in \S\ref{S:5.2}.
This threshold situation $V'\equiv 0$ is excluded by Theorem~\ref{th:2}, which is consistent with 
the fact that the kink of the sine-Gordon equation is \emph{not} asymptotically stable in the sense of Definition~\ref{def:1} (see references in \S\ref{S:1.4}). This observation also motivates the study of some models approximating the sine-Gordon equation; see below and \S\ref{S:applications}.

For the $\phi^4$ model, corresponding to the potential
\begin{equation}\label{W:phi4}
W_4(\phi)=(\phi^2-1)^2,
\end{equation}
the asymptotic stability of the kink is usually conjectured but seems very challenging to prove in general mainly because of the presence of one even resonance and one odd internal mode.
In particular, the condition~\eqref{on:V} does not hold for the $\phi^4$ model (see \S\ref{S:5.3.1})
and Theorem~\ref{th:2} is inconclusive in this case.
However, the main result in \cite{KMM} establishes the asymptotic stability of the static kink in the case of \emph{odd} perturbations, avoiding the resonance, but taking into account the internal mode. 
In the same context, for odd initial data in weighted spaces,
a very recent result~\cite{DM} shows refined decay estimates on the solution, up to a time of
size $\delta_0^{-\beta}$, where $\delta_0$ is the size of the initial perturbation and $\beta$
is any number less than $4$.

The above two classical models contain the main obstructions and difficulties encountered in addressing asymptotic stability and continue to be great challenges.

This being said, we point out that  Theorem~\ref{th:2} allows us to prove new results of 
asymptotic stability for several models from the Physics literature. In Section~\ref{S:applications},
the sufficient condition~\eqref{on:V} is first checked for the $\phi^6$ model
\begin{equation}\label{W:phi6}
W_6(\phi)=\phi^2(\phi^2-1)^2,
\end{equation}
(see Theorem~\ref{th:phi6})
and then for generalized versions of the $\phi^{4n}$ and $\phi^{4n+2}$ models for a large range of parameters
(see Theorems~\ref{th:phi8}, \ref{th:phi10} and~\ref{th:lambdan}).
It is also striking that Theorem~\ref{th:2} implies asymptotic stability of kinks
for several approximations of the sine-Gordon equation
in the $P(\phi)_2$ and double sine-Gordon theories, \emph{arbitrarily close} to the sine-Gordon model
(see Theorems~\ref{th:Wt4n2}, \ref{th:Wt4n} and~\ref{th:DSG1}).

\subsection{References}\label{S:1.4}
The Physics literature provides many references motivating the mathematical study of the dynamical properties of kinks for one-dimensional  scalar field models (see \emph{e.g.} \cite{DP,Goldman,Kevrekidisbook,Lamb,MaSutbook,PeSc,Va,ViSh,MR2318156}).
Two articles have especially motivated the applications considered in Section~\ref{S:applications}:
\cite{Lohe} for the $P(\phi)_2$ theories, and ~\cite{Campbell} for the double sine-Gordon model
(see also \cite{BG,Gani2,Khare}).

Basic properties and orbital stability of kinks for general scalar field models were studied in several previous articles, such as~\cite{AIMa,MR678151,HT,Lohe}.

The results stated in Section~\ref{S:applications}, together with results in~\cite{AMP2,DM,KMM} concern the asymptotic stability of kinks for the $P(\phi)_2$ and double sine-Gordon theories.
Previous works are devoted to asymptotic stability for various related models.

First, the question is similar to the asymptotic stability of the solitons of the nonlinear focusing Klein-Gordon equation, addressed in~\cite{BCS,KMM4,KNS}. 

Second, several authors have considered the model~\eqref{eq:W} under different assumptions on the potential.
In~\cite{KK1,KK2}, the authors consider natural spectral assumptions (absence of resonance, knowledge of internal modes)
but they also need a more technical flatness condition of the potential near the wells (related to assumptions in~\cite{bus_per1,bus_per2}). Moreover, initial data in~\cite{KK1,KK2} are taken in weighted spaces.
A three-dimensional version of the $\phi^4$ model is studied in~\cite{MR2373326}, using dispersive estimates for free solutions available in higher space dimensions.

The sine-Gordon model has been the object of many studies, mainly due to its physical relevance and integrable structure (see the general references above).
Recall from~\cite{CQS} that exceptional periodic solutions called \emph{wobbling kinks} arbitrarily close to the kink,
are obstacle to the asymptotic stability of the kink in the energy space. 
However, there is no known obstruction to asymptotic stability 
in a different topology (this corrects~\cite[End of Remark 1.3]{KMM}).
We refer to~\cite{AMP2} for asymptotic stability of the kink for perturbations with symmetry.

For the sine-Gordon and the $\phi^4$ models, breathers, kink-antikink solutions 
and multi-solitons were studied mathematically in~\cite{AMP,JKL,MuPal}.
Other results concern the nonexistence of breathers for nonintegrable models~\cite{denzler,KMM1,kruskal_segur,Segur}.

More generally, the asymptotic stability of kinks is closely related to the asymptotic behavior of small solutions of nonlinear Klein-Gordon or wave-type equations, which has been studied by many different techniques and various authors, for constant or variable coefficients. 
Pioneering works on the subject are~\cite{K1,K2,Shatah1,SoWe1,SoWe2,Delort}.
We also refer to~\cite{Delort_fourier,MR2056833, Bam_Cucc,DM, GP, HN,HN1a,LLS,LLS2,LS1,LS2,LS3,Ste} for notable progress on this question,
in different directions. For the special case of the wave equation in one dimension, we refer to~\cite{LT,WY}.

Last, recall that the question of the asymptotic stability of solitary waves has been addressed by various techniques for nonlinear dispersive models, like the generalized Korteweg de Vries equation (generalization of the classical KdV equation) and variants of the nonlinear Schr\"odinger equation, with or without potential,
which have some analogy with the models considered here.
We briefly quote some results and refer to the review~\cite{KMM3} for more references.
The techniques used and developed in the present paper are partly reminiscent of the ones introduced in~\cite{Ma,MMjmpa,MM1,MM2} to prove the asymptotic stability of the solitons of the gKdV equations.
See also~\cite{CMPS} for an extension to a 2-dimensional variant of this model.
Recall that the first result on such problem was obtained in~\cite{PW2} by different methods.
Still different techniques, involving  the special algebraic structure of the modified KdV equation and refined linear estimates were used in~\cite{GPR}. For the nonlinear Schr\"odinger equations, we mention the 
pioneering articles~\cite{bus_per1,bus_per2} 
and~\cite{cuccagna_3,cuccagna_4,CuMa,Delort2,KS,Sc06,Sc07,We85}, as well as
\cite{cuc_cub_nls} concerning the integrable case. See also~\cite{FMR,MR} in the blow up context.
The article~\cite{Z} concerns kinks of nonlinear Schr\"odinger equations.

For the proof of Theorem~\ref{th:2}, we follow and develop the approach to asymptotic stability of kinks and solitons for wave-type equations initiated in~\cite{KMM} and~\cite{KMM4}. 
In particular, we adapt the method to the case of initial data without symmetry and non zero speed.
We emphasize that the articles~\cite{CGNT,Ma,MMjmpa,MM2,RR} for the nonlinear Schr\"odinger equations, the generalized KdV equations and the wave maps were also inspiring for this approach.

\section{Modulation around the kink family}\label{S:2}
In this section, we study the neighborhood of kinks.
This includes the expansion of the conservation laws around the kink, the study of the linearized operator and a  decomposition result around the kink family by modulation.

\subsection{Preliminary observations about the potential and the kink}\label{s:2.1}
First, we provide a proof of Lemma \ref{pr:H}, containing standard information on the kink
(see also~\cite{MR678151,JKL}).
Set
\begin{equation}\label{def:omega}
\omega_- = \sqrt{W''(\zeta_-)},\quad \omega_+=\sqrt{W''(\zeta_-)},
\quad \omega = \min(\omega_-;\omega_+)>0.
\end{equation}

\begin{proof}[Proof of Lemma~\ref{pr:H}]
We integrate explicitly the equation $H''=W'(H)$ after multiplication by $H'$
as $(H')^2=2W(H)$. We obtain $G(H(x))=x$, where the function $G:(\zm,\zp)\to \R$ is defined by
\begin{equation}\label{def:G}
G(h)=\int_{\zeta_0}^h \frac{\ud s}{\sqrt{2 W(s)}},
\quad \zeta_0=\frac{\zm+\zp}{2}.
\end{equation}
The invariance by translation of the equation is handled by the arbitrary choice $H(0)=\zeta_0$.
Note that by the assumptions~\eqref{on:W} on $W$ and the Taylor expansion at the points $\zpm$, we have
for $s\sim \zpm$,
\begin{equation}\label{TaylorW}
W(s) = \frac{W''(\zpm)}{2!} (s-\zpm)^2 + \frac{W'''(\zpm)}{3!} (s-\zpm)^3 + o\left((s-\zpm)^3\right).
\end{equation}
In particular, it holds
\begin{equation*}
\frac 1{\sqrt{2 W(s)}} = \frac 1{\sqrt{W''(\zpm)}}
\frac 1{|s-\zpm|}
\left( 1 - \frac 16 \frac{W'''(\zpm)}{W''(\zpm)} (s-\zpm)+o(s-\zpm)\right).
\end{equation*}
This justifies that $G$ is an increasing one-to-one map from $(\zm,\zp)$ to $\R$.
Thus, $H$ is uniquely defined on $\R$ by $H(x)= G^{-1}(x)$.
Moreover, by integration of the above Taylor expansion, one has
\begin{equation*}
G(h) = \mp\frac 1{\sqrt{W''(\zpm)}}\log|h-\zpm|+C_\pm+C_{\pm}' (h-\zpm) + o(h-\zpm),
\end{equation*}
for some constants $C_\pm$ and $C_\pm'$. Using $G(H(x))=x$, this gives
\[
x=\mp\frac 1{\sqrt{W''(\zpm)}}\log|H(x)-\zpm|+C_\pm+C_{\pm}' (H(x)-\zpm) + o(H(x)-\zpm).
\]
It follows that for two positive constants $\lambda_\pm$,
\begin{equation}\label{on:H}
H(x) = \zeta_\pm \mp \lambda_\pm e^{\mp \omega_\pm x} + O\left( e^{\pm 2 \omega_\pm x}\right)
\quad \mbox{as $x\to \pm\infty$.}
\end{equation}
Using $H'(x) = \sqrt{2W(H(x))}$,~\eqref{TaylorW} and~\eqref{on:H} we also obtain
\begin{equation}\label{on:dH}
H'(x)=\omega_\pm \lambda_\pm e^{\mp \omega_\pm x} + O\left( e^{\mp 2 \omega_\pm x}\right)
\quad \mbox{as $x\to \pm\infty$.}
\end{equation}
By the equation $H''=W'(H)$ and the $\cC^3$ regularity of $W$, it is clear that $H$ is of class $\cC^4$
and $H''' = H' W''(H)$, $H^{(4)}=H'' W''(H) + H'^2 W'''(H)$.
Moreover, by Taylor expansions of $W'$, $W''$ and $W'''$ near $\zpm$, it holds
\begin{equation}\label{on:ddH}\begin{aligned}
H''(x)&=\mp \omega_\pm^2 \lambda_\pm e^{\mp \omega_\pm x} + O\left( e^{\mp 2 \omega_\pm x}\right),\\
H'''(x)&=\omega_\pm^3 \lambda_\pm e^{\mp \omega_\pm x} + O\left( e^{\mp 2 \omega_\pm x}\right),\\
H^{(4)}(x)&=\mp \omega_\pm^4 \lambda_\pm e^{\mp \omega_\pm x} + O\left( e^{\mp 2 \omega_\pm x}\right).
\end{aligned}\end{equation}
In particular, $H''/H'$ is bounded on $\R$. Another observation is 
\[
|W''(H(x))-\omega_{\pm}^2|\leq C e^{\mp \omega_\pm x}.
\]
This completes the proof of the lemma.
\end{proof}

Second, we observe that the sufficient condition for asymptotic stability~\eqref{on:V} on the potential $V$  is a reformulation of the condition that will actually be used in this paper.
The formulation~\eqref{on:V} is privileged in our presentation since it can be checked directly on the potential $W$, without using the expression of the kink $H$.

\begin{lemma}\label{rk:P}
Let $P:\R\to\R$ be the function of class $\cC^1$ defined by
\[
P=\frac{(W'(H))^2}{W(H)}-W''(H)
=2 \left(\frac {H''}{H'}\right)^2 - \frac{H'''}{H'}
=H' \left( \frac 1{H'} \right)''.
\]
The condition~\eqref{on:V} on $V$ is equivalent to $P'\not \equiv 0$ on $\R$ and $P$ satisfies one of the following \begin{description}
\item[Repulsivity at a point] there exists $x_0\in \R$ such that $(x-x_0)P'\leq 0$ on $\R;$
\item[Repulsivity at $-\infty$] $P'\leq 0$ on $\R;$
\item[Repulsivity at $+\infty$] $P'\geq 0$ on $\R$.
\end{description}
Moreover,
\begin{equation}\label{eq:L2.1}
\lim_{\pm\infty} P = V(\zeta_\pm) = W''(\zeta_\pm)=\omega_\pm^2.
\end{equation}
\end{lemma}
\begin{proof}
The result follows from $P(x)=V(H(x))$ so that $P'(x) = H'(x) V'(H(x))$ and the fact that $H'>0$ on $\R$.
The cases $\zeta_0=\zm$ and $\zeta_0=\zp$ in~\eqref{on:V} correspond to the repulsivity of $P$ at $-\infty$ and $+\infty$, respectively. The case $\zeta_0\in (\zm,\zp)$ corresponds to the repulsivity of $P$ at the point $x_0=H^{-1}(\zeta_0)\in \R$.

Last, it is direct from its definition \eqref{def:V} and Taylor expansion that $V$ extends by continuity to $[\zeta_-,\zeta_+]$ and that~\eqref{eq:L2.1} holds.
\end{proof}
\begin{remark}\label{rk:RS}
\begin{enumerate}
\item 
Recall that by the standard virial argument, this assumption on $P$ implies the absence of eigenvalues for the Schr\"{o}dinger operator $L_0=-\partial_x^2 + P$. See~\cite[Theorem~XIII.60]{RS}.
See also the discussion in \S\ref{s.4.6}.
\item Since we consider $H^1\times L^2$ perturbations of the kink, the potential $W$ only needs to be defined in a neighborhood of the interval $[\zp,\zm]$.
We also observe that the assumption~\eqref{on:V} of Theorem~\ref{th:2} only concerns the values of $W$ on the range $(\zeta_-,\zeta_+)$ of the kink. For a potential having several kinks joining different zeros of the potential, the criterion \eqref{on:V} for asymptotic stability may hold or not depending on each kink. See examples in \S\ref{S:5.3}. 
\item By convention, we have chosen to consider the increasing kink  $H$ corresponding to the heteroclinic orbit 
of $H''=W'(H)$ joining $\zm$ to $\zp$.
One can use the change of variable $x\mapsto -x$ to treat the decreasing kink $H(-x)$.
\item To prove the stability result in Theorem~\ref{th:1}, the assumption that $W$ is of class~$\cC^2$ is sufficient.
For the asymptotic stability result Theorem~\ref{th:2}, our method requires $\cC^3$ regularity.
\item The non degeneracy assumption $W''(\zpm)>0$ is needed for the exponential convergence of $H$ to its limits at $\pm \infty$.
However, some models, such as special cases of the generalized $\phi^8$ model (see~\cite[p. 268]{Kevrekidisbook}) and the double sine-Gordon model (see \S\ref{S:5.4}), involve degenerate potentials $W$, \emph{i.e.} satisfying $W''(\zpm)=0$. Such case would require a specific study, taking into account the lower decay rates of the kink at $\pm\infty$ and different spectral properties for the linearized operator.
\end{enumerate}
\end{remark}
For future reference, we also state and prove some decay properties for auxiliary functions that will be considered throughout the paper. Define the function $Q$ (to be used in \S\ref{S:4.9}) by
\begin{equation}\label{def:Q}
Q=\left(\log H'\right)''.
\end{equation}

\begin{lemma}\label{le:PQ}
The functions $P$ and $Q$ are of class $\cC^1$ on $\R$ and satisfy
\begin{align}
&|P(x)-\omega_{\pm}^2|\leq C e^{\mp \omega x},\quad
|P'(x)|\leq C e^{- \omega |x|},\label{on:P}\\
&|Q(x)|+|Q'(x)|\leq C e^{- \omega |x|}.\label{on:Q}
\end{align}
Moreover, under hypothesis~\eqref{on:V}, $P\geq \omega^2$ on $\R$, where $\omega$ is defined in~\eqref{def:omega}.
\end{lemma}

\begin{proof}
The decay properties~\eqref{on:Q} and~\eqref{on:P} follow directly from the definitions of $P$ and $Q$
and the asymptotics~\eqref{on:dH} and~\eqref{on:ddH} on $H$.
The last property is a consequence of Lemma~\ref{rk:P}.
\end{proof}

\subsection{Notation and expansion of the conservation laws}
We will use the following notation for
functions $u,$ $v\in L^2$, $\uu=(u_1,u_2)$, $\vv=(v_1,v_2)\in L^2$,
\[
\inprod{u}{v}=\int u(x) v(x) \ud x,\quad
\inprod{\uu}{\vv}=\int \left[ u_1(x) v_1(x) + u_2(x) v_2(x)\right] \ud x.
\]
For $c\in (-1,1)$ and $y\in \R$, we define the functions $H_{c,y}$ and $\HH_{c,y}$
\begin{equation*}
H_{c,y}(x)= H_c(x-y),\quad  \HH_{c,y}(x)= \HH_c(x-y).
\end{equation*}
Note that with this notation
$H'_{c,y}(x) = \gamma H'\left(\gamma(x-y)\right)$.

\begin{lemma}\label{le:ener}
For any $c\in (-1,1)$ and $y\in \R$, the following hold.
\begin{enumerate}
\item Conservation laws for the kink.
\begin{equation*}
\cE[\HH_{c,y}] = \gamma \|H'\|_{L^2}^2, \quad
\cP[\HH_{c,y}]= -c\gamma \|H'\|_{L^2}^2.
\end{equation*}
\item Expansion of the conservation laws around the kink.
For any $\uu=(u_1,u_2)$ such that $\|u_1\|_{L^\infty}\leq 1$,
\begin{align}
\cE[\HH_{c,y}+\uu]&=\cE[\HH_{c,y}] - c \int H_{c,y}' (u_2-c\partial_x u_1)
\nonumber \\ & \quad +\frac 12 \left( \| u_2\|_{L^2} ^2 + \inprod{\cL_{c,y} u_1}{u_1} +\cR\right),\label{Expansion_E}\\
\cP[\HH_{c,y}+\uu]&=\cP[\HH_{c,y}]
+ \int H_{c,y}' (u_2-c\partial_x u_1) + \cP[\uu],\label{Expansion_P}
\end{align}
where
\begin{equation}\label{def:L}
\cL_{c,y} = -\partial_x^2 + W''(H_{c,y}),
\end{equation} 
and
\begin{equation}\label{2.3bis}
|\cR|\leq C \|u_1\|_{L^\infty}\|u_1\|_{L^2}^2,
\end{equation}
for some constant $C>0$.
\end{enumerate}
\end{lemma}
\begin{remark}\label{rk:1}
These computations motivate the use of the orthogonality relation
\begin{equation}\label{eq:motivation}
\int H_{c,y}' (u_2-c\partial_x u_1)=0 \iff
\inprod{\uu}{ \begin{pmatrix} c H_{c,y}''\\ H_{c,y}' \end{pmatrix} }=0.
\end{equation}
\end{remark}

\begin{proof}
The proof of (1) follows from direct computations.

Proof of (2). First, we expand $\cE[\HH_{c,y}+\uu]$ for $\uu=(u_1,u_2)$.
We have
\begin{multline*}
\cE[\HH_{c,y}+\uu]= \cE[\HH_{c,y}] +
\int \left[ -c H_{c,y}' u_2 + H_{c,y}' \partial_x u_1 + W'(H_{c,y}) u_1 \right] \\
 +\frac 12 \int \left\{ u_2^2 + (\partial_x u_1)^2
+ 2\left[ W(H_{c,y}+u_1) -W(H_{c,y})- W'(H_{c,y}) u_1\right]\right\}.
\end{multline*}
Using integration by parts and $W'(H_{c,y})=\gamma^{-2} H_{c,y}''$, we obtain
\[
\int \left[ -c H_{c,y}' u_2 + H_{c,y}' \partial_x u_1 + W'(H_{c,y}) u_1 \right]
=- c \int H_{c,y}' (u_2-c\partial_x u_1).
\]
For $u_1$ such that $\|u_1\|_{L^\infty}<1$, we have by the Taylor formula (recall that we assume the potential $W$ of class $\cC^3$)
\[
\left| W(H_{c,y}+u_1)-W(H_{c,y})-W'(H_{c,y}) u_1 - W''(H_{c,y}) \frac{u_1^2}2\right| \leq C |u_1|^3
\]
where $C=\frac 16 \sup_{[-1+\zm,1+\zp]} |W'''|$.
This proves~\eqref{Expansion_E}.

Second, we observe by a direct expansion
\begin{equation*}
\cP[\HH_{c,y} + \uu] = \cP[\HH_{c,y}]
 + \int \left(H_{c,y}' u_2 -cH_{c,y}' \partial_x u_1 \right) + \cP[\uu],
\end{equation*}
which is~\eqref{Expansion_P}.
\end{proof}

\subsection{Generalized null spaces}\label{S2.2}
We introduce some more notation related to the kink family.
Let
\[
\TT_{c,y} = -\partial_y \HH_{c,y}=\HH_{c,y}',\quad \DD_{c,y} = \partial_c \HH_{c,y},
\]
respectively related to translations and the Lorentz transform, given explicitly by
\begin{equation*}
\begin{aligned}
\TT_{c,y}(x)&:=
\begin{pmatrix} H_{c,y}'(x) \\ -c H_{c,y}''(x)\end{pmatrix},\\
\DD_{c,y}(x) &:= 
\begin{pmatrix} c \gamma^2 (x-y) H_{c,y}'(x) \\ 
-\gamma^2 H_{c,y}'(x) - c^2 \gamma^2 (x-y) H_{c,y}''(x) \end{pmatrix}.\end{aligned}
\end{equation*}
We also define
\begin{equation*}
\GG_{c,y} = \JJ \TT_{c,y},\quad \FF_{c,y}=\JJ \DD_{c,y}\quad\mbox{where}\quad \JJ=\begin{pmatrix} 0 & 1\\ -1 & 0\end{pmatrix}.
\end{equation*}
Let
\begin{equation}\label{def:LL}
\LL_{c,y} = 
\begin{pmatrix} \cL_{c,y} & - c\partial_x \\ c \partial_x & 1 \end{pmatrix},
\end{equation}
so that
\begin{equation}\label{def:LLbis}
\inprod{\LL_{c,y} \uu}{\uu}=\| u_2\|_{L^2} ^2 + \inprod{\cL_{c,y} u_1}{u_1} + 2 c\cP[\uu].
\end{equation}
The following relations also motivate the introduction of $\TT_{c,y}$ and $\DD_{c,y}$.
\begin{lemma}
It holds
\begin{gather}
\LL_{c,y} \TT_{c,y}= 0,\quad \LL_{c,y} \DD_{c,y} = \JJ \TT_{c,y} ,\label{eigen:LL}\\
\inprod{\DD_{c,y}}{\GG_{c,y}}= -\inprod{\TT_{c,y}}{\FF_{c,y}}= \gamma^3 \norm{H'}_{L^2}^2.\label{eq:nondeg}
\end{gather}
\end{lemma}
\begin{proof}
First, the statement $\LL_{c,y} \TT_{c,y}=0$ follows easily by differentiating in $x$ the relation $(1-c^2) H_{c,y}'' + W'(H_{c,y})=0$.
To compute $\LL_{c,y} \DD_{c,y}$, we use $\LL_{c,y} \TT_{c,y}= 0$ and a commutator relation
\begin{align*}
\LL_{c,y} \DD_{c,y} & = \gamma^2 \LL_{c,y} \left\{ \begin{pmatrix} 0\\ - H'_{c,y}\end{pmatrix}
+ c(x-y) \TT_{c,y}\right\}\\
&= \gamma^2 \begin{pmatrix} cH_{c,y} ''\\ - H'_{c,y}\end{pmatrix} 
+ \gamma^2 c\left[ \LL_{c,y}, (x-y) \right] \TT_{c,y}\\
& =\gamma^2 \begin{pmatrix} cH_{c,y} ''\\ - H'_{c,y}\end{pmatrix}
+ \gamma^2 c\begin{pmatrix} -2 \partial_x & -c \\ c & 0 \end{pmatrix} \TT_{c,y}=\JJ \TT_{c,y}.
\end{align*}
The proof of~\eqref{eq:nondeg} is elementary.
\end{proof}
\begin{remark}
The functions $\TT_{c,y}$ and $\DD_{c,y}$ span the generalized null space of the operator $\JJ\LL_{c,y}$, 
whereas the functions $\GG_{c,y}$ and $\FF_{c,y}$ span the generalized null space of $\LL_{c,y}\JJ$.
The function $\GG_{c,y}$ is also related to the expansion of the invariant quantities in Lemma~\ref{le:ener}.
See~\eqref{eq:motivation} in Remark~\ref{rk:1}.
We refer to~\cite{KK1,KK2} for the introduction of these functions.
\end{remark}

For future reference, we also compute
\[
\AAA_{c,y} = - \partial_y \DD_{c,y} = \partial_c \TT_{c,y}
=\begin{pmatrix}
c\gamma^2 H'_{c,y}+c\gamma^2(x-y) H''_{c,y}\\
-(1+c^2) \gamma^2 H''_{c,y} - c^2\gamma^2(x-y)H'''_{c,y}\end{pmatrix}
\]
and
\[
\BB_{c,y} = \partial_c \DD_{c,y}= \begin{pmatrix}
\gamma^4(1+2c^2)(x-y) H'_{c,y}+c^2\gamma^4(x-y)^2 H''_{c,y}\\
-3c\gamma^4 H'_{c,y} - c\gamma^4(3+2c^2) (x-y)H''_{c,y}
- c^3\gamma^4 (x-y)^2H'''_{c,y}\end{pmatrix}.
\]

\subsection{Lorentz invariant conserved quantity}
In addition to the energy $\cE$ and momentum $\cP$, we introduce the following Lorentzian invariant quantity $\cM$
\begin{equation*}
\cM [\pp] = \left( \cE[\pp] \right)^2 - \left( \cP[\pp] \right)^2 .
\end{equation*}
We continue the computations of Lemma~\ref{le:ener}.
\begin{lemma}
For any $c\in (-1,1)$ and $y\in \R$, the following hold.
\begin{enumerate}
\item Conservation law for the kink.
\begin{equation*}
\cM[\HH_{c,y}] = \|H'\|_{L^2}^4.
\end{equation*}
\item Expansion of the conservation law around the kink.
For any $\uu=(u_1,u_2)$ such that
$\|u_1\|_{L^\infty}\leq 1$ and $\inprod{\uu}{\GG_{c,y}}=0$,
it holds
\begin{multline}\label{Expansion_M}
\cM[\HH_{c,y} + \uu]= \|H'\|_{L^2}^4\\
+ \left( \inprod{\LL_{c,y} \uu}{\uu} + \cR\right)
\left(\gamma\|H'\|_{L^2}^2+ \frac 14 \inprod{\LL_{-c,y} \uu}{\uu}+\frac 14 \cR\right).
\end{multline}
\end{enumerate}
\end{lemma}
\begin{remark}\label{rk:2}
The quantity $\cM$ enjoys two remarkable properties: the value of $\cM$ for a kink $\HH_{c,y}$ does not depend on its speed $c$ and any kink is a critical point of $\cM$.
\end{remark}
\begin{proof} The first point is a direct consequence of (1) of Lemma~\ref{le:ener}.
Using (1) of Lemma~\ref{le:ener},~\eqref{Expansion_E},~\eqref{Expansion_P} and the orthogonality $\inprod{\uu}{\GG_{c,y}}=0$, we recall that
\begin{align*}
\cE[\HH_{c,y}+\uu] 
&= \gamma \|H'\|_{L^2}^2+\frac 12\left( \| u_2\|_{L^2} ^2 + \inprod{\cL_{c,y} u_1}{u_1} + \cR\right),\\
\cP[\HH_{c,y}+\uu]
&= -c\gamma \|H'\|_{L^2}^2+\cP[\uu].
\end{align*}
As a consequence, we compute
\begin{align*}
\cM[\HH_{c,y} + \uu] &= \|H'\|_{L^2}^4
+\gamma\|H'\|_{L^2}^2\left(\| u_2\|_{L^2} ^2 + \inprod{\cL_{c,y} u_1}{u_1} +2c \cP[\uu] + \cR\right)\\
&\quad +\frac 14 \left(\| u_2\|_{L^2} ^2 + \inprod{\cL_{c,y} u_1}{u_1}+\cR\right)^2 - (\cP[u])^2.
\end{align*}
Using the expression of $\inprod{\LL_{c,y} \uu}{\uu}$ in~\eqref{def:LLbis}, we write
\begin{align*}
\cM[\HH_{c,y} + \uu] &= \|H'\|_{L^2}^4
+\gamma\|H'\|_{L^2}^2\left( \inprod{\LL_{c,y} \uu}{\uu} + \cR\right)\\
&\quad +\frac 14 \left(\inprod{\LL_{c,y} \uu}{\uu}+ \cR\right)\left(\inprod{\LL_{-c,y} \uu}{\uu}+ \cR \right),
\end{align*}
which is~\eqref{Expansion_M}. To justify Remark~\ref{rk:2}, observe that even without the assumption $\inprod{\uu}{\GG_{c,y}}=0$, the terms $\inprod{\uu}{\GG_{c,y}}$ in~\eqref{Expansion_E} and~\eqref{Expansion_P} vanish at order $1$ in $\uu$ in the expression of $\cM[\HH_{c,y}+\uu]$.
\end{proof}

\subsection{Linearized operator around the static kink}
\begin{lemma}[See \emph{e.g.}~\cite{Lohe}]\label{on:L}
The linear operator $\cL$ on $L^2$ with domain $H^2$ defined by
\begin{equation*}
\cL = -\partial_x^2 + W''(H)
\end{equation*}
 satisfies the following properties.
\begin{enumerate}
\item The absolutely continuous spectrum of $L$ is $[\omega, +\infty)$ where $\omega>0$ is defined in Lemma~\ref{pr:H}.
\item The operator $L$ is non negative.
\item Denote $Y=H'>0$.
Then, $L Y=0$ and there exists $\mu_0\in (0,1)$ such that for any $v\in H^1$,
\begin{equation}\label{eq:coerL0}
\inprod{v}{Y}=0 \implies \inprod{L v}{v} \geq \mu_0 \|v\|_{H^1}^2.
\end{equation}
\item There exists $\mu_1\in (0,1)$ such that for any $Z\in L^2$ with 
$\inprod{Z}{Y}=\|Y\|_{L^2}^2$ and for any $v_1\in H^1$,
\begin{equation}\label{eq:coerLZ}
 \inprod{\cL v_1}{v_1} \geq \frac 1{\|Z\|_{L^2}^2} \left(\mu_1 \norm{v_1}_{H^1}^2 - \frac 1{\mu_1} \inprod{v_1}{Z}^2\right).
\end{equation}
\end{enumerate}
\end{lemma}
\begin{proof}
The first property follows from $\lim_{\pm \infty} W''(H)= W''(\zeta_\pm)$, the definition of~$\omega$ and standard arguments (see~\cite{MR678151}).
Differentiating the equation $H''=W'(H)$, we see with the notation $Y=H'$ that
$Y''=W''(H) Y$. Moreover, from Lemma~\ref{pr:H}, $Y>0$ on $\R$, which means that $Y$
is the first eigenfunction of $L$ and thus $L$ is non negative.
The coercivity property~\eqref{eq:coerL0} then follows from standard arguments
(see \emph{e.g.}~\cite{MR678151,We85}).

Now, we justify~\eqref{eq:coerLZ}. We decompose $v_1 = v+ a Y$, where $\inprod{v}{Y}=0$, so that 
$\inprod{\cL v_1}{v_1} =\inprod{\cL v}{v}\geq \mu_0 \norm{v}_{H^1}^2$.
Next, $a \inprod{Y}{Z} = \inprod{v_1}{Z} - \inprod{v}{Z}$
yields the estimate $|a|\leq C \left| \inprod{v_1}{Z}\right|+ C\norm{Z}_{L^2} \norm{v}_{L^2} $ and so
$\norm{v_1}_{H^1}\leq C (1+\norm{Z}_{L^2})\norm{v}_{H^1}+C\left|\inprod{v_1}{Z}\right|$, which completes the proof.
\end{proof}

\subsection{Lorentz invariant norm}\label{S.LIN}
Recall that for any $c\in (-1,1)$, we have defined the norm $\|\cdot\|_c$ on $H^1\times L^2$ as follows
\begin{equation}\label{def:norm}
\norm{\uu}_c^2= \gamma^{-1} \norm{\partial_x u_1}_{L^2}^2
+\gamma \norm{u_1}_{L^2}^2+ \gamma \norm{u_2+c\partial_x u_1}_{L^2}^2.
\end{equation}
For fixed $c$, this norm is equivalent to the standard norm. We have
\begin{equation}\label{eq:coerLLbis}
\norm{\uu}_{H^1\times L^2}\leq 2 \gamma^{\frac 12} \norm{\uu}_c
\quad\mbox{and}\quad
\norm{u_1}_{L^\infty} \leq \norm{\uu}_c.
\end{equation}
Indeed, first using $\norm{u_2+c\partial_x u_1}_{L^2}^2 \geq \frac 12\|u_2\|_{L^2}^2 -\|\partial_x u_1\|_{L^2}^2$, we have
\begin{equation*}
\norm{u_1}_{H^1}^2+\gamma^2 \norm{u_2+c\partial_x u_1}_{L^2}^2
\geq \norm{u_1}_{H^1}^2+\frac 12 \norm{u_2+c\partial_x u_1}_{L^2}^2\geq \frac 14 \norm{\uu}_{H^1\times L^2}^2.
\end{equation*}
Second, we observe by standard arguments $\|u_1\|_{L^\infty}^2\leq 2\|\partial_x u_1\|_{L^2}\|u_1\|_{L^2}\leq \norm{\uu}_c^2$.

We also observe that for $\gb_c (x) = (g(\gamma x), h(\gamma x)-c\gamma g'(\gamma x))$, it holds
\begin{equation*}
\|\gb_c\|_c^2 = \|g\|_{H^1}^2+ \|h\|_{L^2}^2.
\end{equation*}
For future reference, we also prove the following result.
\begin{lemma}
There exist constants $C,\delta_0>0$ such that the following holds.
Let $c_0\in (-1,1)$ and $y_0\in \R$. Set $\gamma_0=(1-c_0^2)^{-\frac 12}$. For any $c\in (-1,1)$ and $y\in \R$ such that
\begin{equation}\label{eq:closecy}
\gamma_0^2 |c-c_0| + \gamma_0 |y-y_0|=:\delta \leq \delta_0 ,
\end{equation}
it holds
\begin{enumerate}
\item Close kinks estimate.
\begin{equation}\label{eq:Hclose}
\|\HH_{c_0,y_0} - \HH_{c,y} \|_{c_0} 
\leq C \delta.
\end{equation}
\item Close norms estimate. For all $u\in H^1\times L^2$,
\begin{equation}\label{eq:nul}
\|u\|_c \leq C \|u\|_{c_0}.
\end{equation}
\end{enumerate}
\end{lemma}
\begin{proof}
First, by the definition of $\HH_{c,y}$ \and change of variable, we see that
\[
\|\HH_{c_0,y_0} - \HH_{c_0,y} \|_{c_0}^2
=\|H-H(\cdot -\gamma_0(y-y_0))\|_{H^1}
\leq C \gamma_0 |y-y_0|.
\]
Second, we have
\begin{align*}
\|\HH_{c_0,y} - \HH_{c,y} \|_{c_0}^2 & =
\left\|H'-H'\left( \frac{\gamma}{\gamma_0} \cdot\right)\right\|_{L^2}^2
+\left\|H-H\left( \frac{\gamma}{\gamma_0} \cdot\right)\right\|_{L^2}^2\\
&\quad+\gamma\gamma_0(c_0-c)^2\|H'\|_{L^2}^2.
\end{align*}
Thus,
\[
\|\HH_{c_0,y} - \HH_{c,y} \|_{c_0}\leq C \left|1-\frac{\gamma}{\gamma_0}\right|+ C \sqrt{\gamma\gamma_0}|c_0-c|.
\]
The estimate 
\begin{equation}\label{eq:gamma}
\left|1-\frac{\gamma}{\gamma_0}\right| \leq C \gamma_0^2 |c-c_0|\leq C \delta
\end{equation}
completes the proof of~\eqref{eq:Hclose}.
To prove~\eqref{eq:nul}, we observe from~\eqref{eq:gamma} and~\eqref{eq:closecy} that $\frac 12 \gamma_0<\gamma<2\gamma_0$ and $|c_0-c|\leq \gamma_0^{-2}\delta$.
In particular,
\[
\left| \|u_2+c\partial_x u_1\|_{L^2}- \|u_2+c_0\partial_x u_1\|_{L^2}\right|
\leq |c_0-c| \left\|\partial_x u_1\right\|_{L^2}\leq \gamma_0^{-2}\delta_0\|\partial_x u_1\|_{L^2}.
\]
Estimate~\eqref{eq:nul} follows.
\end{proof}

\subsection{Spectral properties}
For the vector-valued operator $\LL_{c,y}$ defined in~\eqref{def:L}, \eqref{def:LL}
and related to the expansion of $\cM$ in~\eqref{Expansion_M}, we prove a coercivity result under the orthogonality condition with respect to $\FF_{c,y}$.
\begin{lemma}
There exist constants $C>0$ and $\mu\in (0,1)$ such that for any $c\in (-1,1)$, $y\in \R$ and for any $\uu \in H^1 \times L^2$,
the following hold.
\begin{enumerate}
\item Bound. 
\begin{equation}\label{eq:bdLL}
 | \gamma \inprod{\LL_{c,y} \uu}{\uu}| \leq C \norm{\uu}_c^2.
\end{equation}
\item Coercivity. If $\inprod{\uu}{\FF_{c,y}}=0$ then
\begin{equation}\label{eq:coerLL}
 \gamma \inprod{\LL_{c,y} \uu}{\uu} \geq \mu \norm{\uu}_c^2.
\end{equation}
\end{enumerate}
\end{lemma}
\begin{proof} 
Let $c\in (-1,1)$, $y\in \R$ and $\uu=(u_1,u_2) \in H^1 \times L^2$.
We change variable, letting $\vv = (v_1,v_2)\in H^1 \times L^2$ be such that
\begin{gather*}
 u_1 (x) = v_1 (\gamma (x-y)),\quad u_2(x) = v_2(\gamma (x-y)),\\
\gamma \|\partial_x u_1\|_{L^2}^2 = \gamma^2 \|\partial_x v_1\|_{L^2}^2,\quad \gamma \|u_1\|_{L^2}^2 = \|v_1\|_{L^2}^2,\quad \gamma \|u_2\|_{L^2}^2 = \|v_2\|_{L^2}^2.
\end{gather*}
In particular, by direct computations,
\begin{equation}\label{eq:dir1}
\|\uu\|_c^2 = \|v_2+c\gamma \partial_x v_1\|_{L^2}^2 + \|v_1\|_{H^1}^2,
\end{equation}
and
\begin{align}
\gamma \inprod{\LL_{c,y} \uu}{\uu}
&= \int v_2^2+ \gamma^2\int (\partial_x v_1)^2 + \int W''(H) v_1^2 + 2 \gamma c\int (\partial_x v_1) v_2\nonumber\\
&= \|v_2+c\gamma \partial_x v_1\|_{L^2}^2 + \inprod{\cL v_1}{v_1} .
\label{eq:dir2}\end{align}
From~\eqref{def:L}, we have
$|\inprod{\cL v_1}{v_1}|\leq C \|v_1\|_{H^1}^2$ and~\eqref{eq:bdLL} follows from~\eqref{eq:dir1}-\eqref{eq:dir2}.

Now, observe that
\begin{equation}\label{eq:equivv}
\inprod{\uu}{\FF_{c,y}}=0 \iff 
\inprod{v_1}{Z} + c\gamma^{-1} \inprod{v_2+c\gamma \partial_x v_1}{xY}=0,
\end{equation}
where
\begin{equation*}
Z= (1+c^2) Y + 2c^2 xY',\quad \inprod{Z}{Y}=\|Y\|_{L^2}^2.
\end{equation*}
Fix a constant $C\geq 1$ independent of $\gamma$ such that
\begin{equation}\label{bd:Z}
\frac 1C \leq \|Z\|_{L^2}\leq C ,
\end{equation}
and from~\eqref{eq:equivv},
\begin{equation}\label{eq:csq}
\inprod{\uu}{\FF_{c,y}}=0 \implies 
|\inprod{v_1}{Z}|\leq C \gamma^{-1} \norm{v_2+c\gamma \partial_x v_1}_{L^2}.
\end{equation}
Using~\eqref{eq:coerLZ},~\eqref{bd:Z} and~\eqref{eq:csq}, we observe that
\begin{align*}
\inprod{\cL v_1}{v_1} 
& \geq \frac{\mu_1} {C^2} \norm{v_1}_{H^1}^2 
- \frac{C^2} {\mu_1} \inprod{v_1}{Z}^2 \\
& \geq \frac{\mu_1} {C^2} \norm{v_1}_{H^1}^2 
- \frac{C^4}{\mu_1 \gamma^2} \|v_2+c\gamma \partial_x v_1\|_{L^2}^2.
\end{align*}
By $\cL\geq 0$, $\mu_1\in (0,1)$, $\gamma\geq 1$, and
then the above estimate, we obtain
\begin{align*}
\inprod{\cL v_1}{v_1} + \|v_2+c\gamma \partial_x v_1\|_{L^2}^2
&\geq \frac{\mu_1}{2C^4} \inprod{\cL v_1}{v_1} + \|v_2+c\gamma \partial_x v_1\|_{L^2}^2\\
&\geq \frac{\mu_1^2} {2C^8} \norm{v_1}_{H^1}^2 
+\frac 12 \|v_2+c\gamma \partial_x v_1\|_{L^2}^2.
\end{align*}
This proves~\eqref{eq:coerLL}.
\end{proof}

\begin{lemma}\label{cor:1}
There exist constants $C>0$ and $\mu,\delta\in (0,1)$ such that for any $c\in (-1,1)$, $y\in \R$ and for any $\uu \in H^1 \times L^2$
with $\|\uu\|_c\leq \delta$, the following hold.
\begin{enumerate}
\item Bound.
\begin{equation}\label{eq:bdM}
\cM[\HH_{c,y}+\uu] - \|H'\|_{L^2}^4 \leq C \|\uu\|_c^2.
\end{equation}
\item Coercivity. If $\inprod{\uu}{\FF_{c,y}}=\inprod{\uu}{\GG_{c,y}}=0$ then
\begin{equation}\label{eq:coerM}
\cM[\HH_{c,y}+\uu] - \|H'\|_{L^2}^4 \geq \mu \|\uu\|_c^2.
\end{equation}
\end{enumerate}
\end{lemma}
\begin{proof}
Note that~\eqref{2.3bis},~\eqref{def:norm},~\eqref{eq:coerLLbis} imply 
$\gamma | \cR| \leq C \|\uu\|_c^3$. Thus,~\eqref{eq:bdM}-\eqref{eq:coerM} follow from~\eqref{Expansion_M},
\eqref{eq:bdLL},~\eqref{eq:coerLL} for $\delta>0$ small enough independent of $(c,y)$.
\end{proof}
\subsection{Time-independent modulation}
We use a standard argument to modulate any function $\pp$ close to a kink $\HH_{c_0,y_0}$ in terms of the orthogonality conditions
identified in the previous results.
\begin{lemma}\label{le:modulation1}
There exist $C,\delta_1>0$ such that for any $c_0\in (-1,1)$, $y_0\in\R$, $\delta\in (0,\delta_1)$,
and for any $\pp\in (H,0)+ H^1\times L^2$ with $\|\pp-\HH_{c_0,y_0}\|_{c_0}\leq \delta$,
there exist unique $c=c(\pp) \in (-1,1)$ and $y=y(\pp)\in \R$ such that setting
\begin{equation}\label{u1u2}
\uu = \pp - \HH_{c,y}
\end{equation}
the following hold.
\begin{enumerate}
\item Smallness.
\begin{equation*}
\gamma_0^2|c-c_0|+ \gamma_0|y-y_0|+\|\uu\|_{c_0} \leq C\delta.
\end{equation*}
\item Orthogonality.
\begin{equation*}
\inprod{\uu}{\GG_{c,y}}=\inprod{\uu}{\FF_{c,y}}=0.
\end{equation*}
\end{enumerate}
\end{lemma}
\begin{proof}
For some $\delta_1>0$ small to be fixed, let $\delta\in(0,\delta_1)$. Let $c_0\in (-1,1)$, $y_0\in \R$. 
We define
\begin{multline*}
\cB(c_0,y_0,\delta)= 
\{(c,y,\pp) \in(-1,1)\times \R\times ((H,0)+ H^1\times L^2) : \\
\quad \gamma_0^2|c-c_0|< \delta,\ \gamma_0|y-y_0|< \delta,\ \|\pp-\HH_{c_0,y_0}\|_{c_0}< \delta\}.
\end{multline*}
For any $(c,y,\pp)\in \cB(c_0,y_0,\delta)$, we also define $\cF=(\cF_1,\cF_2)$ where
\[
\cF_1(c,y;\pp) = \inprod{\pp-\HH_{c,y}}{\FF_{c,y}},\quad
\cF_2(c,y;\pp) = \inprod{\pp-\HH_{c,y}}{\GG_{c,y}}.
\]
We see that $\cF(c_0,y_0;\HH_{c_0,y_0})=(0,0)$ and we compute
using \S\ref{S2.2} (in particular, identity~\eqref{eq:nondeg}),
\begin{align*}
 \frac{\partial \cF_1}{\partial y} (c,y;\pp) 
&= -\inprod{\partial_y \HH_{c,y}}{\FF_{c,y}}+\inprod{\pp-\HH_{c,y}}{\partial_y\FF_{c,y}}\\
&= -\gamma^3 \|H'\|_{L^2}^2-\inprod{\pp-\HH_{c,y}}{\JJ \AAA_{c,y}},\\
\frac{\partial \cF_1}{\partial c} (c,y;\pp)
&= - \inprod{\partial_c\HH_{c,y}}{\FF_{c,y}}
+\inprod{\pp-\HH_{c,y}}{\partial_c\FF_{c,y}}=\inprod{\pp-\HH_{c,y}}{\JJ\BB_{c,y}}\\
\frac{\partial \cF_2}{\partial y} (c,y;\pp) &= - \inprod{\partial_y\HH_{c,y}}{\GG_{c,y}}+\inprod{\pp-\HH_{c,y}}{\partial_y\GG_{c,y}}
=-\inprod{\pp-\HH_{c,y}}{\JJ\TT_{c,y}'}\\
\frac{\partial \cF_2}{\partial c} (c,y;\pp) &= 
-\inprod{\partial_c \HH_{c,y}}{\GG_{c,y}}+\inprod{\pp-\HH_{c,y}}{\partial_c\GG_{c,y}}\\
&= - \gamma^3 \|H'\|_{L^2}^2+\inprod{\pp-\HH_{c,y}}{\JJ \AAA_{c,y}}.
\end{align*}
We claim that
\begin{equation}\label{cl:mar2}\begin{aligned}
|\cF_1(c,y;\pp)|&\leq C \gamma^2 \|\pp-\HH_{c,y}\|_{c}\leq C \gamma_0^2 \delta,\\
|\cF_2(c,y;\pp)|&\leq C \gamma \|\pp-\HH_{c,y}\|_{c}\leq C \gamma_0 \delta,
\end{aligned}\end{equation}
and
\begin{equation}\label{cl:mar}\begin{aligned}
|\inprod{\pp-\HH_{c,y}}{\JJ \TT_{c,y}'}|&\leq C \gamma^2 \|\pp-\HH_{c,y}\|_{c}\leq C \gamma_0^2 \delta,\\
|\inprod{\pp-\HH_{c,y}}{\JJ \AAA_{c,y}}|&\leq C \gamma^3 \|\pp-\HH_{c,y}\|_{c}\leq C \gamma_0^3 \delta,\\ 
|\inprod{\pp-\HH_{c,y}}{\JJ \BB_{c,y}}| &\leq C \gamma^4 \|\pp-\HH_{c,y}\|_{c}\leq C \gamma_0^4 \delta.
\end{aligned}\end{equation}
Indeed, for any Schwartz functions $V$ and $W$, setting $\VV_{c,y}=(V_{c,y},W_{c,y})$, it holds
\begin{align*}
\inprod{\uu}{\JJ \VV_{c,y}}
&=- \inprod{u_2}{V_{c,y}} + \inprod{u_1}{W_{c,y}} \\
&=- \inprod{u_2+c\partial_x u_1 }{V_{c,y}} + \inprod{u_1}{W_{c,y}-cV_{c,y}'}\\
&=- \gamma^{-1} \inprod{v_2+c\gamma \partial_x v_1}{V}+\gamma^{-1} \inprod{v_1}{W-\gamma cV'}.
\end{align*}
In particular,
\begin{equation}\label{eq:truc}
|\inprod{\uu}{\JJ \VV_{c,y}}|\leq \gamma^{-1} \|\uu\|_c \left(\|V\|_{L^2}+\|W-\gamma cV'\|_{L^2}\right).
\end{equation}
By~\eqref{eq:Hclose} and~\eqref{eq:nul}, we have
\begin{align*}
\|\pp-\HH_{c,y}\|_{c}
&\leq C\|\pp-\HH_{c,y}\|_{c_0}\leq
C\|\pp-\HH_{c_0,y_0}\|_{c_0}+C\|\HH_{c_0,y_0}-\HH_{c,y}\|_{c_0}\\
&\leq C\|\pp-\HH_{c_0,y_0}\|_{c_0} + C \left( \gamma_0^2 |c_0-c| + \gamma_0 |y_0-y|\right)
\leq C \delta.
\end{align*}
Combining this with~\eqref{eq:truc},~\eqref{cl:mar}, $|1-\frac{\gamma}{\gamma_0} | \leq C \delta$ (see~\eqref{eq:gamma}) and the definitions of $\FF_{c,y}$, $\GG_{c,y}$, $\TT_{c,y}$, $\AAA_{c,y}$, $\BB_{c,y}$, we obtain~\eqref{cl:mar2}-\eqref{cl:mar}.

It follows from~\eqref{cl:mar2}-\eqref{cl:mar} that for $\delta_1>0$ small enough, for any $(c,y,\pp)\in \cB(c_0,y_0,\delta_1)$
the Jacobian Matrix of $\cF$ writes
\begin{equation*}
{\rm Jac}_{\cF}(c,y;\pp)
=- \gamma_0^3 \|H'\|_{L^2}^2 \begin{pmatrix} 1+O(\delta_1) & O(\gamma_0 \delta_1) \\ O(\gamma_0^{-1} \delta_1)& 1+O(\delta_1) \end{pmatrix}.
\end{equation*}
Therefore, there exists a small constant $\sigma>0$ such that for any $\pp\in (H,0)+ H^1\times L^2$ with $\|\pp-\HH_{c_0,y_0}\|_{c_0}< \delta\leq \sigma \delta_1$,
the Implicit Function Theorem shows the existence of unique $(c,y)$ such that $\cF(c,y;\pp)=(0,0)$
and $\gamma_0^2 |c-c_0|+\gamma_0 |y-y_0|\leq C\delta$. Defining $\uu = \pp-H_{c,y}$, we find
$\|\uu\|_{c_0} \leq \|\pp-H_{c_0,y_0}\|_{c_0} + \|H_{c_0,y_0}-H_{c,y}\|_{c_0}\leq C \delta$.
\end{proof}

\subsection{Time-dependent modulation}
\begin{lemma}\label{le:modulation2}
There exist $C,\delta_1>0$ such that for any $c_0\in (-1,1)$, $y_0\in\R$, $\delta\in (0,\delta_1)$, if
$\pp=(\phi_1,\phi_2)$ is a solution of~\eqref{syst} on $[T_1,T_2]$ satisfying
\[
\sup_{t\in [T_1,T_2]} \left\{\inf_{y_0\in \R} \|\pp(t)-\HH_{c_0,y_0}\|_{c_0} \right\}\leq \delta,
\]
then, there exist unique functions $c:[T_1,T_2]\to (-1,1)$ and $y:[T_1,T_2]\to \R$ of class $\cC^1$,
such that decomposing
\begin{equation}\label{eq:phiu}
\pp(t) = \HH_{c(t),y(t)}+\uu(t),
\end{equation}
the following properties hold, for all $t\in [T_1,T_2]$,
\begin{enumerate}
\item Smallness.
\begin{equation}\label{eq:small}
\gamma_0^2|c(t)-c_0|+\|\uu(t)\|_{c_0} \leq C \delta.
\end{equation}
\item Orthogonality.
\begin{equation}\label{eq:ortho}
\inprod{\uu(t)}{\GG_{c(t),y(t)}}=\inprod{\uu(t)}{\FF_{c(t),y(t)}}=0.
\end{equation}
\item Equation of $\uu$.
\begin{equation}\label{eq:u}
\partial_t \uu 
= \left(\JJ \LL_{c,y} - \Id c \partial_x \right) \uu
+ \left(\dot y-c\right) \TT_{c,y} - \dot c \DD_{c,y}+ \RR,
\end{equation}
where $|\RR|\leq C u_1^2$.
\item Estimate on time derivatives.
\begin{equation}\label{eq:param}
\gamma_0|(\dot y-c)(t)| +\gamma_0^2 |\dot c(t)|
\leq \int u_1^2(x) e^{-\frac12 {\gamma_0 \omega}|x-y|} \ud x.
\end{equation}
\end{enumerate}
\end{lemma}
\begin{proof}
We begin with formal computations.
Inserting~\eqref{eq:phiu} into the system~\eqref{syst}, we observe that
\[
\partial_t \uu 
= \begin{pmatrix} \phi_2\\ \partial_x^2 \phi_1 - W'(\phi_1) \end{pmatrix}
+ \dot y \TT_{c,y} - \dot c \DD_{c,y}.
\]
Using $H_{c,y}'' = c^2 H_{c,y}''+W'(H_{c,y})$, we have
\begin{align*}
\partial_t \uu 
& = \begin{pmatrix} u_2 \\ \partial_x^2 u_1 - [ W'(H_{c,y}+u_1)-W'(H_{c,y})] \end{pmatrix}
+ \left(\dot y-c\right) \TT_{c,y} - \dot c \DD_{c,y} \\
& = \begin{pmatrix} 0 & 1 \\ - \cL_{c,y} & 0 \end{pmatrix} \uu
+ \left(\dot y-c\right) \TT_{c,y}- \dot c \DD_{c,y}
+ \RR,
\end{align*}
where
\[
\RR = - \begin{pmatrix} 0 \\ W'(H_{c,y}+u_1)-W'(H_{c,y})-W''(H_{c,y}) u_1\end{pmatrix}.
\]
For $u_1$ such that $\|u_1\|_{L^\infty}\leq 1$, we have by Taylor expansion (recall that we assume $W$ of class $\cC^3$)
\begin{equation}\label{eq:bdRR}
 \left| W'(H_{c,y}+u_1)-W'(H_{c,y})-W''(H_{c,y}) u_1 \right| \leq \frac { u_1^2} 2 \sup_{[-1+\zm,1+\zp]} |W'''|.
\end{equation}
From the definition~\eqref{def:LL}, we have the relation 
\begin{equation*}
\begin{pmatrix} 0 & 1 \\ - \cL_{c,y} & 0 \end{pmatrix}
=\JJ \LL_{c,y} - \Id c \partial_x,
\end{equation*}
which formally justifies~\eqref{eq:u}. Differentiating the first orthogonality condition in~\eqref{eq:ortho},
using~\eqref{eq:u}, $\GG_{c,y}=\JJ \TT_{c,y}$, and $\inprod{\TT_{c,y}}{\JJ\TT_{c,y}}=0$, we obtain
\begin{align*}
0&= \frac \ud{\ud t} \inprod{\uu}{\GG_{c,y}}
= \inprod{\partial_t \uu}{\GG_{c,y}}+\inprod{\uu}{\partial_t \GG_{c,y}}\\
& = \inprod{ (\JJ \LL_{c,y} - \Id c \partial_x ) \uu}{\GG_{c,y}} 
 - \dot c \inprod{\DD_{c,y}}{\GG_{c,y}}\\
&\quad +\inprod{\RR}{\GG_{c,y}}-\dot y \inprod{\uu}{\partial_x \GG_{c,y}}
+\dot c \inprod{\uu}{\JJ \AAA_{c,y}}.
\end{align*}
Note that by~\eqref{eigen:LL}
\begin{equation*}
\inprod{\JJ \LL_{c,y} \uu}{\GG_{c,y}}
=\inprod{\LL_{c,y} \uu}{\TT_{c,y}}=\inprod{\uu}{\LL_{c,y} \TT_{c,y}}=0.
\end{equation*}
Thus, using also~\eqref{eq:nondeg}, the above identity rewrites
\begin{equation}\label{eq:modG}
\dot c \left( \gamma^3\|H'\|_{L^2}^2 -\inprod{\uu}{\JJ \AAA_{c,y}}\right)
+(\dot y-c) \inprod{\uu}{\JJ \TT_{c,y}'} =
 \inprod{\RR}{\GG_{c,y}}.
\end{equation}
Differentiating the second orthogonality condition in~\eqref{eq:ortho},
using~\eqref{eq:u}, $\FF_{c,y}=\JJ \DD_{c,y}$, and $\inprod{\DD_{c,y}}{\JJ\DD_{c,y}}=0$, we obtain
\begin{align*}
0&= \frac \ud{\ud t} \inprod{\uu}{\FF_{c,y}}
= \inprod{\partial_t \uu}{\FF_{c,y}}+\inprod{\uu}{\partial_t \FF_{c,y}}\\
& = \inprod{ (\JJ \LL_{c,y} - \Id c \partial_x ) \uu}{\FF_{c,y}} 
+ (\dot y-c) \inprod{\TT_{c,y}}{\FF_{c,y}}\\
&\quad +\inprod{\RR}{\FF_{c,y}}-\dot y \inprod{\uu}{\partial_x \FF_{c,y}}
+\dot c \inprod{\uu}{\JJ \BB_{c,y}}.
\end{align*}
Note that by~\eqref{eigen:LL}
\begin{equation*}
\inprod{\JJ \LL_{c,y} \uu}{\FF_{c,y}}
=\inprod{\LL_{c,y} \uu}{\DD_{c,y}}=\inprod{\uu}{\LL_{c,y} \DD_{c,y}}=\inprod{\uu}{\GG_{c,y}}=0.
\end{equation*}
Thus, using also~\eqref{eq:nondeg}, the above identity rewrites
\begin{equation}\label{eq:modF}
(\dot y-c) \left(\gamma^3\|H'\|_{L^2}^2 + \inprod{\uu}{\JJ \AAA_{c,y}}\right)
-\dot c \inprod{\uu}{\JJ \BB_{c,y}} = \inprod{\RR}{\FF_{c,y}}.
\end{equation}
Setting
\begin{gather*}
D = \begin{pmatrix} 
 \gamma^3\|H'\|_{L^2}^2 +\inprod{\uu}{\JJ \AAA_{c,y}} & - \inprod{\uu}{\JJ \BB_{c,y}} \\
\inprod{\uu}{\JJ \TT_{c,y}'} & \gamma^3\|H'\|_{L^2}^2 -\inprod{\uu}{\JJ \AAA_{c,y}} 
\end{pmatrix},\\
M= \begin{pmatrix}\dot y- c \\ \dot c \end{pmatrix}, 
\quad
S = \begin{pmatrix}\inprod{\RR}{\FF_{c,y}}\\ \inprod{\RR}{\GG_{c,y}} \end{pmatrix},
\end{gather*}
we observe that~\eqref{eq:modG}-\eqref{eq:modF} rewrite as
$D M = S$. 
From the estimates of the proof of Lemma~\ref{le:modulation1}, for $\|\uu\|_{c_0}$ small, the matrix $D$ writes
\[
D = \gamma_0^3\|H'\|_{L^2}^2 \begin{pmatrix} 1+O(\delta_1) & O(\gamma_0^{-1} \delta_1) \\ O(\gamma_0 \delta_1) & 1+O(\delta_1)\end{pmatrix},
\]
and thus
\[
D^{-1} = \gamma_0^{-3} \|H'\|_{L^2}^{-2} \begin{pmatrix} 1+O(\delta_1) & O(\gamma_0 \delta_1)
\\ O(\gamma_0^{-1} \delta_1) & 1+O(\delta_1)\end{pmatrix}.
\]
We rewrite
\begin{equation}\label{eq:Scy}
M = D^{-1} S \iff
\frac{\ud}{\ud t} \begin{pmatrix} y\\ c\end{pmatrix}
=\begin{pmatrix} c \\ 0\end{pmatrix} + D^{-1} S.
\end{equation}
From these computations, we see that the relations 
$\frac{\ud}{\ud t} \inprod{\uu}{\GG_{c,y}} = \frac{\ud}{\ud t}\inprod{\uu}{\FF_{c,y}}=0$ are equivalent to 
the differential system~\eqref{eq:Scy}.
The solution $\pp$ of~\eqref{syst} being fixed on some time interval,~\eqref{eq:Scy} is a first-order non-autonomous differential system with Lipschitz continuous dependency in $(c,y)$ and continuity in $t$.
The proof of the lemma now goes as follows: at $t=T_1$, we perform a time-independent modulation of the function $\pp(T_1)$ according to
Lemma~\ref{le:modulation1}. Then, we apply the Cauchy-Lipschitz theorem to the differential system~\eqref{eq:Scy}
to prove existence of a solution $(c(t),y(t))$ on $[T_1,T_2]$.
This justifies~\eqref{eq:small},~\eqref{eq:ortho} and~\eqref{eq:u}. To complete the proof, in view of the form of the matrix $D^{-1}$ above, we only have to establish the estimates
\begin{equation}\label{on:FG}
|\inprod{\RR}{\FF_{c,y}} | \leq C \gamma_0 \|\uu\|_{w,c_0}^2,
\quad |\inprod{\RR}{\GG_{c,y}}|\leq C \|\uu\|_{w,c_0}^2.
\end{equation}
From the definition of $\FF_{c,y}$, $\GG_{c,y}$, $\RR$ and estimates~\eqref{eq:Hasym},~\eqref{eq:gamma},~\eqref{eq:bdRR}, we have
\begin{align*}
|\inprod{\RR}{\FF_{c,y}}| & \leq C\gamma_0^2 \int u_1^2 e^{-\frac 34 \gamma \omega |x-y|}\leq C\gamma_0^2 \int u_1^2 e^{-\frac 12 \gamma_0 \omega |x-y|},\\
|\inprod{\RR}{\GG_{c,y}} | & \leq C\gamma_0 \int u_1^2 e^{-\gamma \omega |x-y|}\leq C\gamma_0 \int u_1^2 e^{-\frac 12 \gamma_0 \omega |x-y|},
\end{align*}
which imply~\eqref{on:FG}.
\end{proof}

\section{Orbital stability}\label{S:3}
Let $\delta>0$ small enough to be chosen. Let $c_0\in (-1,1)$ and $y_0\in \R$.
We consider an initial data $\pp^{in}\in \EE_{\HH}$ such that
\begin{equation}\label{eq:small:in}
\| \pp^{in} - \HH_{c_0,y_0} \|_{c_0} \leq \delta.
\end{equation}

\subsection{Local well-posedness in a neighborhood of a kink}\label{S.LWP}
Looking for a solution $\pp(t)=(\phi_1(t),\phi_2(t))$ of~\eqref{syst} in $\EE_{\HH}$ for all time with data $\pp(0)=\pp^{in}$ and setting 
\begin{align*}
v_1 (t) &= \phi_1(t)-H(\gamma_0 (x -y_0-c_0 t)),\\
v_2(t) &= \phi_2(t) + \gamma_0 c_0 H'(\gamma_0(x -y_0-c_0 t)),
\end{align*}
we are reduced to solve
\begin{equation}\label{syst:th}
\begin{cases}
\partial_t v_1 = v_2 \\
\partial_t v_2 = \partial_x^2 v_1 - F(t,x,v_1),
\end{cases}
\end{equation}
in $H^1\times L^2$, where 
\[
F(t,x,v_1)=W'(H(\gamma_0 (x -y_0-c_0 t))+v_1)- W'(H(\gamma_0 (x -y_0-c_0 t))).
\]
There exists $C>0$ such that for any $v_1,\tilde v_1$, if $\|v_1\|_{L^\infty}\leq 1$ and
$\|\tilde v_1\|_{L^\infty}\leq 1$ then 
\[
|F(t,x,v_1)-F(t,x,\tilde v_1)|\leq C |v_1-\tilde v_1|.
\]
By standard arguments, for $\delta$ small enough, there exists a local in time solution $(v_1,v_2)$ of~\eqref{syst:th} in $H^1\times L^2$.
In this paper, we will only need the above notion of solution
$\pp =(\phi_1,\phi_2)$ of~\eqref{syst}.

\subsection{Proof of Theorem~\ref{th:1}}
Let $\pp(t)$ be the local in time solution of~\eqref{syst} with initial data $\pp^{in}$ given in~\S\ref{S.LWP}.
For $C^*>1$ to be fixed later, define
\begin{multline*}
T^* = \sup\Big\{ t \geq 0 : \mbox{$\pp(t)$ is well-defined on $[0,t]$ and } \\ \sup_{s\in[0,t]} \Big( \inf_{y_1\in \R} \|\pp(s)-\HH_{c_0,y_1}\|_{c_0}\Big) \leq C^* \delta\Big\}.
\end{multline*}
By \S\ref{S.LWP} and continuity, $T^*>0$ is well-defined. 
Moreover, if $T^*<\infty$ then by a continuity argument and \S\ref{S.LWP}, $\pp$ would be well-defined on $[0,t+\tau]$ for some $\tau>0$ and it would hold \begin{equation}\label{eq:saturation}
\inf_{y_1\in \R} \|\pp(T^*)-\HH_{c_0,y_1}\|_{c_0} = C^* \delta .
\end{equation}
We assume $T^*<\infty$ and we work on the time interval $[0,T^*]$. We use the decomposition of $\pp(t)$ given by Lemma~\ref{le:modulation2}.
In particular,
\begin{equation}\label{eq:init}
\gamma_0^2 |c(0)-c_0|+ \|\uu(0)\|_{c_0}\leq C \delta,\quad 
\gamma_0^2 |c(t)-c_0|\leq C C^* \delta.
\end{equation}
By~\eqref{eq:gamma}, we have $\frac 12 \gamma_0\leq \gamma(t)\leq 2 \gamma_0$, and by~\eqref{eq:nul}, $C^{-1} \|\cdot\|_c\leq \|\cdot\|_{c_0}
\leq C \|\cdot \|_c$.

By the conservation of energy and momentum,
and Lemma~\ref{cor:1}, we have for any $t\in [0,T^*]$,
\begin{multline*}
\mu \|\uu(t)\|_{c_0}^2\leq 
\cM[\HH_{c,y}+\uu(t)] - \|H'\|_{L^2}^4 \\ 
=\cM[\HH_{c(0),y(0)}+\uu(0)] - \|H'\|_{L^2}^4\leq 
C \|\uu(0)\|_{c_0}^2\leq C \delta^2.
\end{multline*}
Thus, for all $t\in [0,T^*]$,
\begin{equation}\label{first}
\|\uu(t)\|_{c_0}^2\leq C \delta^2,
\end{equation}
where $C$ is independent of $C^*$.

From~\eqref{eq:coerLLbis}, we have $|\cP[\uu]|\leq C \|\uu\|_{H^1\times L^2}^2 \leq C \gamma \|\uu\|_{c_0}^2$, and so using~\eqref{Expansion_P} and~\eqref{first},
\[
|\cP[ \HH_{c,y} + \uu] - \cP[\HH_{c,y}] |\leq C \gamma \|\uu\|_{c_0}^2 \leq C \gamma \delta^2. 
\]
Thus, by conservation of the momentum $\cP[\pp]=\cP[ \HH_{c,y} + \uu]$ and $P[H_{c,y}]=- c\gamma \|H'\|_{L^2}^2$ (see (1) of Lemma~\ref{le:ener}), we obtain
\[
|c(t)\gamma(t)-c(0)\gamma(0)|\leq C \gamma_0 \delta^2.
\]
Observe by direct computation that
\begin{equation}\label{dcg}
\frac{\ud}{\ud c} (c\gamma) = \gamma^3 \dot c.
\end{equation}
Thus, the previous estimate shows that $\gamma_0^2 |c(t)-c(0)|\leq C \delta^2$.
Combined with~\eqref{eq:init}, this gives, for all $t\in [0,T^*]$,
\begin{equation}\label{second}
\gamma_0^2 |c(t)-c_0|\leq C \delta,
\end{equation}
where $C$ is independent of $C^*$.

By~\eqref{first},~\eqref{second},~\eqref{eq:Hclose} and the triangle inequality, we obtain, for all $t\in [0,T^*]$,
\[
\|\pp(t)-\HH_{c_0,y(t)}\|_{c_0}
\leq \|\uu(t)\|_{c_0}+\|\HH_{c(t),y(t)}-\HH_{c_0,y(t)}\|_{c_0}\leq C_1 \delta
\]
for some $C_1$ independent of $C^*$. We contradict~\eqref{eq:saturation} by fixing $C^*>C_1$.
Therefore, the solution $\pp$ is global for $t\geq 0$, $T^*=\infty$ and~\eqref{first},~\eqref{second} hold for any $t\geq 0$. Moreover, estimate~\eqref{eq:param} shows that
$\gamma_0^2 \sup_{t\geq 0}|\dot y(t) -c_0 |\leq C \delta$.
Last, by time reversibility of the equation, the same holds true for all $t\leq 0$.

\section{Asymptotic stability}\label{S:4}

We work in the framework of the orbital stability result Theorem~\ref{th:1}.
In particular, we consider a solution $\uu\in \cC(\R,H^1\times L^2)$ of~\eqref{eq:ortho}-\eqref{eq:u} that satisfies the uniform smallness condition
\begin{equation}\label{eq:smallfinal}
\gamma_0^2|c(t)-c_0|+\|\uu(t)\|_{c_0} \leq C \delta
\end{equation} 
for all $t\in\R$, where $\delta$, defined from the initial data $\pp^{in}$ in~\eqref{eq:small:in} is to be taken small enough. Moreover,~\eqref{eq:param} holds on $\R$.
In this section, we do not track the dependency in $c_0$ and $\gamma_0$ (this is why the constant $\delta_0>0$
in Theorem~\ref{th:2} may depend on $c_0$) and constants $C$ may depend on $c_0$ from now on.

\subsection{Change of variables}\label{S:4.1}
We define new unknowns $\zz(t,x)=(z_1(t,x),z_2(t,x))$, by setting
\begin{equation}\label{z1z2}
\begin{cases}
u_1(t,x')= z_1(t,\gamma(t)(x'-y(t))),\\
u_2(t,x')= z_2(t,\gamma(t)(x'-y(t)))-c(t)\gamma(t)\partial_xz_1(t,\gamma(t)(x'-y(t))).
\end{cases}
\end{equation}
This change of variables allows to work around the static kink $\HH=(H,0)$. 
This does not mean that we are reduced to the case $c=0$ since we do not change the time variable.
In particular, we note the following relations for the derivatives of $\uu(t,x')$
in the Lorentz frame
\begin{align*}
(\partial_x u_1 + c u_2)\left(t, \frac{x}{\gamma(t)}+y(t)\right)
&=\frac 1{\gamma(t)}(\partial_x z_1+c\gamma z_2)(t,x),\\
(u_2+c \partial_x u_1)\left(t, \frac{x}{\gamma(t)}+y(t)\right)&=z_2(t,x).
\end{align*}
The directional derivative $\partial_x z_1+c\gamma z_2$ of $\zz$ and the function $z_2$ will appear naturally in our computations
(see for example Lemmas~\ref{le:cZ} and~\ref{le:H} and \eqref{def:kk}).
We introduce the notation
\[
\Lambda = x \partial_x.
\]
We summarize the information on $\zz$. 
First, from~\eqref{eq:u} and explicit computations, the equation satisfied by $\zz$ writes
\begin{equation}\label{eq:z}
\begin{cases}
\dot z_1 = z_2 +\MM_1,\\
\dot z_2 = \Theta(\zz) + \MM_2,
\end{cases}
\end{equation}
where
\begin{equation*}
\Theta(\zz)=\partial_x^2 z_1 - \left[W'(H+z_1)- W'(H)\right] + 2 c \gamma\partial_x z_2,
\end{equation*}
and
\begin{align}
\MM_1&=(\dot y-c)\gamma \partial_x (H+z_1)- \dot c c\gamma^2 \Lambda (H+z_1),\label{def:M1}\\
\MM_2&=(\dot y-c) \gamma \partial_x z_2 - \dot cc\gamma^2 \Lambda z_2 + \dot c \gamma \partial_x (H+z_1).\label{def:M2}
\end{align}
Observe that $\Theta(\zz)$ also writes
\begin{equation}\label{other:Theta}
\Theta(\zz)=-L z_1 +2c\gamma \partial_x z_2 + \RRR,
\end{equation}
where the operator $L$ is defined in Lemma~\ref{on:L} and
\begin{equation*}
\RRR = -\left[ W'(H+z_1)-W'(H) - W''(H) z_1\right].
\end{equation*}
We introduce the notation $\Omegab(\ab)=(\Omega_1(\ab),\Omega_2(\ab))$, where $\ab=(a_1,a_2)$ and
\begin{equation}\label{def:Omega}\begin{aligned}
\Omega_1(\ab)&=(\dot y-c)\gamma \partial_x a_1 - \dot c c\gamma^2 \Lambda a_1,\\
\Omega_2(\ab)&=(\dot y-c)\gamma \partial_x a_2 - \dot c c\gamma^2 \Lambda a_2
+\dot c \gamma \partial_x a_1.
\end{aligned}\end{equation}
Then, $\Mb=(\MM_1,\MM_2)$ rewrites as
$\Mb= \Omegab(\HH)+\Omegab(\zz)$.

Second, from $\|\zz\|_{H^1\times L^2}^2 = \|\uu \|_{c}^2$ and~\eqref{eq:smallfinal},
\begin{equation}\label{bd:z}
 \gamma_0^2 |c(t)-c_0|+\|\zz\|_{H^1\times L^2}\leq C \delta.
\end{equation}
Third, we check by direct computations that the orthogonality conditions~\eqref{eq:ortho} on~$\uu$ rewrite
(recall $Y=H'$)
\begin{align}
0 & = 2 c \gamma \langle z_1 , Y'\rangle + \langle z_2 , Y\rangle, \label{ortho:z1}\\
0 & = \gamma \langle z_1, (1+c^2)Y + 2c^2 xY'\rangle + c \langle z_2, xY\rangle.\label{ortho:z2}
\end{align}
Last, from~\eqref{eq:gamma} and \eqref{eq:param}
\begin{equation}\label{bd:yc}
\gamma_0^2 |\dot y - c| + \gamma_0^3 |\dot c| \leq C \int \sech\left(\frac {\omega x}4 \right) z_1^2 .
\end{equation}

\subsection{Heuristic}\label{s.4.6}
Consider the linear problem
\begin{equation}\label{eq:zlin}
\begin{cases}
\dot z_1 = z_2 \\
\dot z_2 = -L z_1 + 2 c \gamma\partial_x z_2
\end{cases}
\end{equation}
where for the sake of simplicity, we have neglected the modulation terms $\MM_1$, $\MM_2$ and the nonlinear term $\RRR$ in the system~\eqref{eq:z} for $\zz$.
We follow the strategy of~\cite{KMM4} (see also references there for previous works), introducing a transformed operator with simplified spectrum.
In order to introduce the transformed problem, we set
\[
U = Y \cdot \partial_x \cdot Y^{-1},\quad U^\star = - Y^{-1} \cdot \partial_x \cdot Y,
\]
where the above notation means $U f = Y \left(f/Y\right)'$,
and
\begin{equation*}
L_0 =- \partial_x^2 + P\quad\mbox{where}\quad P= 2 \left(\frac {Y'}{Y}\right)^2 - \frac{Y''}Y.
\end{equation*}
See Lemmas~\ref{rk:P} and~\ref{le:PQ} for some properties of $P$.
The strategy of the proof of Theorem~\ref{th:2} is based on the following observations, 
inspired by~\cite[Section 3.2]{CGNT}.
\begin{lemma}\label{le:identities}\quad
\begin{enumerate}\item Factorization by first order operators.
\begin{equation}\label{eq:fact}
L = U^\star U,\quad L_0=U U^\star.
\end{equation} 
\item Commutator relation.
\begin{equation}\label{4.30}
 U \partial_x - \partial_x U =\frac 12 ( L- L_0).
\end{equation}
\item Conjugate identities. It holds
$U L = L_0 U$
and more generally, for any $c,\gamma$,
\begin{equation}\label{eq:conj2}
(U+c\gamma \partial_t) (L - 2 c \gamma\partial_x\partial_t)
=(L_0 - 2 c \gamma\partial_x\partial_t)(U+c\gamma \partial_t) .
\end{equation}
\end{enumerate}
\end{lemma}
\begin{proof}
The identities in~\eqref{eq:fact} follow from direct computations
and have motivated the introduction of the operator $L_0$.
Next, we remark by explicit computations that
$U^\star=U-2 \partial_x$, so that
$U^\star U = U^2 - 2 \partial_x U$
and $U U^\star = U^2 - 2 U \partial_x$, which implies~\eqref{4.30}.
Last, we check by expansion and using~\eqref{4.30}
\begin{equation*}
(U+c\gamma \partial_t) (-U^\star U + 2 c \gamma\partial_x\partial_t)
=(-U U^\star + 2 c \gamma\partial_x\partial_t)(U+c\gamma \partial_t) ;
\end{equation*}
then~\eqref{eq:conj2} follows from~\eqref{eq:fact}.
\end{proof}
Setting 
\begin{equation*}
\begin{cases}
g_1 = (U+c\gamma \partial_t) z_1,\\
g_2 = (U+c\gamma \partial_t) z_2,
\end{cases}
\end{equation*}
we find from~\eqref{eq:zlin} and~\eqref{eq:conj2} (and neglecting that $c\gamma$ depends on $t$) the following transformed system for $\gb=(g_1,g_2)$
\begin{equation*}
\begin{cases}
\dot g_1 = g_2 \\
\dot g_2 = -L_0 g_1 + 2 c \gamma\partial_x g_2.
\end{cases}
\end{equation*}
Now, we justify formally that under the condition for $P$ given in Lemma~\ref{rk:P}, a virial argument provides asymptotic stability for $\gb$.
To simplify the discussion, we choose $c=0$ and we are thus reduced to the simple system
\begin{equation*}
\begin{cases}
\dot g_1 = g_2 \\
\dot g_2 = \partial_x^2 g_1 - P g_1.
\end{cases}
\end{equation*}
For a bounded increasing function $\psi:\R\to \R$ to be chosen ($\psi'>0$), let
\[
\cG = \int \left(2 \psi \partial_x g_1+ \psi' g_1\right) g_2.
\]
We obtain formally by integration by parts 
\begin{align*}
\dot{\cG}
&= \int \left(2 \psi \partial_x \dot g_1+\psi' \dot g_1\right) g_2
+\int \left(2 \psi \partial_x g_1+\psi' g_1\right)\dot g_2\\
&= \int \left(2 \psi \partial_x g_1+\psi' g_1\right)(\partial_x^2 g_1 -Pg_1)\\
&= -2 \int (\partial_x g_1)^2 \psi '+ \frac 12 \int g_1^2 \psi'''+ \int g_1^2 P' \psi.
\end{align*}
Setting $\zeta = \sqrt{\psi'}$ and $h_1 = g_1 \zeta$
we obtain (see the proof of Lemma~\ref{le:cK})
\begin{equation}\label{eq:heurist}
-\dot{\cG}= 2\int (\partial_x h_1)^2 + \int \left(\frac{\zeta ''}{\zeta }-\frac{(\zeta')^2}{\zeta^2}\right) h_1^2
- \int h_1^2 P ' \frac{\psi}{\psi'}.
\end{equation}
We claim that by the condition on $P$ from Lemma~\ref{rk:P} and for a suitable choice of~$\psi$, 
the above quadratic form in $h_1$ is positive.
The key idea is to define the function $\psi$ so that $P'\psi\leq 0$ and $P' \psi\not \equiv 0$, in order for the last term in~\eqref{eq:heurist} to be a positive contribution related to $h_1^2$.
We distinguish three cases according to the conditions in Lemma~\ref{rk:P}.
\begin{itemize}
\item If there exists $x_0\in \R$ such that $(x-x_0) P' \leq 0$, we consider a bounded increasing function $\psi$ which is negative on $(-\infty,x_0)$ and positive on $(x_0,+\infty)$.
\item If $P'\leq 0$ on $\R$, we consider
a bounded increasing function $\psi:\R\to (0,\infty)$.
\item If $P'\geq 0$ on $\R$, we consider
a bounded increasing function $\psi:\R\to (-\infty,0)$.
\end{itemize}
Then, the second term on the right-hand side of~\eqref{eq:heurist}
can be absorbed by the term $-\int h_1^2 P'\frac{\psi}{\psi'}$ for a suitable choice of $\psi$ depending on $P'$,
as in~\cite[Section~4]{KMM4}.

Observe that the condition on $P$ given in Lemma~\ref{rk:P} rules out for $L$ the possibility  of having an eigenfunction other than $Y$.
Indeed, by~\eqref{eq:fact}, if $\phi$ is an eigenfunction for $L_0$ then
$U^\star\phi$ is an eigenfunction for $L$ different from~$Y$.
This means that we cannot expect this condition to be necessary. When the condition on $P$ is not satisfied, 
it can be related to the existence of internal modes or resonances for the operator $L$, which may not be definitive obstables to asymptotic stability but may strongly complicate its proof.
Last, from the proof of Theorem~\ref{th:2}, we exhibit examples where asymptotic stability is true, with no such spectral difficulty, but for which the condition on $P$ is not satisfied. 
In Corollary~\ref{cor:new} and Remark~\ref{rk:perturb}
we show that some flexibility in the proof of asymptotic stability can be used to treat some of these cases.

To prove Theorem~\ref{th:2}, we introduce tools to justify the above heuristic arguments and to take into account the following technical difficulties:
\begin{itemize}
\item Regularity issues related to the transformed problem: the function $\gb$ is only bounded in $L^2\times H^{-1}$, which leads us to introduce a regularization $\ff$ of this function (Lemma \ref{le:4p7}).
\item Localization of the virial arguments: 
the proof of Theorem~\ref{th:2} involves a preliminary virial argument on the function $\zz$ (Lemma~\ref{le:cZ})
and a key virial argument on the function $\ff$ related to the formal computation~\eqref{eq:heurist}
(Proposition~\ref{pr:3}) .
Since these two virial arguments involve a bounded function $\psi$, which is a bounded approximation of the function $x\mapsto x$, we only obtain control on localized versions of the functions $\zz$ and $\ff$.
For technical reasons, the functions $\zz$ and $\ff$ have to be localized at two different
scales, denoted respectively by $A$ and $B$ and defined such that $1\ll B \ll A$. See notation introduced in~\S\ref{s:notvirial}.
\item Nonzero speed: the usual virial argument for linearized Klein-Gordon type equations has to be adapted to the non zero speed case $c\neq 0$. General computations are presented in Lemma~\ref{le:cK}, and then applied to both $\zz$ and $\ff$, respectively
in Lemma~\ref{le:cZ} and Proposition~\ref{pr:3}.
A consequence of nonzero speed is that virial type arguments give control only on a particular directional derivative (Lemma~\ref{le:cK}), and thus most estimates have to be integrated in time (see for example the estimate of the cubic terms in Lemma~\ref{le:cubic}).
\end{itemize}
In the rest of this section, the proof of Theorem~\ref{th:2} is organized in three steps.
\begin{description}
\item[Proposition \ref{pr:2}] The virial argument in the variable $\zz$ provides the first key estimate for asymptotic stability, controlling the directional derivative of $\zz$ at the large scale $A$ by a weighted $L^2$ norm of $\zz$.
\item[Proposition \ref{le:coer}] The weighted $L^2$ norm of $\zz$ is estimated in terms of the function $\ff$. It is decisive to use the orthogonality conditions \eqref{ortho:z1},\eqref{ortho:z2} on $\zz$ since they guarantee that no information is lost when applying the transformation $U$.
\item[Proposition \ref{pr:3}] The main step of the proof is a second virial argument in the regularization $\ff$ of the tranformed function $\gb$. Using assumption~\eqref{on:V}, localized versions of~$\ff$ at the scale $B$ are controlled by the weighted $L^2$ norm of $\zz$ \emph{multiplied by an arbitrarily small factor}.
This is due to the fact that ignoring the nonlinear terms, the localization of the virial argument and regularity issues, the virial argument in the transformed variable yields a positive quadratic form in~\eqref{eq:heurist}.

\end{description}
Since we use several auxilliary functions related to the kink perturbation throughout the proof, as a guide for the reader, we list these functions in Table~\ref{Fig:table}.
\begin{table}[ht]
\begin{tabular}{|l l l l|} 
 \hline
 Notation & Use & Definition & Regularity \\ [0.5ex] 
 \hline\hline
 $\ab=(a_1,a_2)$ & Generic virial variable & \eqref{systab} & $H^1\times L^2$     \\
  \hline
 $\uu=(u_1,u_2)$ & Perturbation in the moving frame & \eqref{u1u2} & $H^1\times L^2$ \\ 
 \hline
  $\zz=(z_1,z_2)$ & Perturbation in the fixed frame & \eqref{z1z2} & $H^1\times L^2$     \\
   \hline
  $\ww=(w_1,w_2)$ & Localization of $\zz$ at scale $A$ & \eqref{w1w2} & $H^1\times L^2$ \\
  \hline
  $\kk=(k_1,k_2)$ & Directional derivative of $\zz$ &\eqref{def:kk} & $L^2\times H^{-1}$  \\
 \hline
  $\jj=(j_1,j_2)$ & Regularization of $\kk$ & \eqref{def:jj} & $H^2 \times H^1$   \\
 \hline
$\gb=(g_1,g_2)$  & Transformed function of $\zz$ & \eqref{def:g} & $L^2\times H^{-1}$    \\
 \hline
 $\ff=(f_1,f_2)$ & Regularization of $\gb$  & \eqref{def:f} & $H^2\times H^1$  \\
 \hline
 $\hh=(h_1,h_2)$ & Localization of $\ff$ at scale $B$ & \eqref{h1h2} & $H^2 \times H^1$   \\  
 \hline 
\end{tabular}
\vspace{0.5ex}
\caption{Functions related to the kink perturbation used in \S\ref{S:4}.}\label{Fig:table}
\end{table}
\subsection{General virial computation}
In the sequel, we will need computations of virial type on solutions of two linearized systems related to~\eqref{syst}, the first of them being the system~\eqref{eq:z} satisfied by $\zz$ itself. Thus, we present a computation that is sufficiently general to be applied to both of  these cases.
Let $\ab=(a_1(t,x),a_2(t,x))$ be a solution in $H^1\times L^2$ of the system
\begin{equation}\label{systab}
\begin{cases}
\dot a_1 = a_2 +\Omega_1(\ab)+\beta_1 \\
\dot a_2 = \partial_x^2 a_1 - f(x,a_1) + 2c\gamma \partial_x a_2
+\Omega_2(\ab)+ \beta_2,
\end{cases}
\end{equation}
where $f=f(x,a_1)$ satisfies $f(x,0)=0$ and $\betab=(\beta_1(t,x),\beta_2(t,x))$ is a given function, and where we recall the notation
\begin{align*}
\Omega_1(\ab)&=(\dot y-c)\gamma \partial_x a_1 - \dot c c\gamma^2 \Lambda a_1,\\
\Omega_2(\ab)&=(\dot y-c)\gamma \partial_x a_2 - \dot c c\gamma^2 \Lambda a_2
+\dot c \gamma \partial_x a_1.
\end{align*}
The virial computation has to be adapted to the presence of the transport term
$2c\gamma\partial_x a_2$ at the linear level.
Denote $F(x,a_1)=\int_0^{a_1} f(x,a) \ud a$.
Consider a bounded $\cC^3$ function $\varphi=\varphi(x)$ and set
\[
\cA = \cA_1 + c \gamma (\cA_2 +\cA_3),
\]
where
\begin{align*}
\cA_1 & = \int \left(2 \varphi \partial_x a_1+ \varphi' a_1 \right) a_2 ,\\
\cA_2 & = - 2 \int (\partial_x a_1)^2 \varphi,\\
\cA_3 & = \int \left[ a_2^2 + (\partial_x a_1)^2 + 2 F(a_1) \right] \varphi.
\end{align*}
The term $\cA_1$ corresponds to the standard definition of a localized virial functional for wave type equations. The terms $\cA_2$ and $\cA_3$ are introduced because of the transport term $2 c \gamma\partial_x a_2$ in the equation of $\dot a_2$.

\begin{lemma}[General virial identity]\label{le:cK}
Let $\ab$ be a solution in $H^1\times L^2$ of~\eqref{systab}.
\begin{enumerate}
\item It holds
\begin{align*}
\dot \cA & = -2 \int (\partial_x a_1+ c \gamma a_2)^2\varphi'
- 2 \int (\partial_x a_1 + c\gamma a_2) a_1 \varphi''
- \frac 12 \int a_1^2 \varphi''' \\
&\quad +\int \left(2F(a_1) - a_1 f(a_1)\right)\varphi'
+ 2\int \varphi (\partial_x F) (a_1)\\
&\quad + (\dot y-c) \gamma \cI +\dot c \gamma \cJ 
+ \cN +\dot c \gamma^3 (\cA_2+\cA_3),
\end{align*}
where
\begin{align*}
\cI &= 
\int \left[c\gamma (\partial_x a_1)^2 -c\gamma a_2^2-2 (\partial_x a_1) a_2 \right] \varphi'- \int a_1 a_2 \varphi''\\
&\quad -2c\gamma\int[ F(a_1)\varphi'+(\partial_x F)(a_1)\varphi],
\end{align*}
\begin{align*}
\cJ &= 
\int 2 (\partial_x a_1)^2 (1+c^2\gamma^2) \varphi
+ c\gamma[-c\gamma (\partial_x a_1)^2 +2(\partial_x a_1) a_2 
+c \gamma a_2^2] (x\varphi)' \\
&\quad -\frac 12 \int a_1^2\varphi'' +c\gamma \int a_1 a_2 (x\varphi')' 
+2c^2\gamma^2\int [ (x\varphi)'F(a_1)+x\varphi (\partial_x F)(a_1)],
\end{align*}
and
\begin{align*}
\cN & = \int (2 \varphi \partial_x \beta_1 + \varphi' \beta_1)a_2
 + \int (2 \varphi \partial_x a_1 + \varphi' a_1) \beta_2\\
&\quad + 2c\gamma \int \left[-\varphi (\partial_x \beta_1)(\partial_x a_1) + \varphi \beta_1 f(a_1) \right]+ 2 c\gamma\int \varphi\beta_2 a_2 .
\end{align*}
\item Assuming that $\varphi'>0$ and setting
\[
\zeta = \sqrt{\varphi'}, \quad b_1 = a_1 \zeta,\quad b_2 = a_2 \zeta,
\]
it holds
\begin{align*}
\dot \cA & = 
- 2 \int (\partial_x b_1 + c \gamma b_2)^2
-\int b_1^2 \left(\frac{\zeta''}{\zeta} - \frac{(\zeta')^2}{\zeta^2}\right) \\
&\quad 
+\int \left(2F(a_1) -a_1 f(a_1)\right)\varphi'
+ 2\int \varphi (\partial_x F) (a_1)\\
&\quad + (\dot y-c) \gamma \cI +\dot c \gamma \cJ 
+ \cN +\dot c \gamma^3 (\cA_2+\cA_3).
\end{align*}
\end{enumerate}
\end{lemma}
\begin{proof}
We set
\[
\alpha_1=\Omega_1(\ab)+\beta_1,\quad
\alpha_2=\Omega_2(\ab)+\beta_2.
\]
We compute $\dot \cA$ using the system~\eqref{systab} and various integrations by parts.
We will use that for any functions $f,g$, integrating by parts, it holds
\[
\int (2\varphi f' + \varphi' f) g = - \int (2\varphi g' + \varphi' g) f,\quad
\int (2\varphi f' + \varphi' f) f = 0.
\]
First,
\begin{align*}
\dot \cA_1 & = \int \left(2 \varphi \partial_x \dot a_1+ \varphi' \dot a_1 \right) a_2
+\left(2 \varphi \partial_x a_1+ \varphi' a_1 \right) \dot a_2\\
& = \int \left(2 \varphi \partial_x a_1+ \varphi' a_1 \right) \left(\partial_x^2 a_1 - f(a_1) + 2 c\gamma \partial_x a_2\right)\\
&\quad - \int \left(2 \varphi \partial_x a_2 + \varphi' a_2 \right) \alpha_1 
+ \int \left(2 \varphi \partial_x a_1 + \varphi' a_1 \right) \alpha_2.
\end{align*}
We observe that
\[
\int \left(2 \varphi \partial_x a_1+ \varphi' a_1 \right) \partial_x^2 a_1
=-2 \int (\partial_x a_1)^2 \varphi' + \frac 12 \int a_1^2 \varphi''',
\]
and
\begin{align*}
- \int \left(2 \varphi \partial_x a_1+ \varphi' a_1 \right) f(a_1)
& = - \int \left[ 2 \varphi \partial_x \{F(a_1)\} - 2 \varphi (\partial_x F) (a_1) + \varphi' a_1 f(a_1)\right]\\
& = - \int \varphi' \left[a_1 f(a_1)-2F(a_1)\right] + 2\int \varphi (\partial_x F) (a_1).
\end{align*}
Moreover,
\begin{align*}
\int \left(2 \varphi \partial_x a_1+ \varphi' a_1 \right) \left(\partial_x a_2\right)
& = 2\int \varphi (\partial_x a_1)(\partial_x a_2) 
- \int \varphi' (\partial_x a_1) a_2 - \int \varphi '' a_1 a_2.
\end{align*}
Thus,
\begin{align*}
\dot \cA_1 & = - \int \left[2 (\partial_x a_1)^2 +\left(a_1 f(a_1)-2F(a_1)\right)\right]\varphi' 
+ 2\int \varphi (\partial_x F) (a_1) + \frac 12 \int a_1^2 \varphi'''\\
& \quad + 4 c \gamma \int (\partial_x a_1)(\partial_x a_2) \varphi 
- 2 c\gamma \int (\partial_x a_1) a_2 \varphi' - 2 c\gamma \int a_1 a_2 \varphi''\\
& \quad - \int \left(2 \varphi \partial_x a_2 + \varphi' a_2 \right) \alpha_1 
+ \int \left(2 \varphi \partial_x a_1 + \varphi' a_1 \right) \alpha_2.
\end{align*}
Second,
\begin{equation*}
\dot \cA_2
 = - 4 \int (\partial_x a_1)(\partial_x \dot a_1) \varphi
 = - 4 \int (\partial_x a_1)(\partial_x a_2) \varphi 
- 4 \int (\partial_x a_1)(\partial_x \alpha_1) \varphi .
\end{equation*}
Last,
\begin{align*}
\dot \cA_3 & = 
2 \int \left[a_2 \dot a_2 + (\partial_x a_1)(\partial_x \dot a_1) + f(a_1) \dot a_1 \right] \varphi\\
& = 2 \int \left[ a_2 (\partial_x^2 a_1) - a_2 f(a_1) +2c\gamma (\partial_x a_2) a_2 
+ (\partial_x a_1) (\partial_x a_2) + f(a_1) a_2\right] \varphi\\
& \quad + 2 \int \left[ a_2 \alpha_2 + (\partial_x a_1) (\partial_x \alpha_1) + f(a_1) \alpha_1 \right]\varphi.
\end{align*}
Thus,
\begin{align*}
\dot \cA_3 & = - 2 \int (\partial_x a_1) a_2 \varphi '- 2 c \gamma \int a_2^2 \varphi'\\
&\quad + 2 \int \left[ (\partial_x a_1)(\partial_x \alpha_1)+ f(a_1) \alpha_1 \right]\varphi +2 \int \alpha_2 \varphi a_2.
\end{align*}
Gathering the above computations 
and using $\frac {\ud}{\ud t} (c\gamma) = \dot c \gamma^3$, we obtain 
\begin{align*}
\dot \cA & = -2 \int (\partial_x a_1+ c \gamma a_2)^2\varphi'
- 2 \int (\partial_x a_1 + c\gamma a_2) a_1 \varphi''
- \frac 12 \int a_1^2 \varphi''' \\
&\quad +\int \left(2F(a_1)-a_1 f(a_1)\right)\varphi'
+ 2\int \varphi (\partial_x F) (a_1)
+ \cN(\alphab)+\dot c \gamma^3 (\cA_2+\cA_3),
\end{align*}
where $\cN(\alphab)$ is defined by
\begin{align*}
\cN(\alphab) & = \int (2 \varphi \partial_x \alpha_1 + \varphi' \alpha_1) a_2
 + \int (2 \varphi \partial_x a_1 + \varphi' a_1) \alpha_2\\
&\quad + 2c\gamma \int \left[-\varphi (\partial_x \alpha_1)(\partial_x a_1) + \varphi \alpha_1 f(a_1) \right]+ 2 c\gamma\int \varphi\alpha_2 a_2.
\end{align*}
To complete the proof of (1), we expand $\cN(\alphab)$,
using $\alphab=\Omegab(\ab)+\betab$ so that
$\cN(\alphab)=\cN(\Omegab(\ab))+\cN(\betab)$. In the expression of $\cN(\Omegab(\ab))$, terms multiplying $(\dot y-c)\gamma$ are denoted by $\cI$. We compute using cancellations and integration by parts
\begin{align*}
\cI&= - \int (2\varphi \partial_x a_2+\varphi' a_2) \partial_x a_1
+\int (2\varphi \partial_x a_1+\varphi' a_1) \partial_x a_2\\
&\quad +2 c\gamma \int [-\varphi (\partial_x^2 a_1)(\partial_x a_1)
+\varphi(\partial_x a_1)f(a_1)]
+2 c\gamma \int \varphi (\partial_x a_2) a_2\\
&=
\int \left[c\gamma (\partial_x a_1)^2 -c\gamma a_2^2-2 (\partial_x a_1) a_2 \right] \varphi'- \int a_1 a_2 \varphi''\\
&\quad -2c\gamma\int[ F(a_1)\varphi'+(\partial_x F)(a_1)\varphi].
\end{align*}
Similarly, terms multiplying $\dot c\gamma$ in the expression of $\cN(\Omegab(\ab))$ are denoted by $\cJ$ and we compute
\begin{align*}
\cJ
& = c\gamma \int (2\varphi \partial_x a_2+\varphi' a_2) \Lambda a_1
- c\gamma \int (2\varphi \partial_x a_1+\varphi' a_1) \Lambda a_2\\
&\quad +\int (2\varphi \partial_x a_1+\varphi' a_1) \partial_x a_1- 2 c^2 \gamma^2 \int [-\varphi (\partial_x \Lambda a_1)(\partial_x a_1)
+\varphi (\Lambda a_1) f(a_1)]\\
&\quad - 2 c^2 \gamma^2 \int \varphi (\Lambda a_2) a_2
+2c\gamma \int \varphi (\partial_x a_1) a_2\\
&=\int\left\{2 (\partial_x a_1)^2 (1+c^2\gamma^2) \varphi
+ c\gamma[-c\gamma (\partial_x a_1)^2 +2(\partial_x a_1) a_2
+c \gamma a_2^2] (x\varphi)'\right\} \\
&\quad -\frac 12 \int a_1^2\varphi'' +c\gamma \int a_1 a_2 (x\varphi')' 
+2c^2\gamma^2\int [ (x\varphi)'F(a_1)+x\varphi (\partial_x F)(a_1)].
\end{align*}
This justifies (1).

Last, we use $\varphi'=\zeta^2$ and $b_1 = a_1 \zeta$, $b_2= a_2 \zeta$
to expand
\begin{align*}
\int (\partial_x b_1 + c \gamma b_2)^2 
&= \int \left( (\partial_x a_1 + c\gamma a_2) \zeta + a_1 \zeta'\right)^2\\
&= \int (\partial_x a_1+c\gamma a_2)^2 \varphi'
+\int a_1^2 (\zeta')^2 + \int (\partial_x a_1 +c\gamma a_2) a_1 \varphi''.
\end{align*}
From this and
$2 (\zeta')^2 -\frac 12 \varphi'''= (\zeta')^2-\zeta'' \zeta$,
we see that (1) implies (2).
\end{proof}

\subsection{Notation for virial arguments and repulsivity}\label{s:notvirial}
First, we choose a cut-off function adapted to the potential $P$.
From the assumption~\eqref{on:W}, and Lemma~\ref{rk:P},
the function $P'$ is continuous. Since we assume $P'\not \equiv 0$, there exists
$C_1>0$ and $x_1,x_2\in \R$, $x_1<x_2$ such that
\begin{equation*}
\mbox{for all $x\in [x_1,x_2]$, } |P'(x)|\geq \frac 1{C_1}.
\end{equation*}
We define $\bar x = \frac 12 (x_1+x_2)$ and a smooth nondecreasing function $\eta:\R\to \R$ such that
\begin{equation*}
\eta(x) =\begin{cases} 
-1 &\mbox{if $x<x_1$}\\
1 &\mbox{if $x>x_2$}
\end{cases}\quad \mbox{and}\quad \eta(\bar x)=0.
\end{equation*}
In particular, it holds $(x-\bar x)\eta(x) \geq 0$ on $\R$.
For any $K\geq 1$, we define the function $\zeta_K$ as follows
\[
\zeta_K(x)=\exp\left(-\frac \omega K (x-\bar x)\eta(x) \right).
\]
(Recall that $\omega$ is defined in~\eqref{eq:Hasym}.)
For $A\geq 10$, we define the following function $\varphi_A$
\begin{equation*}
\varphi_A(x)=\int_0^x \zeta_A^2(y) dy,\quad x\in \R.
\end{equation*}
For $B\geq 10$, we define
\begin{equation*}
\varphi_B(x)= \int_{x_0}^x \zeta_B^2(y) dy,\quad x\in \R,
\end{equation*}
where, following the three cases in Lemma~\ref{rk:P}, 
\begin{itemize}
\item either $x_0\in \R$ is such that $(x-x_0) P'\leq 0$ on $\R$,
\item or $x_0=-\infty$ if $P'\leq 0$ on $\R$,
\item or $x_0=+\infty$ if $P'\geq 0$ on $\R$.
\end{itemize}
In the three cases, since $|P'|> 0$ on $[x_1,x_2]$, we have $x_0\not \in [x_1,x_2]$.
Moreover, by the definition of $\zeta_B$ and $(x-\bar x)\eta(x) \geq 0$, 
there exists $C_2>0$ such that,
for all $B\geq 1$, for all $x\in [x_1,x_2]$,
\[
|\varphi_B(x) |= \left| \int_{x_0}^x \zeta_B^2(y) dy \right|
\geq \left| \int_{x_0}^x \zeta_{B=1}^2(y) dy \right|\geq \frac 1{C_2}.
\]
In particular, for $C_0=C_1C_2$, the following lemma holds.
\begin{lemma}\label{le:sp}
There exist $C_0>0$ and $x_1,x_2\in \R$ with $x_1<x_2$ such that, for any $B>1$,
\begin{equation*}
\mbox{$\varphi_B P' \leq 0$ on $\R$} \quad \mbox{and}\quad \varphi_B P' \leq -\frac1{C_0} \mbox{ on $[x_1,x_2]$.}
\end{equation*}
\end{lemma} 

This property of \emph{strict repulsivity} of the potential $P$ is essential in the proof.
In particular, it excludes the resonant case $P'\equiv 0$.

Next, we consider an even, smooth cut-off function $\chi:\R\to\R$ that satisfies
\begin{equation*}
\mbox{$\chi=1$ on $[-1,1]$, $\chi=0$ on $(-\infty,-2]\cup[2,+\infty)$ and
$\chi'\leq 0$ on $[0,+\infty)$.}
\end{equation*}
We define the function $\PSI$ by setting
\begin{equation}\label{def:CHI}
\PSI(x)=\CHI^2(x)\varphi_B(x)\quad\mbox{where}\quad
\CHI(x)=\chi\left(\frac{x}{A}\right),\quad x\in \R.
\end{equation}
The functions $\varphi_A$ and $\PSI$ will be used in two distinct virial arguments at different scales $A$ and $B$ satisfying $1\ll B \ll A$.
In the final step of the proof of Theorem~\ref{th:2}, the constants $A$ and $B$ are fixed, and the parameter $\delta>0$ which controls the size of the initial data (see~\eqref{eq:smallfinal})
is then taken small enough, depending of these constants.
In what follows, we will systematically assume the following
\begin{equation}\label{delta:A}
0<\delta \leq \frac 1{A^4}.
\end{equation}

For future reference, we provide estimates that follow directly from the definitions.
We denote by $\ONE_J$ the indicator function of an interval $J$.
First,
\begin{align*}
\frac{\zeta_K'}{\zeta_K} & = 
-\frac \omega K\left( \eta(x) +(x-\bar x) \eta'(x)\right),\\
\frac{\zeta_K''}{\zeta_K} & = 
\left(\frac{\zeta_K'}{\zeta_K}\right)^2 
- \frac \omega K \left( 2 \eta'(x) + (x-\bar x) \eta''(x)\right)
\end{align*}
and
\[
 \frac{\zeta_K''}{\zeta_K}
- \frac{(\zeta_K')^2}{\zeta_K^2}
=- \frac \omega K \left( 2 \eta'(x) + (x-\bar x) \eta''(x)\right).
\]
Thus, there exists $C>1$ such that any $K\geq 1$, on $\R$
\begin{equation}\label{on:zetaK}
\begin{cases}
\frac 1C \sech\left(\frac{\omega}K x\right)\leq \zeta_K(x)\leq C \sech\left(\frac{\omega}K x\right),\\
|\zeta_K'(x)|\leq C K^{-1} \sech\left(\frac{\omega}K x\right),\\
|\zeta_K''(x)|\leq C K^{-2} \sech\left(\frac{\omega}K x\right) + CK^{-1} \ONE_{[x_1,x_2]}(x),\\
|x \zeta_K(x)|\leq C K,\quad |x\zeta_K'(x)|\leq C,
\end{cases}
\end{equation}
and
\begin{equation}\label{eq:preta}
\left| \frac{\zeta_K''}{\zeta_K}
- \frac{(\zeta_K')^2}{\zeta_K^2}\right| \leq \frac{C}{K}\ONE_{[x_1,x_2]}.
\end{equation}
Note that the expression in \eqref{eq:preta} appears in (2) of Lemma~\ref{le:cK}.
For the function $\varphi_A$, it holds on $\R$,
\begin{equation}\label{on:phiA}
\begin{cases}
0<\varphi_A'(x)\leq 1,\quad |\varphi_A(x)|\leq |x|,\quad |\varphi_A(x) |\leq CA,\\
|\varphi_A(x) H^{(k)}(x)|+|\varphi_A(x) \Lambda^{k} H(x)|\leq C\sech\left(\frac{3\omega}{4} x\right),\quad
\mbox{for $k= 1,2,3,4$.}
\end{cases}
\end{equation}
For the functions $\varphi_B$ and $\PSI$, it holds on $\R$,
\begin{equation}\label{on:PSI}
\begin{cases}
|\varphi_B(x)|\leq C B,\quad |\PSI(x)|\leq CB \ONE_{|x|<2A},\\
|\PSI'(x)|\leq C \frac BA \ONE_{A<|x|<2A}(x) + C \sech\left(\frac{\omega}K x\right) \ONE_{|x|<2A}(x).
\end{cases}
\end{equation}

Let $\rho$ be the following weight function
\begin{equation*}
\rho(x)=\sech\left( \frac{\omega}{10} x\right).
\end{equation*}
In particular, by~\eqref{on:zetaK}, for $K\geq 10$, it holds on $\R$
\begin{equation}\label{zeta:rho}
\rho \leq C \zeta_K.
\end{equation}
For any $T>0$, we consider the norm $\|\cdot\|_{T,\rho}$ defined by
\begin{equation}\label{def:NTrho}
\|z \|_{T,\rho}^2=\int_0^T\!\! \int z^2(t,x) \rho^2 (x) \ud x \ud t.
\end{equation}

\subsection{Virial computation on the linearized system}
We perform a virial computation on the function $\zz$ solution of~\eqref{eq:z}, applying Lemma~\ref{le:cK} with
$f(x,z_1)=W'(H+z_1)-W'(H)$ and so
$F(x,z_1)=W(H+z_1)-W(H)-W'(H)z_1$.
We set
\[
\cZ = \cZ_1 + c \gamma (\cZ_2 + \cZ_3) ,
\]
where
\begin{align*}
\cZ_1 & = \int \left(2 \varphi_A \partial_x z_1 + \varphi_A' z_1 \right) z_2,\\
\cZ_2 & = - 2 \int (\partial_x z_1)^2 \varphi_A,\\
\cZ_3 & = \int \left\{ z_2^2 + (\partial_x z_1)^2 + 2 \left[ W(H+z_1)- W(H)-W'(H) z_1\right] \right\} \varphi_A .
\end{align*}
We also define $\ww=(w_1,w_2)$, where
\begin{equation}\label{w1w2}
w_1 = \zeta_A z_1,\quad w_2 = \zeta_A z_2 .
\end{equation}
\begin{lemma}\label{le:cZ}
Assuming ~\eqref{delta:A} it holds
\begin{equation*}
\int_0^T \int (\partial_x w_1 + c \gamma w_2)^2
\le C A \delta^{2} + C \|z_1\|_{T,\rho}^2 + C \int_0^T \!\! \int |z_1|^3 \zeta_A^2.
\end{equation*}
\end{lemma}
\begin{proof}
From (2) of Lemma~\ref{le:cK}, it holds
\begin{align*}
\dot \cZ &
=- 2 \int (\partial_x w_1 + c \gamma w_2)^2 
-\int z_1^2 \left[ \zeta_A'' \zeta_A - (\zeta_A')^2\right]\\
&\quad + \int \left[ 2W(H+z_1)-(W'(H+z_1)+W'(H))z_1-2W(H)\right]\zeta_A^2\\
&\quad + 2\int \left[W'(H+z_1)-W''(H)z_1-W'(H)\right] H' \varphi_A\\
&\quad +\cN_{\HH} + (\dot y-c) \gamma \cI_{\zz} +\dot c \gamma \cJ_{\zz}
+\dot c \gamma^3 (\cZ_2+\cZ_3),
\end{align*}
where
\begin{align}
\cN_{\HH} & = \int (2 \varphi_A \partial_x \Omega_1(\HH) + \varphi_A' \Omega_1(\HH))z_2
 + \int (2 \varphi_A \partial_x z_1 + \varphi' z_1) \Omega_2(\HH)\nonumber \\
&\quad + 2c\gamma \int \left[-\varphi_A (\partial_x \Omega_1(\HH))(\partial_x z_1)
+ \varphi_A \Omega_1(\HH) (W'(H+z_1)-W'(H))\right]
\nonumber \\ 
&\quad + 2 c\gamma\int \varphi_A\Omega_2(\HH) z_2 ,\label{def:NH}
\end{align}
and
\begin{align}
\cI_{\zz} & = \int \left[c\gamma (\partial_x z_1)^2 -c\gamma z_2^2-2 (\partial_x z_1) z_2 \right] \varphi_A'
- \int z_1 z_2 \varphi_A ''\nonumber \\
&\quad -2c\gamma\int [W(H+z_1)-W(H)-W'(H)z_1]\varphi_A'\nonumber\\
&\quad -2c\gamma \int [W'(H+z_1)-W'(H)-W''(H)z_1] H'\varphi_A,\label{def:Iz}
\end{align}
\begin{align}
\cJ_{\zz} & =\int 2 (\partial_x z_1)^2 (1+c^2\gamma^2) \varphi_A
+c\gamma[-c\gamma (\partial_x z_1)^2 +2(\partial_x z_1) z_2 
+c \gamma z_2^2] (x\varphi_A)' \nonumber \\
&\quad -\frac 12 \int z_1^2\varphi_A'' +c\gamma \int z_1 z_2 (x\varphi_A')' \nonumber\\
&\quad +2c^2\gamma^2\int  (x\varphi_A)'[W(H+z_1)-W(H)-W'(H)z_1]\nonumber\\
&\quad +2c^2\gamma^2\int x\varphi_A H'[W'(H+z_1)-W'(H)-W''(H)z_1]. \label{def:Jz}
\end{align}
By~\eqref{eq:preta}
\[
\int z_1^2 \left| \zeta_A'' \zeta_A - (\zeta_A')^2\right|
\leq CA^{-1} \int z_1^2 \rho^2.
\]
By Taylor expansion and the bound $\|z_1\|_{L^\infty} \leq C \delta$ (from~\eqref{bd:z}), we have
\begin{equation*}
\int \left|2W(H+z_1)-(W'(H+z_1)+W'(H))z_1-2W(H)\right|\zeta_A^2
\leq C \int |z_1|^3 \zeta_A^2.
\end{equation*}
By~\eqref{on:phiA}, we have
$|\varphi_A H'|\leq \rho^2$, and so by Taylor expansion,
\begin{equation*}
\int \left|\left[W'(H+z_1)-W'(H)-W''(H)z_1\right]H' \varphi_A \right|
\leq C \int z_1^2\rho^2 .
\end{equation*}

Next, we compute and estimate the term $\cN_{\HH}$.
From~\eqref{def:NH}, we compute
\begin{equation*}
\cN_{\HH}=(\dot y -c)\gamma p_{\zz}+\dot c \gamma q_{\zz}
\end{equation*}
where
\begin{align*}
p_{\zz}&=
\int (2\varphi_A H''+\varphi_A'H')z_2\\
&\quad +2c\gamma \int \left\{ (\varphi_A H'')' z_1+\varphi_A H''[W'(H+z_1)-W'(H)]\right\},
\end{align*}
and
\begin{align*}
q_{\zz}&=-c\gamma\int (2\varphi_A(\Lambda H)'+\varphi_A'\Lambda H) z_2
\\&\quad -\int\left[2\varphi_A H''+\varphi_A'H'
-2 c^2 \gamma^2 (\varphi_A (\Lambda H)')'\right]z_1\\
&\quad -2c^2\gamma^2 \int \varphi_A (\Lambda H)'[W'(H+z_1)-W'(H)].
\end{align*}
Using the decay properties~\eqref{eq:Hasym} of the derivatives of $H$ and~\eqref{on:phiA}, we observe that
\[
|p_{\zz}|+|q_{\zz}|\leq C \| \zz\|_{L^2}.
\]
Using \eqref{def:Omega},  \eqref{bd:z} and~\eqref{bd:yc}, we obtain
\[
|\cN_{\HH}| \leq C \delta \| z_1 \rho\|_{L^2}^2.
\]
By the expressions of $\cI_{\zz}$ and $\cJ_{\zz}$ in~\eqref{def:Iz}-\eqref{def:Jz} and
\eqref{eq:Hasym},~\eqref{on:phiA}, we observe that
\[
|\cI_{\zz}|\leq C \|\zz\|_{H^1\times L^2}^2\leq C \delta^2,\quad
|\cJ_{\zz}|\leq C A\|\zz\|_{H^1\times L^2}^2 \leq C A \delta^{2}.
\]
Thus, by~\eqref{bd:yc}, we obtain
\[
|(\dot y-c) \cI_{\zz}|+|\dot c \gamma \cJ_{\zz}|
\leq C A \delta^2\| z_1 \rho\|_{L^2}^2.
\]
Last, we obtain similarly, by~\eqref{bd:yc},
\[
|\cZ_2|+|\cZ_3|\leq C A \|\zz\|_{H^1\times L^2}^2 \leq CA\delta^2,
\]
and
\[
|\dot c \cZ_2|+|\dot c \cZ_3|\leq C A \delta^2\| z_1 \rho\|_{L^2}^2.
\]
Gathering these estimates, and using \eqref{delta:A}, we have proved
\begin{equation*}
\dot \cZ \leq - 2 \int (\partial_x w_1 + c \gamma w_2)^2
+ C \int z_1^2 \rho^2+ C \int |z_1|^3 \zeta_A^2.
\end{equation*}
Integrating in time on $[0,T]$ and using
\[
|\cZ|\leq C A \|\zz\|_{H^1\times L^2}^2 \leq C A \delta^{2},
\]
the proof is complete.
\end{proof}

\subsection{Technical estimates}
First, we prove a general inequality to be used to estimate  the cubic term of Lemma~\ref{le:cZ}.
This result is analoguous to Claim~1 in~\cite{KMM4}. Here, since we need to control the cubic term by a norm of  $\partial_x w_1+ c\gamma w_2$ instead on simply $\partial_x w_1$, time integration is necessary.

\begin{lemma}\label{le:cubic} For any $T>0$,
let $w\in \cC([0,T],H^1)\cap \cC^1([0,T],L^2)$ and $K>0$
be such that 
\[(t,x)\mapsto w(t,x)\cosh\left(\frac{\omega}{K} x\right) \in L^\infty([0,T]\times \R).\]
and  $\sigma\in \cC^1([0,T], \R)$ such that
$$
\sup_{t \in [0,T]} |\dot \sigma (t)|\le \frac{\omega}{8 K}.
$$
Then, the following inequality holds
\begin{multline*}
\int_0^T\int |w|^3 \cosh\left(\frac{\omega}{K} x\right) \ud x \ud t
\\ \leq \frac{24 K}{\omega}  \|w \cosh\left(\frac{\omega}{K} x\right)\|_{L^\infty([0,T]\times \R)} \\
 \quad \times\left[\frac{K}{\omega} \int_0^T \int (\partial_x w+ \sigma \dot w)^2\ud x\ud t
+   \sup_{t\in [0,T]} |\sigma(t)|\|w(t)\|_{L^2}^2 \right].
\end{multline*}
\end{lemma}
\begin{proof}
On the one hand, for fixed $t\in [0,T]$, integrating by parts and using the assumption on $w\cosh\left(\frac{\omega}{K} x\right)$
to eliminate the terms at $\pm \infty$, we have
\[
\int |w|^3 \cosh\left(\frac{\omega}{K} x\right) \ud x
=- 3 \frac{K}{\omega}\int (\partial_x w) w |w| \sinh\left(\frac{\omega}{K} x\right) \ud x.
\]
On the other hand, for fixed $x\in \R$,
\[
 3 \int_0^T \sigma(t) ( \dot w w |w|)(t,x) \ud t
 = \sigma(T)|w(T,x)|^3 - \sigma(0) |w(0,x)|^3 
 -  \int_0^T \dot \sigma(t) |w|^3(t,x) \ud t.
\]
Thus, using the Fubini theorem
\begin{align*}
\frac{\omega}{K}\int_0^T \int |w|^3 \cosh(x) \ud x \ud t 
& = -3  \int_0^T\int (\partial_x w+\sigma \dot w) w |w| \sinh\left(\frac{\omega}{K} x\right) \ud x \ud t
\\ &\quad + \int (\sigma(T)|w(T,x)|^3 - \sigma(0) |w(0,x)|^3)\sinh\left(\frac{\omega}{K} x\right) \ud x
\\  &\quad - \int_0^T \dot \sigma(t)\int  |w|^3(t,x) \sinh \left(\frac{\omega}{K} x\right) \ud x \ud t
\end{align*}

Now, we estimate terms on  the right-hand side of the above identity.
First, we observe that for $t=0,T$, using $|\sinh x|\leq \cosh x$, it holds
\begin{align*}
&\left|\sigma(t)\int |w(t,x)|^3\sinh\left(\frac{\omega}{K} x\right) \ud x\right| \\
&\quad \leq \|w \cosh\left(\frac{\omega}{K} x\right)\|_{L^\infty([0,T]\times \R)} \sup_{t\in [0,T]} |\sigma(t)|\|w(t)\|_{L^2}^2 .
\end{align*}
and, by the assumption on the size of $\dot \sigma$,
\begin{equation*}
    \left| \int_0^T \dot \sigma(t)\int  |w|^3(t,x)  \sinh\left(\frac{\omega}{K} x\right) \ud x \ud t \right|
    \le\frac{\omega}{8 K} \int_0^T \int |w|^3 \cosh\left(\frac{\omega}{K} x\right) \ud x \ud t .
\end{equation*}
Second, by $|\sinh x|\leq \cosh x$ and the Cauchy-Schwarz inequality 
\begin{align*}
&\left|\int_0^T\int (\partial_x w+\sigma \dot w) w |w| \sinh\left(\frac{\omega}{K} x\right) \right|\\
& \leq \|w \cosh\left(\frac{\omega}{K} x\right)\|_{L^\infty([0,T]\times \R)}^{\frac 12}
\left(\int_0^T\hspace{-0.6em}\int (\partial_x w+ \sigma \dot w)^2 \right)^{\frac 12}
\left(\int_0^T\hspace{-0.6em}\int |w|^3 \cosh\left(\frac{\omega}{K} x\right) \right)^{\frac 12}
\\ & \leq \frac{\omega}{4K} \int_0^T\hspace{-0.6em}\int |w|^3 \cosh\left(\frac{\omega}{K} x\right) + \frac{K}{\omega}
\|w \cosh\left(\frac{\omega}{K} x\right)\|_{L^\infty([0,T]\times \R)}
\int_0^T\hspace{-0.6em}\int (\partial_x w+ \sigma \dot w)^2
\end{align*}
Combining the above estimates we have proved
\begin{align*}
\frac{\omega}{8K }\int_0^T \int |w|^3 \cosh\left(\frac{\omega}{K} x\right) 
&\leq \frac{3 K}{\omega} \|w \cosh\left(\frac{\omega}{K} x\right)\|_{L^\infty([0,T]\times \R)}
 \int_0^T\int (\partial_x w+ \sigma \dot w)^2 \\
&+ 2 \|w \cosh\left(\frac{\omega}{K} x\right)\|_{L^\infty([0,T]\times \R)} \sup_{t\in [0,T]}|\sigma(t)| \|w(t)\|_{L^2}^2,
\end{align*}
which provides the desired inequality.
\end{proof}

Second, the following lemma, inspired by Lemma~4 in~\cite{KMM4}, will allow us to compare localized norms of a function $w$ at different scales $K$. As in the previous lemma, the use of the quantity $\partial_x w + \sigma \dot w$ instead of  $\partial_x w$ requires time integration.

\begin{lemma}\label{le:tech:h}
Let $J$ be any non empty open interval.
There exists $C=C(J)>0$ such that, for any $K\geq 1$, $T>0$, 
$w\in \cC([0,T],H^1)\cap \cC^1([0,T],L^2)$, and  $\sigma\in \cC^1([0,T], \R)$ such that
\[
\sup_{t \in [0,T]} |\dot \sigma (t)|\le \frac{\omega}{4 K}
\]
the following inequality holds
\begin{multline*}
 \int_0^T\int w^2(t,x) \sech\left(\frac{\omega}Kx\right) \ud x \ud t 
 \leq C K^2 \int_0^T \int (\partial_x w + \sigma \dot w)^2(t,x) \ud x\ud t\\
\qquad+ C K \int_0^T \int_J w^2(t,x) \ud x \ud t +C K \sup_{t\in [0,T]}|\sigma(t)| \|w(t)\|_{L^2}^2.
\end{multline*}
\end{lemma}
\begin{proof}
Let $y\in J$. Integrating by parts in space, one has
\[
2 \int_y^\infty (\partial_x w) w e^{-\frac \omega K(x-y)} \ud x
=\frac \omega K \int_y^\infty w^2e^{-\frac \omega K(x-y)} \ud x
- w^2(t,y).
\]
Integrating by parts in time, one has
\[
2 \int_0^T \sigma(t) \dot w (t)w(t) \ud t = \sigma(T)w^2(T)-\sigma(0) w^2(0) - \int_0^T \dot \sigma (t) w^2(t) \ud t
\]
Thus, it holds
\begin{align*}
 \frac \omega K \int_0^T \int_y^\infty w^2e^{-\frac \omega K(x-y)} \ud x \ud t
= 2 \int_0^T \int_y^\infty (\partial_x w + \sigma \dot w) w e^{-\frac \omega K(x-y)} \ud x \ud t \\
  + \int_0^T w^2(t,y) \ud t - \int_y^\infty e^{-\frac \omega K(x-y)}\left(\sigma(T) w^2(T)- \sigma(0) w^2(0)\right) \ud x\\
+\int_0^T \dot \sigma (t)\int_y^{\infty} w^2(t,x) e^{-\frac \omega K(x-y)}\ud{x} \ud t
\end{align*}
By the Cauchy-Schwarz inequality and $e^{-\frac \omega K(x-y)} \leq 1 $ for $x\geq y $, we obtain
\begin{align*}
 \left|\int_0^T \int_y^\infty (\partial_x w + \sigma \dot w) w e^{-\frac \omega K(x-y_1)}  \ud x \ud t\right| 
  &  \leq \frac \omega{4K} \int_0^T \int_y^\infty w^2 e^{-\frac \omega K (x-y)}  \ud x \ud t
\\&\quad + \frac{K}{\omega} \int_0^T \int_y^\infty (\partial_x w+ \sigma \dot w)^2  \ud x \ud t.
\end{align*}
On the other hand, by the assumption on $\dot \sigma$
\begin{align*}
  \left|  \int_0^T \dot \sigma (t)\int_y^{\infty} w^2(t,x) e^{-\frac \omega K(x-y)}\ud{x} \ud t \right|  
  \le \frac{\omega}{4 K} \int_0^T\int_y^{\infty} w^2(t,x) e^{-\frac \omega K(x-y)}\ud{x} \ud t 
\end{align*}
and thus
\begin{align*}
\frac{\omega}{4 K} \int_0^T \int_y^\infty w^2e^{-\frac \omega K(x-y)}  \ud x \ud t
\leq  \frac{2 K}{\omega}\int_0^T \int_y^\infty (\partial_x w + \sigma \dot w)^2 \ud x \ud t \\
 + \int_0^T w^2(t,y)\ud t + 2\sup_{[0,T]} |\sigma(t)| \|w(t)\|_{L^2}^2.
\end{align*}
Summing the analogous estimate for $x<y$, we obtain
\begin{align*}
\frac{\omega}{4 K} \int_0^T \int_\R w^2 e^{-\frac \omega K|x-y|}  \ud x \ud t
&\leq \frac{2K}{\omega}\int_0^T \int_\R (\partial_x w + \sigma \dot w)^2 \ud x \ud t \\
 &+ 2 \int_0^T w^2(t,y)\ud t + 4\sup_{[0,T]} |\sigma(t)| \|w(t)\|_{L^2}^2.
\end{align*}
Finally, averaging the inequality over $y\in J$ and using, for $K\ge 1$
\begin{align*}
    \frac{1}{|J|} \int_{J} e^{-\frac{\omega}{K}|x-y|} \ud y
   & \ge  \frac{e^{-\frac \omega K |x|} }{|J|} \int_{J} e^{-\frac{\omega}{K}|y|} \ud y\\
   & \ge \frac{e^{-\frac \omega K |x|} }{|J|} \int_{J} e^{- \omega|y|}  \ge C(J) \sech\left(\frac{\omega}{K} x\right) ,
\end{align*}
we obtain the desired bound.
\end{proof}

We also state the following elementary estimates that will be necessary to treat regularized functions.
The Fourier transform of a function $g$ is denoted by $\hat g$.

\begin{lemma}\label{le:4p7}
For $\varepsilon> 0$, let $\XX=(1-\varepsilon \partial_x^2)^{-1}$ be the bounded operator from $L^2$ to $H^2$
defined by its Fourier transform as
\[
\widehat{X_\epsilon g} (\xi) = \frac {\hat g(\xi)}{1+\varepsilon \xi^2}\quad \mbox{for any $g\in L^2$.}
\]
The following estimates hold.
\begin{enumerate} 
\item For any $\varepsilon\in (0,1)$ and $g\in L^2$,
\begin{equation}\label{on:Xeps}\begin{aligned}
&\|\XX g\|_{L^2} \leq \|g\|_{L^2},\quad
\|\partial_x \XX g\|_{L^2} \leq \varepsilon^{-\frac 12}\|g\|_{L^2},\\
& \|\partial_x^2 \XX g\|_{L^2} \leq \varepsilon^{-1}\|g\|_{L^2},\quad
\|\partial_x \XX^{\frac 12} g\|_{L^2} \leq \varepsilon^{-\frac 12}\|g\|_{L^2}.
\end{aligned}\end{equation}
\item There exist $\varepsilon_1>0$ and $C>0$ such that for any $\varepsilon\in (0,\varepsilon_1)$, $K\ge 1$ and $g\in L^2$,
\begin{equation}\label{on:Xepsloc}
\left\|\sech\left(\frac \omega K x\right) \XX g \right\|_{L^2}
\leq C \left\|\XX \left[\sech\left(\frac \omega K x\right) g\right]\right\|_{L^2},
\end{equation}
and
\begin{equation}\label{on:Xepslocbis}
\left\|\cosh\left(\frac \omega K x\right) \XX \left[\sech\left(\frac \omega K x\right) g\right]\right\|_{L^2}
\leq C \left\|\XX g\right\|_{L^2}.
\end{equation}
\end{enumerate}
\end{lemma}
\begin{remark}\label{rk:4.1}
In particular, \eqref{on:Xepsloc} and \eqref{on:Xepslocbis} for $K=10$ write
\begin{gather}
\|\rho \XX g\|_{L^2} \leq C \|\XX [\rho g]\|_{L^2}\leq  C \|\rho g\|_{L^2},\label{rho:Xeps}\\
\|\rho^{-1} \XX [\rho g]\|_{L^2} \leq C \|\XX g\|_{L^2}\leq  C \|g\|_{L^2}.\label{rho:Xepsbis}
\end{gather}
\end{remark}
\begin{proof}
(1) Let $f=\XX g$ so that 
$g = (1-\varepsilon \partial_x^2) f $, and by expanding and integrating by parts,
\begin{equation*}
\|g\|_{L^2}^2 = \|(1-\varepsilon \partial_x^2) f\|_{L^2}^2
= \|f\|_{L^2}^2 + 2 \varepsilon \|\partial_x f\|_{L^2}^2 + \varepsilon^2 \|\partial_x^2 f\|_{L^2}^2.
\end{equation*}
Moreover,
\[
\|\XX^\frac 12\partial_x g\|_{L^2}^2 = |\langle \XX \partial_x^2 g,g\rangle|
\leq \varepsilon^{-1} \|g\|_{L^2}^2.
\]
This implies~\eqref{on:Xeps}.

The proof of (2) is an adaptation of~\cite[Lemma 5]{KMM4}.
Let $h = \sech\left(\frac \omega K x\right) \XX g$ and $k = \XX [\sech\left(\frac \omega K x\right) g]$.
We have
\begin{align*}
g & = \cosh\left(\frac \omega K x\right) (1-\varepsilon \partial_x^2) k
 = (1-\varepsilon \partial_x^2) \left[ \cosh\left(\frac \omega K x\right) h\right] \\
 & = \cosh\left(\frac \omega K x\right) h - \frac{\varepsilon\omega^2}{K^2} \cosh\left(\frac \omega K x\right) h 
 -2 \frac{\varepsilon\omega}{K} \sinh\left(\frac \omega K x\right) h' - \varepsilon \cosh\left(\frac \omega K x\right) h''.
\end{align*}
Thus,
\[
 (1-\varepsilon \partial_x^2) k 
 = \left[ \left(1- \frac{\varepsilon\omega^2} {K^2}\right) - \varepsilon \partial_x^2\right] h 
 - 2\frac{\varepsilon\omega} K \tanh\left(\frac \omega K x\right) h'.
\]
Applying the operator $[(1-\varepsilon\omega^2 K^{-2}) - \varepsilon \partial_x^2]^{-1}$ to this identity, we obtain
\begin{align*}
h &= \left[ \left(1-\frac{\varepsilon\omega^2}{K^2}\right) - \varepsilon \partial_x^2\right]^{-1}(1-\varepsilon \partial_x^2) k\\
&\quad + 2 \frac{\varepsilon\omega} K \left[ \left(1-\frac{\varepsilon\omega^2}{K^2}\right) - \varepsilon \partial_x^2\right]^{-1}
\left[\tanh\left(\frac {\omega}K x\right)h'\right].
\end{align*}
We note that for a constant $C$ uniform for $\varepsilon$ small and $K\geq 1$,
\begin{align*}
& \| [ (1-\varepsilon\omega^2 K^{-2}) - \varepsilon \partial_x^2]^{-1} (1-\varepsilon \partial_x^2) \|_{\cLL (L^2,L^2)} \leq C,\\
& \| [ (1-\varepsilon\omega^2 K^{-2}) - \varepsilon \partial_x^2]^{-1} \partial_x \|_{\cLL (L^2,L^2)} \leq C\varepsilon^{-\frac 12}.
\end{align*}
Thus, 
\[\|[ (1-\varepsilon\omega^2 K^{-2}) - \varepsilon \partial_x^2]^{-1}(1-\varepsilon \partial_x^2) k\|_{L^2}\leq 
C \|k\|_{L^2}\]
and 
\begin{align*}
&\left\|\left[ \left(1-\frac{\varepsilon\omega^2}{K^2}\right) - \varepsilon \partial_x^2\right]^{-1}
\left[\tanh\left(\frac \omega K x\right)h'\right]\right\|_{L^2}
\\&\quad \leq 
\left\|\left[ \left(1-\frac{\varepsilon\omega^2}{K^2}\right) - \varepsilon \partial_x^2\right]^{-1}
\left[\tanh\left(\frac \omega K x\right)h\right]'\right\|_{L^2}\\
&\quad \quad 
+ \frac 1K\left\|\left[ \left(1-\frac{\varepsilon\omega^2}{K^2}\right) - \varepsilon \partial_x^2\right]^{-1}
\left[\sech^2\left(\frac \omega Kx\right) h \right]\right\|_{L^2}\leq C \varepsilon^{-\frac 12}\|h\|_{L^2}.
\end{align*}
We deduce, for a constant $C>0$ uniform for $\varepsilon$ small and $K\geq 1$,
\[
\|h\|_{L^2} \leq C \|k\|_{L^2} + C \varepsilon^{\frac 12} \|h\|_{L^2},
\]
which implies~\eqref{on:Xepsloc} for $\varepsilon$ small enough.

We prove~\eqref{on:Xepslocbis} similarly.
Using $(1-\varepsilon \partial_x^2) k=\sech\left(\frac \omega K x\right) g$, we compute
\begin{equation*}
(1-\varepsilon \partial_x^2) \left[\cosh\left(\frac \omega K x\right)k\right]
=g - 2 \frac{\varepsilon\omega}K \partial_x \left[\sinh\left(\frac \omega K x\right) k\right]
+\frac{\varepsilon\omega^2}{K^2} \cosh\left(\frac \omega K x\right) k,
\end{equation*}
and so
\[
\cosh\left(\frac \omega K x\right)k = \XX g
- 2\frac{\varepsilon\omega}K \partial_x \XX \left[\sinh\left(\frac \omega K x\right) k\right]
+\frac{\varepsilon\omega^2}{K^2} \XX\left[\cosh\left(\frac \omega K x\right) k\right]
\]
Using~\eqref{on:Xeps} and $K\geq 1$, it follows that
\[
\left\|\cosh\left(\frac \omega K x\right)k \right\|_{L^2}
\leq \|\XX g\|_{L^2} + C  \varepsilon^{\frac 12}\left\|\cosh\left(\frac \omega K x\right)k \right\|_{L^2},
\]
and the result follows for $\varepsilon$ small enough.
\end{proof}

\subsection{Transfer estimates}
The next lemma will allow us to exchange information between the components of $\zz$, at any localization scale $K\geq 10$.
\begin{lemma}\label{le:H}
For any $K\geq 10$ and any $T>0$, it holds
\[
\int_0^T\int z_2^2 \zeta_K^2
\leq C \delta^2 + C\int_0^T \int \left[ (\partial_x z_1+c\gamma z_2)^2 +z_1^2 \right] \zeta_K^2.
\]
\end{lemma}
\begin{proof}
For $K\geq 10$, let
\begin{equation}\label{def:H}
\cH = \int z_1 z_2 \zeta_K^2.
\end{equation}
We have, using~\eqref{eq:z} and integration by parts
\begin{align*}
\dot \cH 
& = \int \dot z_1 z_2 \zeta_K^2 + \int z_1 \dot z_2 \zeta_K^2 \\
& = \int z_2^2 \zeta_K^2 - \int \left\{ (\partial_x z_1)^2+z_1 \left[W'(H+z_1)- W'(H)\right]\right\}\zeta_K^2
- 2 c \gamma \int (\partial_x z_1) z_2 \zeta_K^2 \\
&\quad +\frac 12 \int z_1^2 (\zeta_K^2)''- 2 c \gamma \int z_1 z_2 (\zeta_K^2)'
+\int \left(z_2 \MM_1 + z_1 \MM_2 \right) \zeta_K^2 .
\end{align*}
We rewrite
\begin{align*}
\dot \cH 
& = (1+c^2\gamma^2) \int z_2^2 \zeta_K^2 - \int \left\{ (\partial_x z_1+c\gamma z_2)^2+z_1 \left[W'(H+z_1)- W'(H)\right]\right\}\zeta_K^2\\
&\quad +\frac 12 \int z_1^2 (\zeta_K^2)''- 2 c \gamma \int z_1 z_2 (\zeta_K^2)'
+\int \left(z_2 \MM_1 + z_1 \MM_2 \right) \zeta_K^2 .
\end{align*}
From~\eqref{def:M1},~\eqref{def:M2} and integration by parts, we have
\begin{multline*}
\int \left(z_2 \MM_1 + z_1 \MM_2 \right) \zeta_K^2 =
(\dot y-c) \gamma \int H' z_2 \zeta_K^2
-\dot c c\gamma \int (\Lambda H) z_2 \zeta_K^2 
+\dot c \gamma \int H' z_1 \zeta_K^2\\
 -(\dot y-c) \gamma \int z_1 z_2 (\zeta_K^2)'
+\dot c c \gamma^2\int z_1 z_2 \left(x \zeta_K^2\right)'
-\frac{\dot c \gamma}2 \int z_1^2 (\zeta_K^2)'.
\end{multline*}
By the properties of $H$ and $\zeta_K^2+|(\zeta_K^2)'|+ |x (\zeta_K^2)'|\leq C$ from~\eqref{on:zetaK}, \eqref{bd:z} and~\eqref{bd:yc}, we obtain
\[
\left| \int \left(z_2 \MM_1 + z_1 \MM_2 \right) \zeta_K^2 \right|
\leq C \delta \|z_1 \rho\|_{L^2}^2
\leq C \delta\|z_1 \zeta_K\|_{L^2}^2
\]
(note that $\rho\leq C \zeta_K$ follows from the assumption $K\geq 10$, see \eqref{zeta:rho}).
Now, we estimate the other terms in the expression of $\dot \cH$ above.
Using the Cauchy-Schwarz inequality,
\[
\left| 2 c \gamma \int z_1 z_2 (\zeta_K^2)' \right|
\leq \frac 12 \int z_2^2 \zeta_K^2 + C \int z_1^2 \zeta_K^2,
\]
then $|\int z_1^2 (\zeta_K^2)''|\leq C \int z_1^2 \zeta_K^2$ and last
\[
\left| \int z_1 \left[W'(H+z_1)- W'(H)\right] \zeta_K^2 \right|
\leq C \int z_1^2 \zeta_K^2,
\]
so that
\[
\dot \cH \geq \frac 12 \int z_2^2 \zeta_K^2
- \int (\partial_x z_1+c\gamma z_2)^2 \zeta_K^2 - C \int z_1^2 \zeta_K^2.
\]
Integrating in time on $[0,T]$ and using the bound
$|\cH|\leq \|z_1\|_{L^2}\|z_2\|_{L^2} \leq C \delta^2$,
we have proved the lemma.
\end{proof}

In the sequel, we will also need a similar estimate for the function $\partial_x z_1 +c\gamma z_2$.
We introduce
\begin{equation}\label{def:kk}
\begin{cases}
k_1 = \partial_x z_1 +c\gamma z_2\\
k_2 = \partial_x z_2 +c\gamma \Theta(\zz).
\end{cases}
\end{equation}
We compute from~\eqref{eq:z} the system formally satisfied by $\kk=(k_1,k_2)$.
First, using the relation $\frac {\ud}{\ud t} (c\gamma) = \gamma^3\dot c $,
\begin{align*}
\dot k_1 & = \partial_x \dot z_1 + c\gamma \dot z_2 
+ \dot c\gamma^3 z_2 \\
&= \partial_x z_2 + c\gamma \Theta(\zz) + \partial_x \MM_1 + c\gamma \MM_2 + \dot c\gamma^3 z_2 \\
& = k_2+ \OO_1,
\end{align*}
where we have set
\[
\OO_1 = \partial_x \MM_1 + c\gamma \MM_2+ \dot c\gamma^3 z_2.
\]
Second, 
\begin{align*}
\dot k_2
& = \partial_x \dot z_2 + c\gamma (\partial_x^2 \dot z_1 - W''(H+z_1)\dot z_1) + 2 c^2 \gamma^2\partial_x \dot z_2
+2 \dot c c \gamma^4 \partial_x z_2+\dot c\gamma^3 \Theta(\zz)\\
&= \partial_x \Theta(\zz) + c\gamma (\partial_x^2 z_2 - W''(H+z_1)z_2) + 2 c^2 \gamma^2 \partial_x \Theta(\zz) + \OO_2,
\end{align*}
where
\[
\OO_2= \partial_x \MM_2 + c\gamma (\partial_x^2 \MM_1 - W''(H+z_1)\MM_1) 
+ 2c^2\gamma^2 \partial_x M_2
+2 \dot c c \gamma^4 \partial_x z_2+ \dot c\gamma^3 \Theta(\zz).
\]
We observe
\begin{align*}
\partial_x \Theta(\zz)
&=\partial_x^2 (\partial_x z_1)
- W''(H+z_1) \partial_x z_1-\left[W''(H+z_1)- W''(H)\right] H'\\
&\quad+ 2 c \gamma\partial_x (\partial_x z_2).
\end{align*}
Thus,
\begin{equation*}
\dot k_2
= \partial_x^2 k_1 - W''(H+z_1) k_1-\left[W''(H+z_1)- W''(H)\right] H' + 2 c \gamma \partial_x k_2+ \OO_2.
\end{equation*}
In conclusion, $\kk$ satisfies the system
\begin{equation*}
\begin{cases}
\dot k_1 = k_2+\OO_1\\
\dot k_2 = \partial_x^2 k_1 - W''(H) k_1
+ 2 c\gamma \partial_x k_2\\
\quad\quad -\left[W''(H+z_1)- W''(H)\right] (H'+k_1) + \OO_2.
\end{cases}
\end{equation*}

Now, we define $\jj=(j_1,j_2)$ where
\begin{equation}\label{def:jj}
j_1 = \XX k_1,\quad
j_2 = \XX k_2,
\end{equation}
where $\XX=(1-\varepsilon\partial_x^2)^{-1}$, see Lemma \ref{le:4p7}.
Using
\[
(1-\varepsilon \partial_x^2) \left[ W''(H) j_1 \right]
= W''(H) k_1 - \varepsilon (W''(H))'\partial_x j_1
-\varepsilon \partial_x [(W''(H))' j_1],
\]
we obtain the system satisfied by $\jj$
\begin{equation*}
\begin{cases}
\partial_t j_1 =j_2+\XX \OO_1\\
\partial_t j_2 = \partial_x^2 j_1 - W''(H) j_1 + 2 c\gamma \partial_x j_2 \\
\quad\qquad - \varepsilon \XX \left[(W''(H))'\partial_x j_1 +\partial_x ((W''(H))' j_1)\right]\\
\quad\qquad -\XX \left[\left(W''(H+z_1)- W''(H)\right)(H'+k_1)\right]+ \XX \OO_2.
\end{cases}
\end{equation*}
\begin{lemma}\label{le:jj}
With $\varepsilon_1>0$ defined in Lemma \ref{le:4p7} there exists $\varepsilon_2\in (0,\varepsilon_1)$ such that
for any $\varepsilon\in (0,\varepsilon_2)$ the following holds. Assume ~\eqref{delta:A}
 and let $K\in [10,A]$. Then,
\begin{multline*}
\varepsilon\int_0^T\int ( j_2^2+\varepsilon (\partial_x j_2)^2) \zeta_K^2\\
\leq C \delta^2 + C \|z_1 \|_{T,\rho}^2
 + C\int_0^T \int \left[j_1^2+\varepsilon (\partial_x j_1)^2 + \varepsilon^2 (\partial_x^2 j_1)^2\right]  \zeta_K^2.
 \end{multline*}
\end{lemma}
\begin{proof}
We introduce the following functional
\[
\cK = \int k_1 j_2 \zeta_K^2.
\]
The first observation is that $\cK$ is well-defined and satisfies
$|\cK|\leq C \varepsilon^{-\frac 12} \delta^2$ since
by~\eqref{on:Xeps} and the definitions of $\kk$ and $\jj$, one has
\[
\|k_1\|_{L^2}\leq \|\zz\|_{H^1\times L^2},\quad
\|j_2\|_{L^2}=\|\XX k_2\|_{L^2}
\leq \varepsilon^{-\frac 12} \|\zz\|_{H^1\times L^2}.
\]
Next, we have, using the systems satisfied by $\kk$ and $\jj$ and integrating by parts
\begin{align*}
\dot\cK & = \int (\partial_t k_1) j_2 \zeta_K^2 + \int k_1 (\partial_t j_2)\zeta_K^2\\
& = \int (j_2^2+\varepsilon (\partial_x j_2)^2) \zeta_K^2
- \int \left[(\partial_x j_1)^2+\varepsilon (\partial_x^2 j_1)^2
+ W''(H) k_1 j_1\right]\zeta_K^2\\
&\quad
-\frac \varepsilon 2 \int j_2^2 (\zeta_K^2)''
+\frac 12 \int j_1^2 (\zeta_K^2)''
+2 c\gamma \int k_1 (\partial_x j_2) \zeta_K^2\\
&\quad - \varepsilon \int k_1 \XX \left[(W''(H))'\partial_x j_1+\partial_x[(W''(H))' j_1]\right] \zeta_K^2 \\
&\quad -\int k_1 \XX\left[\left(W''(H+z_1)- W''(H)\right) (H'+k_1)\right]\zeta_K^2\\
&\quad + \int (j_2 \OO_1 + k_1 \XX \OO_2) \zeta_K^2.
\end{align*}
We observe
\[
\left| \int W''(H) k_1 j_1 \zeta_K^2 \right|
\leq C \int (j_1^2 +k_1^2) \zeta_K^2
\leq C \int \left[j_1^2 + \varepsilon^2 (\partial_x^2 j_1)^2\right]  \zeta_K^2,
\]
and, by \eqref{on:zetaK}, for $\varepsilon$ small enough,
\[
\left|\frac \varepsilon 2 \int j_2^2 (\zeta_K^2)''\right|
\leq \frac 14 \int j_2^2 \zeta_K^2,\quad 
\left| \int j_1^2 (\zeta_K^2)''\right| \leq C \int j_1^2 \zeta_K^2.
\]
Moreover,
\begin{align*}
\left|2c\gamma \int k_1 (\partial_x j_2) \zeta_K^2\right|
&\leq \frac{\varepsilon}4 \int (\partial_x j_2)^2 \zeta_K^2 + C \varepsilon^{-1} \int k_1^2 \zeta_K^2 \\
&\leq \frac{\varepsilon}4 \int (\partial_x j_2)^2 \zeta_K^2
+C\varepsilon^{-1} \int \left[j_1^2+\varepsilon^2 (\partial_x^2 j_1)^2\right]\zeta_K^2.
\end{align*}
Next,
\begin{equation*}
\left| \varepsilon \int k_1 \XX \left\{(W''(H))'\partial_x j_1+\partial_x[(W''(H))' j_1]\right\} \zeta_K^2\right|
\leq C\int \left[j_1^2 + \varepsilon^2 (\partial_x^2 j_1)^2\right]\zeta_K^2.
\end{equation*}
Besides, by Taylor expansion and~\eqref{eq:Hasym},
\[
\left|\left(W''(H+z_1)- W''(H)\right) H'\right|
\leq C |z_1|\rho
\]
so that
\begin{multline*}
\left| \int k_1 \XX\left[\left(W''(H+z_1)- W''(H)\right) (H'+k_1)\right]\zeta_K^2\right|\\
\leq C\| k_1 \zeta_K\|_{L^2}^2
+ C\|z_1\rho \|_{L^2}^2 
\leq C \int \left[j_1^2 + \varepsilon^2 (\partial_x^2 j_1)^2\right]\zeta_K^2
+C\|z_1\rho \|_{L^2}^2.
\end{multline*}
Last, to estimate $\int (j_2 \OO_1 + k_1 \XX \OO_2) \zeta_K^2$, we do not seek special cancellation, but
we use the regularisation from $\XX$ and thus estimates~\eqref{on:Xeps}.
We start by observing
\[
\left|\int j_2 \OO_1 \zeta_K^2\right|
=\left|\int (\XX^\frac 12 k_2)\XX^\frac12(\OO_1 \zeta_K^2)\right|
\leq \|\XX^\frac 12 k_2\|_{L^2}\|\XX^\frac12( \OO_1 \zeta_K^2) \|_{L^2}.
\]
By the definition of $k_2$ and \eqref{on:Xeps}, it holds
\[
\| \XX^\frac12 k_2\|_{L^2}^2 
\leq \varepsilon^{-\frac 12}\|\zz\|_{H^1\times L^2}
\leq \varepsilon^{-\frac 12} \delta.
\]
By the definition of $\OO_1$, \eqref{on:Xeps}, \eqref{on:zetaK},  and then~\eqref{bd:yc}
and  $K\leq A\leq \delta^{-\frac 14}$, it holds
\begin{align*}
\|\XX^\frac12(\OO_1 \zeta_K^2)\|_{L^2}
&\leq \left(|\dot y -c| + |\dot c|\right)(C+C\varepsilon^{-\frac 12}K\|\zz\|_{H^1\times L^2})\\
&\leq \|z_1 \rho\|_{L^2}^2  (C+C\varepsilon^{-\frac 12}K\|\zz\|_{H^1\times L^2})
\leq C \varepsilon^{-\frac 12}\|z_1 \rho\|_{L^2}^2.
\end{align*}
Thus,
\[
\left|\int j_2 \OO_1 \zeta_K^2\right|
\leq C \varepsilon^{-1}\delta \|z_1 \rho\|_{L^2}^2.
\]
Similarly, we check that
\[
\|k_1\|_{L^2} \leq C \delta,\quad
\| (\XX \OO_2) \zeta_K^2\|_{L^2}
\leq C  \varepsilon^{-1} \|z_1 \rho\|_{L^2}^2.
\]
We conclude
\[
\left| \int (j_2 \OO_1 + k_1 \XX \OO_2) \zeta_K^2\right|
\leq C \varepsilon^{-1}\delta \|z_1 \rho\|_{L^2}^2.
\]

Combining these estimates, we have proved
\begin{align*}
\varepsilon\dot \cK \geq \frac \varepsilon2 \int (j_2^2+\varepsilon (\partial_x j_2)^2) \zeta_K^2
- C  \|z_1 \rho\|_{L^2}^2
- C \int \left[j_1^2+\varepsilon (\partial_x j_1)^2 + \varepsilon^2 (\partial_x^2 j_1)^2\right]  \zeta_K^2.
\end{align*}
Integrating this estimate in time on $[0,T]$ and using the bound
$|\cK|\leq C \varepsilon^{-\frac 12} \delta^2$ proved earlier, we have proved~\eqref{le:jj}.
\end{proof}

\subsection{First key estimate}
Using the virial argument in the variable $\zz$ (Lemma~\ref{le:cZ}) and the above technical estimates,
we are in a position to 
state the first key estimate for asymptotic stability, relating the directional derivative of $\zz$ 
on a large scale $A$ (see the definition of $\ww$ in \eqref{w1w2}) to the weighted $L^2$ norm of $z_1$
(see the definition of $\|\cdot\|_{T,\rho}$ in~\eqref{def:NTrho}). Below, $\varepsilon_2$ is the constant in the statement of Lemma \ref{le:jj}.

\begin{proposition}\label{pr:2}
There exist $\varepsilon_3\in (0,\varepsilon_2)$, $A_1 >1$  and $C>0$ such that for any $\varepsilon\in (0,\varepsilon_3)$ and $A \in [A_1,\delta^{-1/4}] $,
the following holds.
For any $T>0$ we have
\begin{equation}\label{eq:virialz}
\int_0^T \int (\partial_x w_1 + c \gamma w_2)^2 
\leq C (A \delta^2 + \|z_1\|_{T,\rho}^2).
\end{equation}
Moreover, for any $K\in [10, A/2]$,
\begin{equation}\label{eq:vir2}
\int_0^T \int (\partial_x z_1 + c\gamma z_2)^2 \zeta_K^2
\leq C (A \delta^2 + \|z_1\|_{T,\rho}^2),
\end{equation}
\begin{equation}\label{eq:vir3}
\int_0^T \int \left[(\partial_x z_1)^2+z_1^2+z_2^2\right]\zeta_K^2
\leq CK^2 (A \delta^2 +\|z_1\|_{T,\rho}^2),
\end{equation}
\begin{equation}\label{eq:vir4}
\int_0^T \int \left[j_1^2+\varepsilon (\partial_x j_1)^2+\varepsilon^2 (\partial_x^2 j_1)^2 + \varepsilon j_2^2 + \varepsilon^2 (\partial_x j_2)^2\right]\zeta_K^2
\leq C (A \delta^2 + \|z_1\|_{T,\rho}^2).
\end{equation}
\end{proposition}
\begin{proof}
First, we prove~\eqref{eq:virialz}.
Using  \eqref{on:zetaK}, we write
$$
\int_0^T \int |z_1|^3 \varphi_A'
\leq C \int_0^T\int |w_1|^3 \cosh\left(\frac{\omega x}A\right).
$$
By \eqref{eq:param}, and the assumption on $A$,
\begin{equation} \label{eq:fix_A1}
\left|\frac d{dt} (c\gamma) \right| \le C\delta^2 \le  \frac{C A_1^{-7} }{A} 
\end{equation}
and thus, by taking $A_1$ sufficiently large, we may apply Lemma~\ref{le:cubic} with $w=w_1$, $K= A$ and $\sigma= c\gamma$.
We conclude that 
\begin{align*}
\int_0^T \int |z_1|^3 \varphi_A'
&\leq C \int_0^T\int |w_1|^3 \cosh\left(\frac{x}A\right)\\
&\leq C A \|z_1\|_{L^\infty([0,T]\times \R)}
\left[ A \int_0^T \int (\partial_x w_1 + c\gamma \dot w_1)^2+ \sup_{[0,T]} \|w_1\|_{L^2}^2 \right].
\end{align*}
By~\eqref{bd:z}, $\|z_1\|_{L^\infty([0,T]\times \R)}\leq C \delta$ and $A\leq \delta^{-\frac 14}$,
we obtain
\[
\int_0^T \int |z_1|^3 \varphi_A'
\leq C A^{-2}\left[ \int_0^T \int (\partial_x w_1 + c\gamma w_2)^2
+\int_0^T \int (w_2 - \dot w_1)^2\right]
+ C A \delta^3.
\]
From $\dot z_1 - z_2 = \MM_1$, the expression of $\MM_1$ in \eqref{def:M1}, \eqref{bd:yc},
$\|z_1\|_{H^1}\leq C \delta$ and 
the estimates $\zeta_A\leq 1$, $|x|\zeta_A\leq C A$ (from \eqref{on:zetaK}), we observe that
(using also $\delta\leq A^{-4} $)
\begin{equation}\label{eq:ag}
\|w_2-\dot w_1\|_{L^2}
= \|M_1 \zeta_A\|_{L^2}
\leq C \|z_1 \rho\|_{L^2}^2 + C A \delta \|z_1 \rho\|_{L^2}^2\leq C \|z_1 \rho\|_{L^2}^2.
\end{equation}
Thus, we have proved
\[
\int_0^T \int |z_1|^3 \varphi_A'
\leq C A^{-2} \int_0^T \int (\partial_x w_1 + c\gamma w_2)^2
+ C \delta^{\frac34}\int_0^T \|z_1 \rho\|_{L^2}^2 +C A \delta^3.
\]
Estimate~\eqref{eq:virialz} now follows from Lemma~\ref{le:cZ} taking $A_1$ large enough.

Second, we prove~\eqref{eq:vir2} and~\eqref{eq:vir3}.
Let $10\leq K\leq \frac A2$. Using the inequality $(u+v)^2 \geq \frac 12u^2- v^2$, we observe
\begin{align*}
\int (\partial_x w_1 + c\gamma w_2)^2\frac{\zeta_K^2}{\zeta_A^2}
&=\int \left[ (\partial_x z_1 + c\gamma z_2)\zeta_A + z_1 \zeta_A'\right]^2 \frac{\zeta_K^2}{\zeta_A^2}\\
&\quad \geq \frac 12 \int (\partial_x z_1 + c\gamma z_2)^2\zeta_K^2 
- \int z_1^2 \left( \frac{\zeta_A'}{\zeta_A}\right)^2\zeta_K^2.
\end{align*}
Thus, by \eqref{on:zetaK},
\begin{equation}\label{TRUC}
\int (\partial_x z_1 + c\gamma z_2)^2\zeta_K^2
\leq 2 \int (\partial_x w_1 + c\gamma w_2)^2
+ \frac C{A^2}\int z_1^2 \zeta_K^2.
\end{equation}
Using $10 \leq K \leq \frac A2$ and \eqref{on:zetaK}, we see that
\[
\int z_1^2 \zeta_K^2
=\int w_1^2 \frac{\zeta_K^2}{\zeta_A^2}
\leq \int w_1^2 \sech\left(\frac \omega K x\right).
\]
By \eqref{eq:fix_A1}, we may apply Lemma~\ref{le:tech:h} on $w_1$ with $\sigma=c\gamma$.
Therefore
\begin{align*}
\int_0^T \int z_1^2 \zeta_K^2
&\leq \int_0^T \int w_1^2 \sech\left(\frac \omega K x\right)\\
&\leq C K^2 \int_0^T \int (\partial_x w_1 + c\gamma \dot w_1)^2  
+ CK \int_0^T \int w_1^2 \rho^2 + CK \delta^2\\
&\leq C K^2 \int_0^T \int (\partial_x w_1 + c\gamma w_2)^2 
+ C K^2 \|z_1\|_{T,\rho}^2+ CK\delta^2
\end{align*}
where we have used~\eqref{eq:ag} in the last line.
Using~\eqref{eq:virialz}, we obtain
\[
\int_0^T \int z_1^2 \zeta_K^2
\leq CK^2 (A\delta^2 + \|z_1\|_{T,\rho}^2),
\]
and so, by \eqref{TRUC},
\[
\int_0^T \int k_1^2 \zeta_K^2=
\int_0^T \int (\partial_x z_1 + c\gamma z_2)^2\zeta_K^2
\leq CA \delta^2 +C\|z_1\|_{T,\rho}^2,
\]
which is \eqref{eq:vir2}.
Last, by Lemma~\ref{le:H}, we obtain
\[
\int_0^T\int z_2^2 \zeta_K^2
\leq CK^2 (A \delta^2+ \|z_1\|_{T,\rho}^2)
\]
and \eqref{eq:vir3} follows.

Finally, we prove~\eqref{eq:vir4}. Indeed, expanding
\[
\int k_1^2 \zeta_K^2=
\int \left[j_1^2 + 2 \varepsilon (\partial_x j_1)^2 + \varepsilon^2 (\partial_x^2 j_1)^2\right]\zeta_K^2
- \varepsilon \int j_1^2 (\zeta_K^2)''.
\]
and using $|(\zeta_K^2)''|\leq \frac CK\zeta_K^2$ (from \eqref{on:zetaK}), for $\varepsilon$ small, we obtain
\[
\int \left[j_1^2 + \varepsilon (\partial_x j_1)^2 + \varepsilon^2 (\partial_x^2 j_1)^2\right]\zeta_K^2
\leq C \int k_1^2 \zeta_K^2.\]
We complete the proof of \eqref{eq:vir4} by using~Lemma~\ref{le:jj} and~\eqref{eq:vir2}.
\end{proof}

\subsection{Transformed problem}\label{S:4.9}
Following the heuristic strategy outlined  in \S\ref{s.4.6}, with $\zz=(z_1, z_2)$ satisfying \eqref{eq:z}, we set
\begin{equation}\label{def:g}
\begin{cases}
g_1 = U z_1 +c\gamma z_2\\
g_2 = U z_2 +c\gamma \Theta(\zz).
\end{cases}
\end{equation}
We compute from~\eqref{eq:z} the system formally satisfied by $\gb=(g_1,g_2)$.
First, using the relation $\frac {\ud}{\ud t} (c\gamma) =  \gamma^3\dot c$,
\begin{align*}
\dot g_1 & = U \dot z_1 + c\gamma \dot z_2+\dot c\gamma^3 z_2 \\
&= U z_2 + c\gamma \Theta(\zz) + U \MM_1 + c\gamma \MM_2 + \dot c\gamma^3 z_2 \\
&= g_2+ \NN_1,
\end{align*}
where we have set
\[
\NN_1 = U \MM_1 + c\gamma \MM_2+ \dot c\gamma^3 z_2.
\]
Second, 
\begin{align*}
\dot g_2
& = U \dot z_2 + c\gamma (\partial_x^2 \dot z_1 - W''(H+z_1)\dot z_1) + 2 c^2 \gamma^2\partial_x \dot z_2
+2 \dot c c \gamma^4 \partial_x z_2+\dot c\gamma^3 \Theta(\zz)\\
&= U \Theta(\zz) + c\gamma (\partial_x^2 z_2 - W''(H+z_1)z_2) + 2 c^2 \gamma^2 \partial_x \Theta(\zz) + \NN_2,
\end{align*}
where
\[
\NN_2= U \MM_2 + c\gamma (\partial_x^2 \MM_1 - W''(H+z_1)\MM_1) 
+ 2c^2\gamma^2 \partial_x M_2
+2 \dot c c \gamma^4 \partial_x z_2+ \dot c\gamma^3 \Theta(\zz).
\]
Using the definition of $\RRR$ and identities from Lemma~\ref{le:identities},
\begin{equation}\label{eq:id}
2 U\partial_x = L-L_0 + 2\partial_x U,\quad UL = L_0 U,
\end{equation}
we observe
\begin{align*}
U \Theta(\zz)
&= U (-L z_1 + 2 c\gamma \partial_x z_2 + \RRR)\\
&= -L_0(Uz_1) + c\gamma (L-L_0)z_2+2c\gamma \partial_x U z_2 + U \RRR\\
&= -L_0 g_1 + c\gamma Lz_2+2c\gamma \partial_x U z_2 + U \RRR. \end{align*}
Thus,
\begin{equation*}
\dot g_2 = -L_0 g_1+2c\gamma \partial_x g_2+\SSS+ \NN_2
\end{equation*}
where $\SSb=(0,S_2)$ and
\begin{equation*}
\SSS=-c\gamma [W''(H+z_1)-W''(H)] z_2 + U \RRR.
\end{equation*}
In conclusion, $\gb$ satisfies the system
\begin{equation}\label{eq:g}
\begin{cases}
\dot g_1 =g_2+\NN_1\\
\dot g_2 = -L_0 g_1 + 2 c\gamma \partial_x g_2 + \SSS + \NN_2,
\end{cases}
\end{equation}
and we will now simplify the expressions of $\NN_1$, $\NN_2$ and $\SSS$.
Using 
\[
U H' = 0,\quad -U \Lambda H + H' = 0,
\]
and next (from~\eqref{eq:id})
\begin{align*}
U\partial_x z_1 
&= \partial_x U z_1 +\frac 12 ( L-L_0 ) z_1
= \partial_x U z_1 + Q z_1,\\
U\Lambda z_1 - \partial_x z_1 & = \Lambda U z_1 + x Q z_1 ,
\end{align*}
where we used $Q$, defined in~\eqref{def:Q} as
\[
Q =\frac 12 ( L-L_0 )= \frac{Y'' Y - (Y')^2}{Y^2} = (\log Y)''.
\]
We observe that $\NN_1$ reads
\begin{align*}
\NN_1 & = (\dot y-c) \gamma ( U\partial_x z_1 + c\gamma \partial_x z_2)
-\dot c c\gamma^2 (U\Lambda z_1 - \partial_x z_1 + c\gamma \Lambda z_2)+ \dot c\gamma^3 z_2\nonumber\\
&=(\dot y-c) \gamma \left( \partial_x g_1+ Q z_1\right)
-\dot c c\gamma^2 \left(\Lambda g_1 + x Q z_1\right)+ \dot c\gamma^3 z_2.
\end{align*}
Similar computations lead to 
\begin{equation*}
\NN_2 
=(\dot y-c) \gamma \left( \partial_x g_2+ Q z_2\right)
-\dot c c\gamma^2 \left(\Lambda g_2 + x Q z_2\right)
+\dot c \gamma \partial_x g_1 + \dot c\gamma^3 \Theta(\zz) .
\end{equation*}
Using the notation of~\eqref{def:Omega}, we find
\begin{align}
\NN_1
&=\Omega_1(\gb)+(\dot y-c) \gamma Q z_1
-\dot c c\gamma^2 x Q z_1 + \dot c\gamma^3 z_2,\label{def:N1}\\
\NN_2
&=\Omega_2(\gb)+(\dot y-c) \gamma Q z_2 
-\dot c c\gamma^2 x Q z_2 
+ \dot c\gamma^3 \Theta(\zz).\label{def:N2}
\end{align}
Last, we observe that
\begin{equation}\label{def:S1}\begin{aligned}
\SSS
&=- g_1 \left[W''(H+z_1)-W''(H)\right]\\
&\quad + \frac {H''}{H'} \left[ W'(H+z_1)-W'(H)-W''(H+z_1) z_1\right]\\
&\quad-H' \left[W''(H+z_1)-W''(H)-W'''(H) z_1\right].
\end{aligned}\end{equation}

Now, we define
\begin{equation}\label{def:f}
f_1 = \XX g_1,\quad f_2 = \XX g_2.
\end{equation}
The next lemma provides simple estimates relating the functions $\zz$, $\gb$ and $\ff$. The constant $\varepsilon_3$ below is defined in Proposition \ref{pr:2}.
\begin{lemma}\label{le:4p10}
There exist $C>0$ and $\varepsilon_4\in (0,\varepsilon_3)$ such that for any $\varepsilon\in (0,\varepsilon_4)$,
the following estimates hold
\begin{align}
&\|g_1\|_{L^2} \leq C \|\zz\|_{H^1\times L^2}, \label{bd:gz1}\\
&\|g_1\rho\|_{L^2}\leq C  \left(\|(\partial_x z_1) \rho\|_{L^2}+\|z_2 \rho\|_{L^2}+\|z_1 \rho\|_{L^2}\right),\label{bd:gz2}\\
&\|f_1\|_{L^2}+ \varepsilon^{\frac 12}\|\partial_x f_1 \|_{L^2}+ \varepsilon\|\partial_x^2 f_1 \|_{L^2}
\leq C\|\zz\|_{H^1\times L^2},\label{bd:gz3}\\
&\varepsilon^{\frac 12}\|f_2\|_{L^2} +\varepsilon \|\partial_x f_2 \|_{L^2}
 \leq C  \|\zz\|_{H^1\times L^2},\label{bd:gz4}\\
& \|f_1 \rho\|_{L^2}+\varepsilon^{\frac 12} \|(\partial_x f_1)\rho\|_{L^2}\leq C \|g_1 \rho\|_{L^2}.\label{on:fgrho} 
\end{align}
\end{lemma}
\begin{proof}
By the definition of $g_1$ in terms of $\zz$ and 
the definition of $U$, $Uz = \partial_x z - \frac{Y'}{Y} z$, we observe that
\begin{equation*}
 |g_1|\leq C (|\partial_x z_1| + |z_1| + |z_2|),
\end{equation*}
which is enough to justify~\eqref{bd:gz1} and~\eqref{bd:gz2}.
Next,~\eqref{bd:gz3} follows easily from~\eqref{on:Xeps}.
Moreover, since
\[
f_2 = \XX g_2 \quad
\mbox{where} \quad g_2 = \partial_x z_2 - \frac{Y'}{Y} z_2 + c\gamma \Theta(\zz),
\]
estimate~\eqref{bd:gz4} follows from~\eqref{on:Xeps}. Note that $Y'/Y$ is bounded on $\R$ because of \eqref{on:ddH}.
Last, \eqref{on:fgrho} follows from~\eqref{on:Xeps} and~\eqref{rho:Xeps}.
\end{proof}

\subsection{Second key estimate}

The second key point of the proof of asymptotic stability is the estimate of the weighted $L^2$ norm of $\zz$ by  similar quantities in $\ff$.
At this point, it is essential to use the orthogonality relations \eqref{ortho:z1}, \eqref{ortho:z2} on $\zz$ since without them some information would be lost in passing to the variable $\ff$.
The following result is inspired by Lemma 6 in \cite{KMM4}.

\begin{proposition}\label{le:coer} Let $\varepsilon_4>0$ be the constant defined in Lemma \ref{le:4p10}. For any $\varepsilon\in (0,\varepsilon_4)$ and any $T>0$,
\begin{equation}\label{eq:coer}
\int_0^T\int z_1^2 \rho^2
\leq CA \delta^{4} + C \int_0^T\int \left(f_1^2+ (\partial_x f_1)^2+f_2^2\right) \rho.
\end{equation}
\end{proposition}
\begin{remark}
Observe that terms in the left-hand and right-hand sides correspond to the same level of regularity
from the definitions of $\gb$ and then $\ff$ in \eqref{def:g} and~\eqref{def:f}.
\end{remark}
\begin{proof}
First, we prove the following estimate.
\begin{equation}\label{eq:coer2}
\int z_1^2 \rho^2
\leq C\delta^2 \int z_1^2 \rho + C \int \left(f_1^2+ (\partial_x f_1)^2+f_2^2\right) \rho .
\end{equation}
Proof of \eqref{eq:coer2}.
Recall from~\eqref{def:g} and~\eqref{def:f} that
\begin{equation}\label{on:ggg}
\begin{aligned}
f_1 -\varepsilon \partial_x^2 f_1 & = g_1=U z_1 + c\gamma z_2\\
f_2 -\varepsilon \partial_x^2 f_2 & =g_2= U z_2 +c \gamma \Theta (\zz).
\end{aligned}
\end{equation}
(Observe that the first line makes sense in $L^2$ while the second line makes sense in $H^{-1}$).
Using the expression of $\Theta(\zz) = - L z_1 + \RRR + 2 c\gamma \partial_x z_2$
from \eqref{other:Theta}, and combining the two lines of~\eqref{on:ggg}, we have
\begin{equation*}
(f_2 -2c\gamma \partial_x f_1)-\varepsilon \partial_x^2 (f_2 - 2c\gamma f_1)
=U z_2 - c\gamma (L+2\partial_x U ) z_1 
+c\gamma \RRR.
\end{equation*}
From~\eqref{eq:fact} and $U^\star+2 \partial_x =U$ (see~\S\ref{s.4.6}) we obtain
\[
L+2\partial_x U 
= U^\star U + 2 \partial_x U = U^2.
\]
Thus,
\begin{equation}\label{UU}
U (z_2 - c\gamma Uz_1)=
(f_2 -2c\gamma \partial_x f_1 -c\gamma \RRR)-\varepsilon \partial_x^2 (f_2 - 2c\gamma \partial_x f_1) .
\end{equation}
We observe that for any function $v$,
\[
\frac{v''}{Y}=
\partial_x\left(\frac{v'}{Y} + \frac{Y'}{Y^2} v \right)+\frac PY v .
\]
Using   the expression of $U$ and the above formula with $v=\varepsilon (f_2 - 2c\gamma \partial_x f_1)$, we rewrite identity~\eqref{UU}  as
\begin{align*}
&\partial_x \left\{\frac 1Y \left(z_2 - c\gamma Uz_1
+\varepsilon \partial_x(f_2 -2c\gamma \partial_x f_1)\right)
+\varepsilon\frac{Y'}{Y^2} (f_2 -2c\gamma \partial_x f_1) \right\}\\
&\quad= \frac {1-\varepsilon P} Y (f_2 -2c\gamma \partial_x f_1)
- c\gamma \frac \RRR Y.
\end{align*}
Integrating on $[0,x]$, and then multiplying by $Y$, it holds, for a constant $a$,
\begin{equation}\label{gtilde}
 z_2 - c\gamma U z_1 =aY + \tilde g_1+ \tilde g_2+ \tilde g_3,
\end{equation}
where
\begin{align*}
\tilde g_1 &= -\varepsilon \partial_x(f_2 -2c\gamma \partial_x f_1)
-\varepsilon \frac{Y'}{Y} (f_2 -2c\gamma \partial_x f_1),\\
\tilde g_2&=
 Y \int_0^x \left\{\frac {1-\varepsilon P} Y(f_2 -2c\gamma \partial_x f_1)\right\},\\
\tilde g_3&=
- c\gamma  Y \int_0^x \frac\RRR Y.
\end{align*}
We use the orthogonality relation~\eqref{ortho:z1} to estimate $a$.
Indeed, we have by taking the scalar product of~\eqref{gtilde} with $Y$
\[
|a| \leq C |\langle z_2 - c\gamma U z_1, Y\rangle|
+C |\langle \tilde g_1,Y\rangle| + C |\langle\tilde g_2 , Y \rangle|+ C |\langle\tilde g_3 , Y \rangle|.
\]
Using $U^\star Y=-2Y'$ and~\eqref{ortho:z1}, it holds
\begin{equation*}
\langle z_2 - c\gamma U z_1, Y\rangle
 = \langle z_2 , Y \rangle + 2 c\gamma \langle z_1, Y'\rangle=0.
\end{equation*}
Integrating by parts
\[
\langle \tilde g_1, Y \rangle
=\varepsilon \int (f_2 -2c\gamma \partial_x f_1) Y'
-\varepsilon \int Y' (f_2 -2c\gamma \partial_x f_1)=0.
\]
By the Cauchy-Schwarz inequality and the  properties of $Y=H'$ (see Lemma~\ref{pr:H})
\begin{align}
|\tilde g_2(x)|^2
&\leq C Y^2(x) \left(\int_0^{|x|}  \frac 1{ Y^2\rho}\right)
 \int  \left(f_2^2+(\partial_x f_1)^2\right)\rho \nonumber\\
& \leq \frac C{\rho(x)} \int \left(f_2^2+(\partial_x f_1)^2\right)\rho,\label{on:g2t}
\end{align}
and so, in particular,
\[
\langle \tilde g_2, Y \rangle^2 
\leq C \int \left(f_2^2 + (\partial_x f_1)^2 \right) \rho.
\]
Moreover, since $|R_2|=|W'(H+z_1)-W'(H)-W''(H)z_1|\leq C z_1^2\leq C \delta |z_1|$, we have similarly
\begin{equation}\label{on:g3t}
|\tilde g_3(x)|^2\leq \frac{C \delta^2}{\rho(x)} \int z_1^2\rho,\quad
\langle \tilde g_3, Y \rangle^2\leq C\delta^2 \int z_1^2\rho.
\end{equation}
Thus, we have obtained the following estimate for $a$
\begin{equation}\label{bd:aa}
a^2\leq C \int (f_2^2 + (\partial_x f_1)^2 ) \rho + C \delta^2 \int z_1^2 \rho.
\end{equation}

Combining the first line of~\eqref{on:ggg} and~\eqref{gtilde}, we obtain
\begin{equation}\label{valueofz2}
z_2 = \frac 1{1+c^2\gamma^2}\left(aY + \tilde g_1+\tilde g_2+\tilde g_3 + c\gamma f_1
-\varepsilon c\gamma \partial_x^2 f_1\right),
\end{equation}
and
\begin{equation}\label{valueofz1}\begin{aligned}
U z_1 & = \frac 1{1+c^2\gamma^2}\left(-c\gamma aY -c\gamma \tilde g_1 -c\gamma( \tilde g_2+\tilde g_3)
+f_1-\varepsilon \partial_x^2 f_1\right)\\
&=\frac 1{1+c^2\gamma^2}\left(-c\gamma aY +\varepsilon c\gamma\frac{Y'}{Y} (f_2 -2c\gamma \partial_x f_1)-c\gamma (\tilde g_2 + \tilde g_3)+f_1\right)\\
&\quad +\frac \varepsilon{1+c^2\gamma^2} \partial_x\left[ c\gamma (f_2 -2c\gamma \partial_x f_1)
- \partial_x f_1\right].
\end{aligned}\end{equation}
The identity~\eqref{valueofz1} rewrites
\begin{align*}
&\partial_x \left\{ \frac {z_1}Y 
-\frac \varepsilon{1+c^2\gamma^2} \left[\frac{c\gamma} Y (f_2 -2c\gamma \partial_x f_1)
- \frac{\partial_x f_1}Y\right]\right\}\\
&\quad =\frac 1{1+c^2\gamma^2} 
\left(-c\gamma a 
 +2\varepsilon c\gamma\frac{Y'}{Y^2} (f_2 -2c\gamma \partial_x f_1)
-c\gamma \frac{\tilde g_2+\tilde g_3}Y  +\frac{f_1}Y - \varepsilon  \frac{Y'}{Y^2} \partial_x f_1\right).
\end{align*}
Thus, by integration, it holds, for a constant $b$
\begin{equation}\label{eq:morr}
 z_1 = b Y+\frac \varepsilon{1+c^2\gamma^2} \left[c\gamma(f_2 -2c\gamma \partial_x f_1)-\partial_x f_1\right]
  + \tilde k,
\end{equation}
where
\begin{multline*}
\tilde k=\frac Y{1+c^2\gamma^2} \int_0^x   \bigg(-c\gamma a 
+2\varepsilon c\gamma\frac{Y'}{Y^2} (f_2 -2c\gamma \partial_x f_1)
\\ -c\gamma \frac{\tilde g_2+\tilde g_3}Y  +\frac{f_1}Y
-\varepsilon  \frac{Y'}{Y^2} \partial_x f_1\bigg). 
\end{multline*}
Arguing as before and using~\eqref{on:g2t}, \eqref{on:g3t} and \eqref{bd:aa}, we obtain
\[
\tilde k^2\leq \frac C{\rho(x)} \int \left[ f_2^2+(\partial_x f_1)^2 + f_1^2 \right]\rho
+ \frac {C\delta^2}{\rho(x)} \int z_1^2\rho,
\]
and so
\begin{equation}\label{on:ktilde}
\int \tilde k^2 \rho^2
\leq C \int \left( f_2^2+(\partial_x f_1)^2+f_1^2\right) \rho + C \delta^2 \int z_1^2 \rho.
\end{equation}
Now, we use the orthogonality relation~\eqref{ortho:z2} to obtain a bound on $b$.
Indeed,
projecting~\eqref{eq:morr} on $(1+c^2)Y + 2c^2 xY'$, using
\[
\langle (1+c^2)Y + 2c^2 xY', Y\rangle = \|Y\|_{L^2}^2,
\]
and
\[
\langle z_1 ,\gamma ((1+c^2)Y + 2c^2 xY') \rangle + c \langle z_2, x Y \rangle =0,
\]
we obtain 
\begin{align*}
b \|Y\|_{L^2}^2 
&= -c\gamma^{-1}\langle z_2, x Y \rangle - \langle \tilde k , (1+c^2)Y + 2c^2 xY' \rangle\\
&\quad -\frac \varepsilon{1+c^2\gamma^2} \langle c\gamma(f_2 -2c\gamma \partial_x f_1)-\partial_x f_1,(1+c^2)Y + 2c^2 xY'\rangle.
\end{align*}
We estimate  $\langle z_2, x Y \rangle$ from the expression of $z_2$ in~\eqref{valueofz2}.
First, by the expression of $\tilde g_1$ and integration by parts,
\[
\langle \tilde g_1,x Y\rangle 
= \varepsilon \langle f_2 -2c\gamma \partial_x f_1,Y\rangle.
\]
Second, using~\eqref{on:g2t}
\[
\langle  \tilde g_2 ,xY\rangle^2
\leq C \int \left( f_2^2+(\partial_x f_1)^2\right) \rho,\quad
\langle  \tilde g_3 ,xY\rangle^2
\leq C \delta^2 \int z_1^2 \rho.
\]
Last, by integration by parts,
\[
\langle  \partial_x^2 f_1,xY\rangle^2
=\langle f_1, (xY)'' \rangle^2\leq \int f_1^2 \rho.
\]
Thus, using also \eqref{bd:aa}, it holds
\[
\langle z_2, xY\rangle^2\leq C \int \left( f_2^2+(\partial_x f_1)^2+f_1^2\right) \rho
+ C\delta^2 \int z_1^2\rho.
\]
Therefore, using~\eqref{on:ktilde} for the term in $\tilde k$ and the
Cauchy-Schwarz inequality for the last term in the expression of $b$, we have proved
\[
b^2\leq C \int \left( f_2^2+(\partial_x f_1)^2+f_1^2\right) \rho + C\delta^2 \int z_1^2\rho.
\]
Inserting this information in~\eqref{eq:morr} and using again~\eqref{on:ktilde}, we have proved~\eqref{eq:coer2}.

Now, we complete the proof of~\eqref{eq:coer}. Using \eqref{eq:vir3} with $K=20$,
and then the constraint $A\leq \delta^{-\frac 14}$, we have
\begin{equation*}
\delta^2 \int_0^T \int z_1^2 \rho
\leq C \delta^2 (A \delta^2 +\|z_1\|_{T,\rho}^2)
\leq C A \delta^{4} + C \delta^2 \|z_1\|_{T,\rho}^2.
\end{equation*}
Thus, \eqref{eq:coer} follows by integrating~\eqref{eq:coer2} on $[0,T]$ and taking $\delta$ small enough.
\end{proof}

\subsection{Third key estimate}\label{S:4.11}

The third key estimate relies on a Virial computation for the transformed problem
where the assumption~\eqref{on:V} is decisive.
For simplicity, we fix the regularization parameter $\varepsilon$ that appears in the definitions of $\jj$ and $\ff$
in~\eqref{def:jj} and~\eqref{def:f} and in Lemma \ref{le:4p10} in terms of $A$ as follows
\begin{equation}\label{epsilon:A}
\varepsilon = \frac 1{\sqrt{A}}.
\end{equation}

\begin{proposition}\label{pr:3}
Let $\varepsilon_4>0$ be the constant appearing in Lemma \ref{le:4p10}. There exist $A_2\geq \max(A_1,1/\varepsilon_4^2)$ such that the following is true.
Let $A\geq A_2$ and assume that ~\eqref{delta:A} and~\eqref{epsilon:A} hold.
Then, for any $T>0$,
\begin{equation}\label{eq:pr3}
\int_0^T \int \left(f_1^2+ (\partial_x f_1)^2+f_2^2\right) \rho
\leq C A \delta^2+ C  A^{-\frac 14} \|z_1\|_{T,\rho}^2.
\end{equation}
\end{proposition}
A crucial point in the estimate \eqref{eq:pr3} is that the term $\|z_1\|_{T,\rho}^2$ in the right hand side
is multiplied by the small factor $A^{-\frac 14}$. This, together with the coercivity result of Proposition \ref{eq:coer} closes the estimate of $\|z_1\|_{T,\rho}^2$.
The rest of \S\ref{S:4.11} is devoted to the proof of Proposition~\ref{pr:3}.

\subsubsection{Equation of $\ff$}
 From~\eqref{eq:g} we compute the system satisfied by $\ff=(f_1,f_2)$
\begin{equation*}
\begin{cases}
\dot f_1 =f_2+\XX \NN_1\\
\dot f_2 = - \XX L_0 g_1 + 2 c\gamma \partial_x f_2 + \XX \SSS + \XX \NN_2.
\end{cases}
\end{equation*}
Since $L_0 = -\partial_x^2 + P$, we have
\begin{align*}
(1-\varepsilon \partial_x^2) L_0 f_1 &= 
L_0 (1-\varepsilon \partial_x^2)f_1 - \varepsilon P' \partial_x f_1 - \varepsilon \partial_x (P' f_1)\\
&=L_0 g_1- \varepsilon P' \partial_x f_1 - \varepsilon \partial_x (P' f_1).
\end{align*}
Thus,
\[
\XX L_0 g_1 = L_0 f_1
+ \varepsilon \XX \left[P' \partial_x f_1 + \partial_x (P' f_1)\right].
\]
Similarly, we have
\begin{align*}
\XX \Omega_1(\gb) & = \Omega_1(\ff) - 2 \varepsilon \dot c c \gamma^2 \XX \partial_x^2 f_1,\\
\XX \Omega_2(\gb) & = \Omega_2(\ff) - 2 \varepsilon \dot c c \gamma^2 \XX \partial_x^2 f_2,
\end{align*}
to be used combined with~\eqref{def:N1} and~\eqref{def:N2}.
Therefore, the system for $(f_1,f_2)$ writes
\begin{equation}\label{eq:f}
\begin{cases}
\dot f_1 =f_2+\Omega_1(\ff)+ \xi_1 \\
\dot f_2 = - L_0 f_1 + 2 c\gamma \partial_x f_2+ \TTT+ \XX \SSS  +\Omega_2(\ff) + \xi_2,
\end{cases}
\end{equation}
where $\TTb=(0,\TTT)$,
\begin{equation}\label{def:T1}
\TTT =- \varepsilon \XX \left[P' \partial_x f_1 + \partial_x (P' f_1)\right],
\end{equation}
and $\xib=(\xi_1,\xi_2)$ is defined by
\begin{align}
\xi_1 & = -2 \varepsilon \dot c c \gamma^2 \XX \partial_x^2 f_1
+(\dot y-c) \gamma \XX (Q z_1)
-\dot cc\gamma^2 \XX (x Qz_1) + \dot c \gamma^3 \XX z_2,\label{def:xi1}\\
\xi_2 & = -2 \varepsilon \dot c c \gamma^2 \XX \partial_x^2 f_2
+(\dot y-c) \gamma \XX (Q z_2)
-\dot cc\gamma^2 \XX (x Qz_2) + \dot c \gamma^3 \XX \Theta(\zz).\label{def:xi2}
\end{align}

\subsubsection{Virial computation for $\ff$}
For some $B \geq 40$ to be fixed in \eqref{wesetB}. Recall that the function $\PSI$ is defined in~\eqref{def:CHI}.
We set
\begin{equation}\label{cF}
\cF = \cF_1 + c \gamma (\cF_2 + \cF_3),
\end{equation}
where
\begin{align*}
\cF_1 & =\int \left(2 \PSI \partial_x f_1 + \PSI' f_1 \right) f_2,\\
\cF_2 & = - 2 \int (\partial_x f_1)^2 \PSI,\\
\cF_3 & = \int \left[ f_2^2 + (\partial_x f_1)^2+ 2 P f_1^2 \right] \PSI,
\end{align*}
and
\begin{equation}\label{h1h2}
h_1 = \CHI\zeta_B f_1, \quad h_2 = \CHI\zeta_B f_2.
\end{equation}
Note that the system~\eqref{eq:f} is in a form that allows the use of Lemma~\ref{le:cK}
with $f(x,z_1)=Pz_1$ and so $F(x,z_1)=\frac 12 Pz_1^2=\frac 12 f(x,z_1) z_1$.
Using (1) of Lemma~\ref{le:cK} on $(f_1,f_2)$, it holds
\begin{align*}
\dot \cF &= -2 \int \left(\partial_x f_1 + c\gamma f_2 \right)^2\PSI'
-2\int (\partial_x f_1+ c\gamma f_2) f_1\PSI'' -\frac 12 \int f_1^2 \PSI'''\\
&\quad + \int \PSI P' f_1^2 +\cN_{\TTb} + \cN_{\SSb}
+ (\dot y-c)\gamma \cI_{\ff}+ \dot c\gamma \cJ_{\ff} + \cN_{\xib} + \dot c\gamma^3 ( \cF_2+\cF_3),
\end{align*}
where
\begin{align*}
\cI_{\ff} & = \int \left[c\gamma (\partial_x f_1)^2 -c\gamma f_2^2-2 (\partial_x f_1) f_2 \right] \PSI'
- \int f_1 f_2 \PSI''-c\gamma\int (P\PSI)' f_1^2 ,\\
\cJ_{\ff} & = \int 2 (\partial_x f_1)^2 (1+c^2\gamma^2) \PSI
+ c\gamma[-c\gamma (\partial_x f_1)^2 +2(\partial_x f_1) f_2 
+c \gamma f_2^2] (x\PSI)' \\
&\quad -\frac 12 \int f_1^2\PSI'' +c\gamma \int f_1 f_2 (x\PSI')' 
+c^2\gamma^2\int (x\PSI P)' f_1^2,
\end{align*}
\begin{align}
\cN_{\TTb}& =\int \left[2 \PSI (\partial_x f_1 +c\gamma f_2)+\PSI' f_1\right] \TTT, \label{NT}\\
\cN_{\SSb}& = 
\int \left[2 \PSI (\partial_x f_1+c\gamma f_2) + \PSI' f_1\right] \XX \SSS, \label{NS}
\end{align}
and
\begin{align*}
\cN_{\xib} & = \int (2 \PSI \partial_x \xi_1 + \PSI' \xi_1)f_2
 + \int (2 \PSI \partial_x f_1 + \PSI' f_1) \xi_2\\
&\quad + 2c\gamma \int \left[-\PSI (\partial_x \xi_1)(\partial_x f_1) + \PSI \xi_1 P f_1 \right]+ 2 c\gamma\int \PSI\xi_2 f_2 .
\end{align*}
We cannot use directly (2) of Lemma~\ref{le:cK} since the function $\PSI$ is not monotone due to the cut-off $\CHI$, but we follow closely its proof.
Using $h_1 = \CHI\zeta_B f_1$ and 
$h_2 = \CHI\zeta_B f_2$, we have
\begin{align*}
\int (\partial_x h_1 + c \gamma h_2)^2
&=\int \left( (\partial_x f_1 +c\gamma f_2) \CHI\zeta_B+f_1(\CHI\zeta_B)' \right)^2
\\
&=\int (\partial_x f_1 +c\gamma f_2)^2 \CHI^2\zeta_B^2 +f_1^2 ((\CHI\zeta_B)')^2
\\ &\quad + \int (\partial_x f_1 +c\gamma f_2) f_1(\CHI^2\zeta_B^2)'.
\end{align*}
Using $\PSI'=\CHI^2\zeta_B^2+(\CHI^2)' \varphi_B$, we obtain
\begin{align*}
&-2 \int \left(\partial_x f_1 + c\gamma f_2 \right)^2\PSI'
-2\int (\partial_x f_1+ c\gamma f_2) f_1\PSI'' -\frac 12 \int f_1^2 \PSI'''\\
&\quad = -2 \int (\partial_x h_1 + c \gamma h_2)^2
- 2 \int (\partial_x f_1+ c\gamma f_2)^2 (\CHI^2)'\varphi_B\\
&\qquad - 2 \int (\partial_x f_1+ c\gamma f_2) f_1 \left[ \PSI''-(\CHI^2\zeta_B^2)'\right]- \int f_1^2 \left[ \frac 12 \PSI''' - 2\left((\CHI\zeta_B)'\right)^2\right].
\end{align*}
Note that
\[
\PSI''-(\CHI^2\zeta_B^2)'
=((\CHI^2)'\varphi_B)',
\]
and
\begin{align*}
\frac 12 \PSI''' - 2\left((\CHI\zeta_B)'\right)^2
& = \CHI^2 \left[ \zeta_B'' \zeta_B - (\zeta_B')^2 \right]\\
& \quad +(\CHI^2)' \zeta_B'\zeta_B +(\CHI')^2\zeta_B^2
+ 3 \CHI'' \CHI \zeta_B^2 + \frac 12 (\CHI^2)''' \varphi_B.
\end{align*}
Thus,
\begin{align*}
&-2 \int \left(\partial_x f_1 + c\gamma f_2 \right)^2\PSI'
-2\int (\partial_x f_1+ c\gamma f_2) f_1\PSI'' -\frac 12 \int f_1^2 \PSI'''\\
& \quad =
-2\int (\partial_x h_1 + c \gamma h_2)^2
- \int \left( \frac{\zeta_B''}{\zeta_B}
- \frac{(\zeta_B')^2}{\zeta_B^2}\right) h_1^2
+\theta
\end{align*}
where the error term $\theta$ defined by
\begin{equation}\label{def:theta}
\begin{aligned}
\theta &= 
-2\int \left(\partial_x f_1+c\gamma f_2\right)^2(\CHI^2)' \varphi_B
 -2 \int (\partial_x f_1+c\gamma f_2) f_1 \left( (\CHI^2)' \varphi_B\right)'  \\
&\quad - \int f_1^2 \left[ \frac 12 (\CHI^2)''' \varphi_B
+ \left(3 \CHI \CHI'' + (\CHI')^2 \right) \zeta_B^2
+(\CHI^2)' \zeta_B \zeta_B'\right]
\\
&=\theta_1+\theta_2+\theta_3
\end{aligned}
\end{equation}
is related to the cut-off function $\CHI$.

In conclusion, we have obtained
\begin{equation}\label{eq:tobeint}
\begin{aligned}
\dot \cF &= 
-2\int \left[ (\partial_x h_1 + c \gamma h_2)^2+ Z_B h_1^2 \right]+\cN_{\TTb}\\
&\quad 
 + \cN_{\SSb} +\theta + (\dot y-c)\gamma \cI_{\ff}+ \dot c\gamma \cJ_{\ff} + \cN_{\xib} + \dot c\gamma^3 ( \cF_2+\cF_3),
\end{aligned}
\end{equation}
where the potential $Z_B$ is defined by
\begin{equation}\label{def:ZB}
Z_{B} = 
\frac 12 \left( \frac{\zeta_B''}{\zeta_B}
- \frac{(\zeta_B')^2}{\zeta_B^2}\right) - \frac{\varphi_B P'}{2\zeta_B^2}.
\end{equation}
 
\subsubsection{Lower bound on the potential $Z_B$}\label{S:step3}
Recall from~\eqref{eq:preta} that
for some constant $C>0$ independent of $B$
\[
\left| \frac{\zeta_B''}{\zeta_B}
- \frac{(\zeta_B')^2}{\zeta_B^2}\right| \leq
C B^{-1}\ONE_{[x_1,x_2]}.
\]
Moreover, by $\zeta_B \leq 1$ on $\R$ and Lemma~\ref{le:sp}, one has
\[
- \frac{\varphi_B P'}{2\zeta_B^2} \geq 0 \mbox{ on $\R$ and }
- \frac{\varphi_B P'}{2\zeta_B^2} \geq- \frac{\varphi_B P'}2 \geq \frac 1{2C_0} \mbox{ on $[x_1,x_2]$},
\]
where the constant $C_0$ is independent of $B$.
Thus for $B\geq B_1$, where 
\begin{equation}\label{wesetB}
B_1=\max (40,  2C_0C),
\end{equation}
it holds
\begin{equation}\label{bh1}
Z_B \geq 0 \mbox{ on $\R$ and } Z_B \geq \frac 1 {4C_0} \mbox{ on $[x_1,x_2]$.}
\end{equation}
We fix $B=B_1$ and we will not track the dependence on $B$ anymore.

\subsubsection{Technical estimates}
\begin{lemma}\label{le:yetaL}
It holds
\begin{align}
\|(\partial_x f_1+c\gamma f_2) \chi_A \rho \|_{L^2} &\leq C \|\partial_x h_1+c\gamma h_2\|_{L^2}
 +C B^{-1}\|f_1\rho\|_{L^2},\label{bd:gz7}\\
\int_{|x|<A} (\partial_x f_1+c\gamma f_2)^2 \rho 
&\leq C \int (\partial_x h_1+c\gamma h_2)^2
+C\int_{|x|<A} f_1^2 \rho,\label{pr:finir2}\\
\int_{|x|<A} f_1^2 \rho &\leq C \int h_1^2 \rho^\frac 12,  \label{bd:gz7bis}\\
\int_{|x|<2A} (\partial_x f_1+c\gamma f_2)^2&
\leq C   \int \left[(\partial_x j_1)^2 + j_2^2 + k_1^2 \right]\zeta_A^4
+ C  \int z_1^2 \rho^2,\label{bd:gz10}\\
\int_{|x|<2A} |f_1|^2&
\leq C \int \left(|\partial_x z_1|^2 + |z_2|^2 + |z_1|^2\right) \zeta_A^4.\label{bd:gz11}
\end{align}
\end{lemma}
\begin{proof}
We prove~\eqref{bd:gz7}.
By the definition of $\hh$ in \eqref{h1h2}, we have
\[
\partial_x h_1 +c\gamma h_2
=\CHI \zeta_B (\partial_x f_1+c\gamma f_2)+(\CHI\zeta_B)' f_1,
\]
and so using~\eqref{on:zetaK} and $A>B$,
\begin{align*}
\CHI^2 \zeta_B^2 (\partial_x f_1+c\gamma f_2)^2
&\leq 2(\partial_x h_1 +c\gamma h_2)^2 + 2[(\CHI \zeta_B)' f_1]^2\\
&\leq 2(\partial_x h_1 +c\gamma h_2)^2 + C B^{-2} \zeta_B^2f_1^2.
\end{align*}
In particular, since $B\geq 20$,
\begin{equation*}
\CHI^2\rho (\partial_x f_1 + c\gamma f_2)^2 
\leq C (\partial_x h_1 +c\gamma h_2)^2 +CB^{-2} f_1^2 \rho,
\end{equation*}
which implies~\eqref{bd:gz7} (multiplicating by $\rho$ and integrating on $\R$) and \eqref{pr:finir2}
(integrating on $[-A,A]$).
Moreover, since $B\geq 40$ (see \eqref{wesetB}), by the definitions of $h_1$ and $\CHI$, we have for $|x|<A$,
\[
f_1^2 \rho = h_1^2 \frac{\rho}{\zeta_B^2}\leq h_1^2 \rho^{\frac 12},
\]
which implies \eqref{bd:gz7bis}.
Last, we prove~\eqref{bd:gz10}-\eqref{bd:gz11}. Recall (\eqref{def:kk}, \eqref{def:jj}, \eqref{def:g}):
\begin{align}
g_1 & = U z_1 + c\gamma z_2 = \partial_x z_1 + c\gamma z_2 - \frac{Y'}{Y} z_1 = k_1 - \frac{Y'}{Y} z_1, \label{auxx1}\\
g_2 & = U z_2+c\gamma \Theta(\zz)=\partial_x z_2+c\gamma\Theta(\zz)-\frac{Y'}{Y} z_2=k_2 - \frac{Y'}{Y} z_2, \nonumber
\end{align}
and thus
\begin{equation}\label{on:fj}
f_1 = j_1 - \XX \left( \frac{Y'}Y z_1\right),\quad f_2 = j_2 - \XX \left( \frac{Y'}Y z_2\right).
\end{equation}
In particular,
\begin{equation*}
\partial_x f_1 + c\gamma f_2 = \partial_x j_1 + c\gamma j_2 - \XX \left( \frac{Y'}Y k_1\right)
-\XX \left( Q z_1\right),
\end{equation*}
where the function $Q= (Y'/Y)'$ is defined in~\eqref{def:Q}.
Using also~\eqref{on:Xeps},~\eqref{on:Xepsloc}, we obtain~\eqref{bd:gz10}.
Last, from \eqref{auxx1},
\begin{equation*}
\int_{|x|<2A} |f_1|^2
\leq C \int  |f_1|^2\zeta_A^4\leq C \int  |g_1|^2\zeta_A^4\
\leq C \int \left(|\partial_x z_1|^2 + |z_2|^2 + |z_1|^2\right) \zeta_A^4
\end{equation*}
which is~\eqref{bd:gz11}.
\end{proof}

\subsubsection{Control of regularization terms $\cN_{\TTb}$}
Recall $\cN_{\TTb}$ in \eqref{NT} and the expression of $\TTT$ in~\eqref{def:T1}
\[\TTT =- \varepsilon \XX \left[P' \partial_x f_1 + \partial_x (P' f_1)\right] .\]
(Note that since the function $P$ is of class $\mathcal C^1$ - see Lemma~\ref{rk:P} - we will use the regularization $X_\varepsilon$  to absorb the derivative of the term $P'f_1$ above.)
Using the Cauchy-Schwarz inequality, then $|\varphi_B|\leq C B$ (see \eqref{on:PSI}), \eqref{rho:Xepsbis}
and the decay property of $P'$ from~\eqref{on:P},
\begin{align*}
&\left| \int 2 \PSI (\partial_x f_1 +c\gamma f_2) \TTT\right|\\
&\qquad \leq C \varepsilon\| \rho  [\PSI (\partial_x f_1 +c\gamma f_2)]\|_{L^2}
\| \XX \left[\rho^{-1}(P' \partial_x f_1 + \partial_x(P' f_1))\right]\|_{L^2}\\
&\qquad \leq C B \varepsilon\| \rho \CHI(\partial_x f_1 +c\gamma f_2)\|_{L^2}
\left(\|(\partial_x f_1)\rho\|_{L^2}+\varepsilon^{-\frac 12}\|f_1\rho\|_{L^2}\right).
\end{align*}
Using~\eqref{on:fgrho} and \eqref{bd:gz7}, we obtain
\begin{align*}
\left|\int 2\PSI (\partial_x f_1 +c\gamma f_2) \TTT\right|
&\leq  CB \varepsilon^{\frac 12} \|\partial_x h_1 +c\gamma h_2\|_{L^2}\|g_1\rho\|_{L^2}
 +C\varepsilon^{\frac 12}\|g_1 \rho\|_{L^2}^2\\
&\leq \|\partial_x h_1 +c\gamma h_2\|_{L^2}^2
+C\varepsilon^\frac 12 \|g_1 \rho\|_{L^2}^2.
\end{align*}
By~\eqref{bd:gz2},
and then estimate~\eqref{eq:vir3} for $K=10$, we deduce
\begin{equation*}
\int_0^T \left|\int \PSI (\partial_x f_1 +c\gamma f_2) \TTT\right|
\leq \int_0^T \int (\partial_x h_1 +c\gamma h_2)^2
+C  \varepsilon^\frac 12 \left( A \delta^2 + \|z_1\|_{T,\rho}^2\right).
\end{equation*}
We treat the second term of $\cN_{\TTb}$ similarly
\begin{align*}
\left|\int \PSI' f_1 \TTT\right|
&=\varepsilon \left| \int   (\PSI'  f_1)  \XX (P' \partial_x f_1 + \partial_x(P' f_1)\right|\\
&\leq \varepsilon \|\rho \chi_Af_1\|_{L^2} \left(\|(\partial_x f_1)\rho\|_{L^2} +\varepsilon^{-\frac 12}\|f_1\rho\|_{L^2}\right)
\leq C \varepsilon^{\frac 12} \|g_1 \rho\|_{L^2}^2.
\end{align*}
Therefore, using $\varepsilon=A^{-\frac 12}$, 
\begin{equation}\label{bh2}
\int_0^T  |\cN_{\TTb}|
\leq \int_0^T \int (\partial_x h_1 +c\gamma h_2)^2 + C A^{-\frac14}\left(A\delta^2 + \|z_1\|_{T,\rho}^2\right).
\end{equation}

We will see that the remaining terms in the virial computation for $f$ 
(the second line of~\eqref{eq:tobeint})
 can be considered as error terms.

\subsubsection{Control of nonlinear terms $\cN_{\SSb}$}
These terms are at least cubic in the size of the perturbation $\delta$.
Recall the expressions of $\cN_{\SSb}$ in \eqref{NS} and $\SSS$ from~\eqref{def:S1}
\begin{align*}
\SSS
&=- g_1 \left[W''(H+z_1)-W''(H)\right]\\
&\quad + \frac {H''}{H'} \left[ W'(H+z_1)-W'(H)-W''(H+z_1) z_1\right]\\
&\quad-H' \left[W''(H+z_1)-W''(H)-W'''(H) z_1\right].
\end{align*}
By Taylor expansions and the definition of $g_1$ in \eqref{def:g}, we see that
\begin{equation}\label{bd:SSS}
\begin{aligned}
|\SSS|
&\leq C |z_1|\left(|g_1|+|z_1|\right)
\leq C |z_1|\left(|\partial_x z_1|+|z_2|+|z_1| \right)\\
&\leq C \delta \left(|\partial_x z_1|+|z_2|+|z_1| \right).
\end{aligned}
\end{equation}
Thus, by~\eqref{on:Xepsloc} and then~\eqref{on:Xeps} and \eqref{on:zetaK},
\begin{align*}
\int_{|x|<2A} |\XX \SSS|^2
&\leq C \int |\XX \SSS|^2 \sech^2\left(\frac {|x|}{2A}\right)
\leq C \int |\SSS|^2 \sech^2\left(\frac {|x|}{2A}\right)\\
&\leq C \delta^2 \int \left((\partial_x z_1)^2+z_2^2+z_1^2 \right) \zeta_A^4.
\end{align*}
Besides, recall from \eqref{on:PSI} that
\[
|\PSI|\leq CB \ONE_{|x|<2A}(x),\quad  |\PSI'|\leq C\ONE_{|x|<2A}(x).
\]
Therefore, by the Cauchy-Schwarz inequality
and~\eqref{bd:gz10}-\eqref{bd:gz11}, we have
\begin{align*}
|\cN_{\SSb}| & \leq C B \delta
\int \left( (\partial_x j_1)^2 + j_2^2+   (\partial_x z_1)^2 + z_2^2 + z_1^2  \right) \zeta_A^4.
\end{align*}
By~\eqref{eq:vir3} and~\eqref{eq:vir4} with $K=\frac A2$, using~$\varepsilon=A^{-\frac 12}$
and $\delta\leq A^{-4}$, we obtain
\[
\int_0^T |\cN_{\SSb}|
\leq C \delta \varepsilon^{-1}B A^2 (A \delta^2 +\|z_1\|_{T,\rho}^2)
\leq C A^{-\frac 32} (A \delta^2 +\|z_1\|_{T,\rho}^2).
\]

\subsubsection{Control of cut-off terms $\theta$} 
Terms in $\theta$ appear because of the presence of the cut-off $\CHI^2$ at scale $A$ in the definition of the function $\PSI = \CHI^2 \varphi_B$, see \eqref{def:theta}.

First, we deal with $\theta_1$.
Using~\eqref{on:PSI}, $|(\chi_A^2)' \varphi_B|\leq C \frac BA \ONE_{|x|<2A}$, and thus by~\eqref{bd:gz10},
\begin{align*}
 |\theta_1|
&\leq C \frac BA \int_{|x|<2A} (\partial_x f_1+c\gamma f_2)^2\\
&\leq \frac CA \int \left[(\partial_x j_1)^2 + j_2^2 + k_1^2\right]\zeta_A^4
+ \frac CA \int z_1^2 \rho^2.
\end{align*}
Now, we apply~\eqref{eq:vir2} and \eqref{eq:vir4} with $K=\frac A2$
so that $\zeta_A^4 \leq C \zeta_K^2$. Using~$\varepsilon=A^{-\frac 12}$,
we find
\[
\int_0^T | \theta_1| \leq \frac{C}{\varepsilon A} (A\delta^2 + \|z_1\|_{T,\rho}^2)
\leq C A^{-\frac 12}(A\delta^2 +\|z_1\|_{T,\rho}^2).
\]
Second, we estimate the term $\theta_3$ in \eqref{def:theta}.
By~\eqref{bd:gz11} 
\begin{equation*}
\left| \int f_1^2  (\CHI^2)''' \varphi_B\right|
\leq C \frac {B}{A^3} \int_{|x|<2A} f_1^2
\leq \frac C{A^3} \int \left[ (\partial_x z_1)^2+z_1^2+z_2^2\right] \zeta_{ A/2}^2.
\end{equation*}
Moreover, since $\CHI'(x)=0$ for $x$ such that $|x|<A$ and $\zeta_B(x)\leq C e^{-\frac AB}$ for 
$|x|>A$, it holds using~\eqref{bd:gz11},
\begin{multline*}
\int f_1^2 \left[\CHI|\CHI''| \zeta_B^2 + (\CHI')^2 \zeta_B^2 + |(\CHI^2)' \zeta_B \zeta_B'|\right]\\
\leq \frac{C}{A} e^{-2 \frac{A}{B}}\int_{A\leq |x|\leq 2A}  f_1^2
\leq \frac{Ce^{-2 \frac{A}{B}}}{A}  \int \left[ (\partial_x z_1)^2+z_1^2+z_2^2\right] \zeta_{ A/2}^2.
\end{multline*}
Thus, by~\eqref{eq:vir3} with $K=\frac A2$, we obtain
\[
\int_0^T |\theta_3| \leq C A^{-1} (A\delta^2 +\|z_1\|_{T,\rho}^2).
\]
Third, we consider the term
\[
\theta_2
=\int (\partial_x f_1+c\gamma f_2) f_1 \left[(\CHI^2)'' \varphi_B+ (\CHI^2)' \zeta_B^2\right].
\]
We observe by similar estimates and the Cauchy-Schwarz inequality
\[
\theta_2 
\leq  C \frac BA \int_{|x|<2A} (\partial_x f_1+c\gamma f_2)^2
+C \frac {B}{A^3} \int_{|x|<2A} f_1^2,
\]
and thus, as before
\[
\int_0^T |\theta_2|
\leq C A^{-1} (A\delta^2 +\|z_1\|_{T,\rho}^2).
\]
In conclusion, we have obtained
\begin{equation}\label{bh3}
\int_0^T |\theta| \leq  C A^{-\frac 12} (A \delta^2+ \|z_1\|_{T,\rho}^2).
\end{equation}

\subsubsection{Control of the modulation terms}
Here, we estimate the following terms from~\eqref{eq:tobeint}
\[ (\dot y-c)\gamma \cI_{\ff},\quad
\dot c\gamma \cJ_{\ff}, \quad
\cN_{\xib}, \quad
\dot c\gamma^3 ( \cF_2+\cF_3).
\]
As in the proof of Lemma~\ref{le:cZ} for $\cI_{\zz}$ and $\cJ_{\zz}$, by the expressions of $\cI_{\ff}$ and $\cJ_{\ff}$ 
in Lemma~\ref{le:cK} and
estimates~\eqref{eq:Hasym},~\eqref{on:phiA} and then~\eqref{bd:gz3}-\eqref{bd:gz4}, \eqref{bd:z}, we observe that
\[
 |\cI_{\ff}|\leq C \|\ff\|_{H^1\times L^2}^2\leq C \varepsilon^{-1}\delta^2,\quad
|\cJ_{\ff}|\leq C A\|\ff\|_{H^1\times L^2}^2 \leq C\varepsilon^{-1} A \delta^2,
\]
and so using~\eqref{bd:yc} and then $\varepsilon=A^{-\frac 12}$ and $\delta<A^{-4}$, we obtain
\[
|(\dot y-c)\gamma \cI_{\ff}| + |\dot c\gamma \cJ_{\ff}|
\leq C\varepsilon^{-1} A \delta^2 \|z_1 \rho\|_{L^2}^2
\leq C A^{-\frac {13}2} \|z_1 \rho\|_{L^2}^2.
\]
Recall that $\xib$ is defined in~\eqref{def:xi1}-\eqref{def:xi2}.
By~\eqref{bd:gz3}-\eqref{bd:gz4} and then~\eqref{bd:z}, \eqref{bd:yc}, we observe that
\begin{equation*}
\|\xi_1\|_{H^1} + \|\xi_2\|_{L^2}
 \leq C \varepsilon^{-\frac 12}\left( |\dot y-c|+|\dot c| \right) \|\zz\|_{H^1\times L^2}
 \leq C \varepsilon^{-\frac 12} \delta \|z_1 \rho\|_{L^2}^2.
\end{equation*}
Thus, by the definition of $\cN_{\xib}$, the Cauchy-Schwarz inequality and $|\PSI|\leq C B$, $|\PSI'|\leq C$,
(see \eqref{on:PSI}),  we obtain (using also $\varepsilon=A^{-\frac 12}$)
\[
|\cN_{\xib}|
\leq C \varepsilon^{-\frac 12} B \delta \|z_1 \rho\|_{L^2}^2\|\ff\|_{H^1\times L^2}
\leq C \varepsilon^{-\frac 32}  \delta^2 \|z_1 \rho\|_{L^2}^2
\leq C A^{-\frac {29}4} \|z_1\rho\|_{L^2}^2.
\]
Last, since by~\eqref{bd:gz3}-\eqref{bd:gz4}
\[
|\cF_2|+|\cF_3|\leq C B \|\ff|_{H^1\times L^2}^2 \leq C\varepsilon^{-1} \delta^2,
\]
we obtain similarly using~\eqref{bd:yc}
\[
|\dot c \cF_2|+|\dot c \cF_3|\leq C\varepsilon^{-1} \delta^2 \| z_1 \rho\|_{L^2}^2
\leq C  A^{-\frac {15}2} \|z_1 \rho\|_{L^2}^2.
\]
Thus, for the modulation terms, we have obtained
\begin{equation}\label{bh4}
\int_0^T \left\{ |(\dot y-c)\gamma \cI_{\ff}| + |\dot c\gamma \cJ_{\ff}|+
|\cN_{\xib}|+|\dot c \cF_2|+|\dot c \cF_3| \right\}\leq  C A^{-\frac {13}2} \|z_1 \|_{T,\rho}^2.
\end{equation}

\medskip

\subsubsection{First conclusion of the virial argument}
Recall $\cF$ from \eqref{cF}. By~\eqref{bd:gz3}-\eqref{bd:gz4} and then~\eqref{epsilon:A} we have
\[
|\cF|\leq C B \|\ff\|_{H^1\times L^2}^2\leq C \varepsilon^{-1}\delta^2
\leq C A^{\frac 12} \delta^{2}.
\] 
Thus,  by integrating~\eqref{eq:tobeint} on $[0,T]$
and using the estimates \eqref{bh1}, \eqref{bh2}, \eqref{bh3} and \eqref{bh4} on error terms,
we have obtained
\begin{equation}\label{eq:virialff11}
\int_0^T\left[ \int (\partial_x h_1 + c\gamma h_2)^2 +\int_{[x_1,x_2]} h_1^2\right]
\leq C A\delta^2+ CA^{-\frac 14} \|z_1\|_{T,\rho}^2.
\end{equation}

\subsubsection{Transfer estimate for $\ff$}
We prove for $\ff$ a result in the  spirit of Lemma~\ref{le:H}. Set
\[
\cT = \int f_1 f_2 \rho.
\]
First note that $\cT$ is well-defined and satisfies $|\cT|\leq \|f_1\|_{L^2}\|f_2\|_{L^2}
\leq C\varepsilon^{-\frac 12} \delta^2$ by~\eqref{bd:gz3}, \eqref{bd:gz4}.
Next, using~\eqref{eq:f}, we compute
\begin{align*}
\dot \cT
& = \int \dot f_1 f_2 \rho + \int f_1 \dot f_2 \rho\\
& = (1+c^2\gamma^2) \int f_2^2 \rho - \int (\partial_x f_1+c\gamma f_2)^2\rho 
 - 2 c \gamma \int f_1 f_2 \rho' +\frac 12 \int f_1^2 \rho'' - \int  f_1^2 P\rho\\
& \quad + \int [f_2 \Omega_1(\ff) + f_1 \Omega_2(\ff)]\rho +\int [f_2 \xi_1 +  f_1 \xi_2] \rho 
+\int f_1 (\TTT+ \XX \SSS) \rho.
\end{align*}
We observe
\[
\left|2 c \gamma \int f_1 f_2 \rho'\right|\leq \frac 12 \int f_2^2\rho + C \int f_1^2 \rho\]
and
\[
\left|\int f_1^2 \rho''\right| + \left| \int  f_1^2 P\rho\right| \leq C \int f_1^2 \rho.
\]
Next, using the notation~\eqref{def:Omega}, and integration by parts,
\[
\int [f_2 \Omega_1(\ff) + f_1 \Omega_2(\ff)]\rho
=-(\dot y-c) \gamma \int f_1 f_2 \rho'
+\dot c c\gamma^2 \int f_1 f_2 (x\rho)'
-\frac 12 \dot c\gamma \int f_1^2 \rho',
\]
and so by \eqref{bd:yc}, \eqref{bd:gz3}, \eqref{bd:gz4},
\[
  \int | f_2 \Omega_1(\ff) + f_1 \Omega_2(\ff) | \rho
\leq C\varepsilon^{-\frac 12} \delta^2 \|z_1 \rho\|_{L^2}^2.
\]
By the definitions of $\xi_1$ and $\xi_2$ in \eqref{def:xi1}, \eqref{def:xi2}, we have
\[
\|\xi_1\|_{L^2}+\|\xi_2\|_{L^2} \leq C \varepsilon^{-\frac 12}\delta \|z_1 \rho\|_{L^2}^2,
\]
and so
\[
\int |f_2 \xi_1 +  f_1 \xi_2| \rho \leq C \varepsilon^{-1}\delta^2 \|z_1 \rho\|_{L^2}^2.
\]
By the definition of $T_2$ \eqref{def:T1}, we have
\[
\int |\TTT|^2 \rho \leq C\varepsilon \int \left[(\partial_x f_1)^2+ f_1^2\right] \rho,
\]
and thus by the Cauchy-Schwarz inequality
\[
\int |\TTT f_1| \rho \leq C \varepsilon^{\frac 12} \int \left[(\partial_x f_1)^2+ f_1^2\right] \rho.
\]
Last, by \eqref{bd:SSS} and then~\eqref{eq:vir3}, we have
\begin{align*}
\int_0^T\int |f_1 (\XX \SSS)| \rho
&\leq C \int_0^T\int f_1^2 \rho + C \delta^2 \int_0^T\int \left[(\partial_x z_1)^2+z_2^2+z_1^2\right] \rho\\
&\leq C \int_0^T\int f_1^2 \rho + C \delta^2 \|z_1\|_{\rho,T}^2.
\end{align*}

Therefore, using $\varepsilon = A^{-\frac 12}$, we obtain
\begin{align*}
\int_0^T \int f_2^2 \rho 
&\leq CA^{\frac 14} \delta^2 + C \int_0^T \int \left[(\partial_x f_1+c\gamma f_2)^2 + f_1^2\right]\rho\\
&\quad + C A^{-\frac 14} \int  (\partial_x f_1)^2\rho + CA^{\frac 12} \delta^2  \|z_1\|_{T,\rho}^2.
\end{align*}
Using also
\[
 \int  (\partial_x f_1)^2 \rho \leq 2 \int \left[(\partial_x f_1+c\gamma f_2)^2 + (c \gamma f_2)^2\right] \rho,
\]
taking $A$ large enough and using \eqref{delta:A}, it follows that
\begin{equation}\label{transf:ff}
\begin{aligned}
\int_0^T \int \left[ f_2^2  +(\partial_x f_1)^2 \right] \rho 
&\leq C A^{\frac 14} \delta^2  + C \int_0^T \int \left[(\partial_x f_1+c\gamma f_2)^2 + f_1^2\right]\rho\\
&\quad + C A^{-\frac {15}2} \|z_1\|_{T,\rho}^2.
\end{aligned}
\end{equation}

\subsubsection{End of the proof of Proposition~\ref{pr:3}}\label{S:stepfin}
In view of \eqref{transf:ff}, our goal is now to estimate the quantity $\int_0^T\int \left[(\partial_x f_1+c\gamma f_2)^2 + f_1^2\right]\rho$ using the estimate~\eqref{eq:virialff11}.
We decompose
\begin{align*}
\int_0^T\int \left[(\partial_x f_1+c\gamma f_2)^2 + f_1^2\right] \rho
& \leq  C\rho^{\frac 12}(A)\int_0^T \int \left[(\partial_x f_1+c\gamma f_2)^2 + f_1^2\right] \rho^\frac 12\\
&\quad + C \int_0^T\int_{|x|<A} \left[(\partial_x f_1+c\gamma f_2)^2 + f_1^2\right] \rho  .
\end{align*}
For the first term in the right hand side, we recall the pointwise bounds 
(see~\eqref{on:fj})
\begin{align*}
&|f_1|\leq |j_1|+\left|\XX\left[ \frac {Y'}{Y} z_1\right]\right|,
\quad
|f_2|\leq |j_2|+\left|\XX\left[ \frac {Y'}{Y} z_2\right]\right|,\\
& |\partial_x f_1|\leq |\partial_x j_1|+\left|\partial_x \XX\left[ \frac {Y'}{Y} z_1\right]\right|.
\end{align*}
Then, by \eqref{on:Xepsloc} and \eqref{on:Xeps}, with $K=40$,
\[
\int \left|\XX\left[ \frac {Y'}{Y} z_1\right]\right|^2 \rho^\frac 12
\leq C \int \left| \XX\left[\frac {Y'}{Y}z_1 \sech\left(\frac \omega {40} x\right)\right] \right|^2
\leq C \int z_1^2 \rho^{\frac 12},
\]
and similarly,
\[
\int \left[ \left|\XX\left[ \frac {Y'}{Y} z_2\right]\right|^2 + 
\left|\partial_x \XX\left[ \frac {Y'}{Y} z_1\right]\right|^2 \right]\rho^{\frac 12}
\leq C  \int (z_2^2+z_1^2+(\partial_x z_1)^2)  \rho^{\frac 12}.
\]
Thus the estimates \eqref{eq:vir3}-\eqref{eq:vir4} of Proposition~\ref{pr:2} (with $K=40$)
and $\varepsilon=A^{-\frac 12}$ yield for $A$ large
\begin{align}
\rho^{\frac 12}(A)  \int_0^T \int \left[(\partial_x f_1+c\gamma f_2)^2 + f_1^2\right] \rho^\frac 12
&\leq C\varepsilon^{-1} \rho^{\frac 12}(A) \left( A \delta^2 + \|z_1\|_{T,\rho}^2\right)\nonumber\\
&\leq C A \delta^2 + C A^{-\frac 14} \|z_1\|_{T,\rho}^2.\label{T:AAAA}
\end{align}
Second, we claim
\begin{equation}\label{finish}\begin{aligned}
&\int_0^T\int_{|x|<A} \left[(\partial_x f_1 + c\gamma f_2)^2+f_1^2\right] \rho\\
&\qquad\leq C \int_0^T\left[ \int (\partial_x h_1 + c\gamma h_2)^2 +\int_{[x_1,x_2]} h_1^2\right]
+C\delta^2+ C A^{-\frac14}\|z_1\|_{T,\rho}^2.
\end{aligned}\end{equation}
Observe that combined with \eqref{eq:virialff11}, \eqref{transf:ff} and \eqref{T:AAAA},
this estimate completes the proof of the Proposition.
To prove~\eqref{finish}, we first use \eqref{pr:finir2} and~\eqref{bd:gz7bis}
\begin{equation*}
\int_{|x|<A} \left[(\partial_x f_1+c\gamma f_2)^2 + f_1^2\right] \rho \leq C 
 \int \left[ (\partial_x h_1 + c\gamma h_2)^2 + h_1^2\rho^{\frac 12}  \right].
\end{equation*}
Using Lemma~\ref{le:tech:h} 
with $K = 20$, $\sigma= c\gamma$ and $J=[x_1,x_2]$ to estimate the term $\int_0^T \int h_1^2 \rho^{\frac 12}$,
we obtain (since $A \ge A_1$)
\begin{align*}
&\int_0^T\int_{|x|<A} \left[(\partial_x f_1 + c\gamma f_2)^2+f_1^2\right] \rho\\
&\quad
\leq C\int_0^T \left[\int (\partial_x h_1 + c\gamma h_2)^2 +\int (\dot h_1 -h_2)^2 + \int_{[x_1,x_2]} h_1^2\right] + C\delta^2.
\end{align*}
Using $\dot h_1 - h_2 = \CHI \zeta_B \XX \NN_1$ ($\NN_1$ is defined in \eqref{def:N1})
and~\eqref{bd:yc}, \eqref{on:zetaK}, \eqref{on:Xeps}, \eqref{on:Xepsloc}, \eqref{bd:gz1}, 
and then $\varepsilon=A^{-\frac 12}$, $A\leq \delta^{-\frac 14}$ and \eqref{bd:z}, we have
\[
\|\dot h_1-h_2\|_{L^2} 
\leq C B \varepsilon^{-\frac 12} \|\zz\|_{H^1\times L^2}  \|z_1\rho\|_{L^2}^2
\leq C\delta^{\frac {15}{16}} \|z_1\rho\|_{L^2}^2
\leq CA^{-\frac{31}{4}} \|z_1\rho\|_{L^2}.
\]
Thus, \eqref{finish} is proved, which completes the proof of Proposition~\ref{pr:3}.

\subsection{Conclusion of the proof of asymptotic stability}\label{end2end}
We place ourselves in the context of Propositions~\ref{le:coer} and~\ref{pr:3}.
In particular, $A\geq A_2$ is to be chosen later and $B=B_1$. 
The parameter $\varepsilon>0$ is fixed as in~\eqref{epsilon:A}, and the parameter $\delta>0$ is chosen small enough
with the additional constraint~\eqref{delta:A}.
Set
\[
\cLL = \int \left[ (\partial_x z_1)^2+z_1^2+z_2^2 \right] \rho^2.
\]
First, we claim
\begin{equation}\label{eq:cl}
\int_0^{+\infty} \cLL(t) \ud t \leq C \delta^2.
\end{equation}
{\it{Proof of~\eqref{eq:cl}}}. Let $T>0$.
Combining estimate~\eqref{eq:coer} of Proposition~\ref{le:coer}
with estimate~\eqref{eq:pr3} of Proposition~\ref{pr:3}, we obtain
\[
\|z_1\|_{T,\rho}^2 
\leq CA \delta^2 + C A^{-\frac 14} \|z_1\|_{T,\rho}^2.
\]
Thus, for $A\geq A_3$, where 
\[A_3=\max(A_2; (2C)^4),\]
 it holds
\[
\|z_1\|_{T,\rho}^2 
\leq 2 C A\delta^2.
\]
Now, $A=A_3$ is fixed and we will not mention the dependence on $A$ anymore.
Note that $\varepsilon$ is also fixed by~\eqref{epsilon:A}.
Using~\eqref{zeta:rho} and \eqref{eq:vir3} with~$K=10$, we obtain the estimate
$\int_0^T \cLL(t) \ud t \leq C \delta^2$. Passing to the limit $T\to \infty$, we have proved~\eqref{eq:cl}.

Second, we prove that
\begin{equation}\label{eq:cv0}
\lim_{t\to +\infty} \cLL(t)=0.
\end{equation}
Using~\eqref{eq:z}, we have
\begin{align*}
\dot\cLL
&= 2 \int \left[ (\partial_x \dot z_1)(\partial_x z_1)+ \dot z_1 z_1+\dot z_2 z_2\right] \rho^2\\
&= 2 \int \left[ (\partial_x z_2)(\partial_x z_1)+z_2z_1+\Theta(\zz)z_2 
+(\partial_x \MM_1) (\partial_x z_1) + \MM_1 z_1 + \MM_2 z_2\right]\rho^2.
\end{align*}
By integration by parts,
\begin{align*}
\dot\cLL
&= 2 \int z_2 \left[z_1-(W'(H+z_1)-W'(H))\right] \rho^2
-2\int (\partial_x z_1) z_2 (\rho^2)'\\
&\quad + 2 (\dot y-c)\gamma \int (H''\partial_x z_1+H'z_1) \rho^2
-(\dot y-c)\gamma \int [(\partial_x z_1)^2 + z_1^2+z_2^2] (\rho^2)'\\
&\quad -2 \dot c c\gamma^2\int [(\Lambda H)'\partial_x z_1 + (\Lambda H)z_1]\rho^2
+2\dot c\gamma \int H'z_2\rho^2\\
&\quad + \dot c c\gamma \int [ (\partial_x z_1)^2 (\rho^2-x(\rho^2)') + z_1^2 (x\rho^2)'+z_2^2(x\rho^2)']
+2\dot c\gamma^2 \int (\partial_x z_1) z_2\rho^2.
\end{align*}
Using this identity, we deduce from the estimate~\eqref{bd:yc} and the decay properties of the functions~$\rho$, $H'$ and $H''$ that
\begin{equation}\label{eq:bof}
|\dot \cLL(t)|\leq C \cLL(t).
\end{equation}
From~\eqref{eq:cl}, there exists a sequence $t_n\uparrow \infty$ such that
$\lim_{n\to \infty} \cLL(t_n)=0$. Let $t\geq 0$. For $n$ large enough so that $t_n>t$, integrating~\eqref{eq:bof} on $[t,t_n]$
and then passing to the limit as $n\to \infty$, we obtain
\[
\cLL(t) \leq C \int_t^{\infty} \cLL(s) ds.
\]
Thus, $\lim_{t\to +\infty} \cLL(t)=0$ follows from~\eqref{eq:cl}.

Third, since $|\dot c|\leq C \cLL$ (see  \eqref{bd:yc}),
it follows from \eqref{eq:cl} that there exists $c_+\in (-1,1)$ such that
\begin{equation}\label{ccp}
\lim_{+ \infty} c = c_+.
\end{equation}
Moreover, by~\eqref{bd:yc} and~\eqref{bd:z} 
\[
\lim_{+ \infty} \dot y = c_+,
\quad
\gamma_0^2 |\dot y - c_0|+
\gamma_0^2 |c_+-c_0| \leq C \delta^2.
\]
By the change of variable~\eqref{z1z2}, estimate~\eqref{eq:cl} implies~\eqref{eq:int0T}.
Finally,~\eqref{eq:cv0}, \eqref{ccp} and \eqref{eq:Hclose}, \eqref{eq:nul}
combined give the conclusion of Theorem~\ref{th:2}.

\subsection{Perturbations of the nonlinearity}

Let $W$ satisfy the condition \eqref{on:W}, 
$\mathcal V$ be a neighborhood of $[\zeta_-,\zeta_+]$ and 
$\alpha_0>0$ be a small parameter to be chosen.
We consider perturbed potentials of the form
\begin{equation}\label{eq:W0W}
W_\alpha  =  (1+\alpha )W
\end{equation}
where the function $\alpha$ satisfies  
\begin{equation}\label{on:alpha}
\begin{cases}
\mbox{$\alpha:\mathcal V\to \R$ is of class $\mathcal C^3$},\\
\mbox{$\sup_{\mathcal V} |\alpha^{(k)}|\leq \alpha_0$ for $k=0,1,2,3$.}
\end{cases}
\end{equation}
It is straightforward to check that for $\alpha_0$ small, the potential
$W_\alpha$ also satisfies the condition~\eqref{on:W}, with the same zeros $\zeta_-$ and $\zeta_+$.
Moreover, the corresponding kink $H_\alpha$ defined in Lemma~\ref{pr:H} is close to $H$
(see \eqref{diff:H0H} below).

In contrast, even if the potential $W$ satisfies~\eqref{on:V},
the perturbed potential $W_\alpha$ may not satisfy \eqref{on:V} for arbitrarily small~$\alpha_0>0$.
We refer to Remark~\ref{rk:perturb} for an explicit example in the $P(\phi)_2$ theory.
However, we show that the result of Theorem~\ref{th:2} extends to such
potentials $W_\alpha$ for $\alpha_0$ small enough.

\begin{corollary}\label{cor:new}
Let $W$ satisfy \eqref{on:W}. Assume that the transformed potential $V$ verifies $V'\not\equiv 0$ on $(\zm,\zp)$ and~\eqref{on:V}.
There exists $\alpha_0>0$ such that if $\alpha$
verifies \eqref{on:alpha} then the kink $H_\alpha$ corresponding to the potential $W_\alpha$ defined in~\eqref{eq:W0W} is asymptotically stable.
\end{corollary}
\begin{remark}\label{rk:new}
Considering a multiplicative perturbation~\eqref{eq:W0W} of the potential $W$ simplifies the proof
since the wells of $W$ and $W_\alpha$ are the same.
However, using the change of variable described in~\S\ref{S:5.1}, one may also consider small perturbations of the potential $W$ with different wells.
\end{remark}
\begin{proof}
One can observe that the only place in the proof where the assumption~\eqref{on:V} is required is to obtain the lower bound~\eqref{bh1} on $Z_B$ (\S\ref{S:step3} of the proof of Proposition~\ref{pr:3}).
To prove Proposition~\ref{pr:3} in the case of the perturbed potential $W_\alpha$,
we consider the same choice of $x_0,x_1,x_2\in \R$ and  $C_0,B>0$ as for $W$, to define the localization functions of Section \ref{s:notvirial}.
Let $Z_B$ be defined in \eqref{def:ZB} and let
$Z_{\alpha,B}$ be defined similarly for the potential $W_\alpha$.
To obtain a lower bound on $Z_{\alpha,B}$, we write
\[
Z_{\alpha,B}\geq Z_B - |Z_{\alpha,B}-Z_B|.
\]
The goal is thus to bound the remainder term 
\[
C\alpha_0 \int_0^T \int |Z_{\alpha,B}-Z_B|
\]
by a fraction of the main term appearing in~\eqref{eq:virialff11}
\[
\int_0^T \left[\int (\partial_x h_1 + c \gamma h_2)^2+  \int_{[x_1,x_2]} h_1^2 \right].
\]

In order to do so, we start by the following estimates that follow directly from Lemma~\ref{pr:H}:
for  $k=1,2,3$, for all $x\in\R$,
\begin{equation}\label{on:H0H}
|H_\alpha(x)-\zpm|\leq C e^{\mp \frac 12 \omega x},\quad 
|H_\alpha^{(k)}(x)|\leq C e^{-\frac 12 \omega |x|};
\end{equation}
the coefficient of the exponentials in the above estimates are fixed to $\frac 12 \omega$ because
the values of $\omega_\pm$ defined in \eqref{def:omega} can be slightly different for $W_\alpha$.

Next,
we check that the kink $H_\alpha$ is close to the kink $H$ in the following sense:
for $k=0,1,2,3$, for all $x\in \R$,
\begin{equation}\label{diff:H0H}
|H_\alpha^{(k)}(x)-H^{(k)}(x)|\leq C \alpha_0  e^{-\frac 14 \omega |x|}.
\end{equation}
{\it{Proof of \eqref{diff:H0H}}}.
We define the functions $G$ and $G_\alpha$ corresponding respectively to $W$ and $W_\alpha$ as given by formula~\eqref{def:G}.
Since $x=G_\alpha(H_\alpha(x))=G(H(x))$, the following identity holds
\begin{equation*}
G(H_\alpha(x))-G_\alpha(H_\alpha(x))=G(H_\alpha(x))-G(H(x)).
\end{equation*}
On the one hand, we have
\begin{equation*}
G(H_\alpha(x))-G_\alpha(H_\alpha(x))
=\int_{\zeta_0}^{H_\alpha(x)}
\frac {\alpha(s)} {1+\sqrt{1+\alpha(s)}} \frac{\ud s}{\sqrt{2W_\alpha(s)}},
\end{equation*}
and so by \eqref{on:alpha} and $G_\alpha(H_\alpha(x))=x$,
\begin{equation*}
|G(H_\alpha(x))-G_\alpha(H_\alpha(x))| \leq  \alpha_0 |x|.
\end{equation*}
On the other hand, 
\[
G(H_\alpha(x)) - G(H(x)) =  \int_{H(x)}^{H_\alpha(x)} \frac {\ud s}{\sqrt{2W(s)}},
\]
and thus, using
\[
\sqrt{W(s)} \leq C (s-\zeta_-) (\zeta_+-s),
\]
for  $s\in(\zeta_-,\zeta_+)$, we obtain by \eqref{on:H0H}
\[
|G(H_\alpha(x)) - G(H(x))| \geq \frac 1 C  e^{\frac 12 \omega |x|} |H_\alpha(x)-H(x)|.
\]
Therefore, \eqref{diff:H0H} for $k=0$ is proved
(the coefficient of the exponentials in \eqref{diff:H0H} became $\frac 14 \omega$ to absorb the factor $|x|$).
For $k=1,2,3$, it suffices to observe
that 
\[
H_\alpha'-H' = \sqrt{2 W_\alpha(H_\alpha)} - \sqrt{2 W(H)},
\]
and to differentiate this relation to derive the estimate~\ref{diff:H0H}.

Next, we check that the transformed potential $V_\alpha$ corresponding to $W_\alpha$ is close
to the transformed potential $V$ of $W$.
From~\eqref{def:V}, we have the following formula
\[
V_\alpha = (1+\alpha) V + W \left(\frac{(\alpha')^2}{1+\alpha}-\alpha''\right)
\]
and thus setting $P_\alpha=V_\alpha(H_\alpha)$,
\[
P_\alpha'
= H_\alpha'V_\alpha'(H_\alpha)
=H_\alpha'V_\alpha'( H) + H_\alpha'R_1(H_\alpha),
\]
where
\[
R_1= \alpha V'+\alpha' V+W' \left(\frac{(\alpha')^2}{1+\alpha}-\alpha''\right)
+ W \left(-\frac{(\alpha')^3}{(1+\alpha)^2}+\frac{2\alpha'\alpha''}{1+\alpha}-\alpha'''\right).
\]
Using \eqref{on:alpha} and \eqref{diff:H0H}, we observe that
\[
|R_1|\leq C \alpha_0,\quad |V'(H_\alpha)-V_0'(H)|\leq C \alpha_0.
\]
Finally, we estimate
\begin{equation}\label{diff:P0P}
\begin{aligned}
|P_\alpha' -P' |
&\leq  |H_\alpha'- H'| |V'(H)| + H_\alpha' |V'(H_\alpha)-V'(H)| + H_\alpha'|R_1(H_\alpha)|\\
&\leq C \alpha_0 e^{-\frac 14 \omega|x|}.
\end{aligned}
\end{equation}

Therefore, we conclude that
\[
|Z_{\alpha,B}-Z_B|=
\left|\frac{\varphi_B P_\alpha'}{\zeta_B^2}-\frac{\varphi_B P'}{\zeta_B^2}\right|
\leq CBe^{\frac {\omega |x|} {B}} |P_\alpha'-P'|
\leq C\alpha_0 e^{-\frac 18\omega|x|}.
\]
Using Lemma~\ref{le:tech:h} with $K=8$ and $\sigma= c\gamma$, we estimate
\begin{align*}
&\int_0^T\int |Z_{\alpha,B}-Z_B| h_1^2 \\
&\qquad\leq C_3 \alpha_0 \int_0^T\int h_1^2 \frac{\phi_B}{2\zeta_B^2}|P_\alpha' -P' | \\
&\qquad \leq  C_3\alpha_0 \left(\int_0^T \int (\partial_x h_1 + c\gamma \dot h_1)^2
 +  \int_0^T \int_{[x_1,x_2]} h_1^2  + \sup_{[0,T]} \gamma(t) \norm{h_1(t)}_{L^2}^2 \right),
\end{align*}
for some constant $C_3$ independent of $c$.
As in \S\ref{S:stepfin} of the proof of Proposition~\ref{pr:3},
we have
\[
\int_0^T \int |\dot h_1 - h_2|^2 \leq
C A^{-\frac 14}  \|z_1 \|_{T,\rho}^2.
\]
Therefore, for $\alpha_0 \le \frac{3}{4} C_3^{-1}$, the remaining terms in \eqref{eq:tobeint} can be bounded as before. We obtain \eqref{eq:virialff11} and from there prove Proposition~\ref{pr:3}
for the potential $W_\alpha$.
\end{proof}

\section{Applications}\label{S:applications}
In this final section, we illustrate the applicability of Theorem~\ref{th:2} by discussing some concrete examples of scalar field models for which the sufficient condition~\eqref{on:V} of asymptotic stability is indeed verified.
Recall that~\eqref{on:V} is a simple condition set on the transformed
potential $V$ defined by~\eqref{def:V}.
In this section, we provide an analytic verification of this criterion for various models
that appear in the Physics literature, notably~\cite{Campbell,Lohe}.
For more complex models, it is possible to check the condition numerically by plotting the associated $V$.

The first condition in~\eqref{on:V} is that the transformed potential $V$ is not constant.
In \S\ref{S:5.2}, we consider the threshold case where $V$ is constant and prove that it is equivalent to 
$W=W_\SG$ (the sine-Gordon potential defined in~\eqref{W:sineG}) up to invariances.
Invariances in the general setting~\eqref{on:W} are discussed in \S\ref{S:5.1}.

Second, in the case of a non constant transformed potential $V$, we observe that the sufficient condition \eqref{on:V} is satisfied either when $V'(\phi)$ has a constant sign $+$ or $-$ in the range $(\zeta_-,\zeta_+)$ of the kink, or when it changes sign only once at some $\zeta_0\in (\zeta_-, \zeta_+)$ and then it holds both $V'\geq 0$ on $(\zeta_-,\zeta_0)$ and $V'\leq 0$ on $(\zeta_0,\zeta_+)$. We will see in \S\ref{S:5.3}-\ref{S:5.4} that these different possibilities do occur in explicit examples of potential.

\subsection{Change of variables}\label{S:5.1}
We summarize how invariances under dilations, translations, scaling and reflection allow to fix some properties of the kink and the potential.
The outcome of this discussion is that for any potential $W$ satisfying \eqref{on:W},
we can restrict ourselves to the case where
\begin{equation}\label{eq:reduction}
\zeta_-=\zeta_1,\quad \zeta_+=\zeta_2,\quad 1=W''(\zeta_1)\leq W''(\zeta_2),
\end{equation}
for any given values $\zeta_1<\zeta_2$ (typical cases are $(\zeta_1,\zeta_2)=(0,1)$ or $(0,2\pi)$).
However, if $W''(\zeta_-)< W''(\zeta_+)$ then one cannot use invariances to reduce to $W''(\zeta_-)= W''(\zeta_+)$.

First, by possibly changing $W(\phi)$ into $W(-\phi)$, we can assume without loss of generality that
$0<W''(\zeta_-)\leq W''(\zeta_+)$ (of course, this changes the values of $\zeta_\pm$, but this will be settled right away).
Second, let $\phi$ satisfy \eqref{eq:W}. 
For any $a\in \R$ and $\lambda,\mu>0$, setting 
\[
\tilde \phi(t,x)=   \lambda(a+ \phi(\mu t, \mu x )),
\]
the equation of $\tilde\phi$ is
\[
\partial_t^2 \tilde \phi - \partial_x^2 \tilde \phi
= \lambda \mu^2 W'(\lambda^{-1} \tilde \phi - a) = \tilde W'(\tilde \phi),
\]
where the potential $\tilde W$ is defined by
\[
\tilde W(\tilde \phi)=  \lambda^2 \mu^2 W(\lambda^{-1} \tilde \phi - a).
\]
Choosing $\lambda>0$ and $a\in \R$ such that
\[
\lambda^{-1} \zeta_1 -a = \zeta_- \quad \mbox{and}\quad \lambda^{-1} \zeta_2 -a = \zeta_+,
\]
we have prescribed the zeros of $\tilde W$ to $\zeta_1$ and $\zeta_2$.
Moreover, since
\[
\tilde W''(\tilde \phi)=  \mu^2 W''(\lambda^{-1} \tilde \phi -a )
\]
by adjusting $\mu>0$, we can fix $\tilde W''(\zeta_1)$ to any given positive value (for example $1$ as \eqref{eq:reduction}). 
By the first reduction, we have $\tilde W''(\zeta_2)\geq \tilde W''(\zeta_1)$, but we have no other free parameter to change $\tilde W''(\zeta_2)$.

\subsection{Case of a constant transformed potential}\label{S:5.2}

In this section, we prove that if the transformed potential $V$ is constant then 
$W$ is the sine-Gordon potential given in \eqref{W:sineG}, up to the invariances.
Using \S\ref{S:5.1}, we restrict ourselves to the case where $(\zeta_-,\zeta_+)=(0,2\pi)$
and we assume $W''(0)=1$.

Recall that the transformed potential $V$ has the form \eqref{def:V}
\[
V=  -W\left(\frac{W'}{W}\right)'.
\]
Note from~\eqref{eq:L2.1} that $V(0) = W''(0)=1$.
Thus, assuming that $V$ is constant on $[0,2\pi]$, we have $V\equiv 1$ and the potential $W$ satisfies on $(0,2\pi)$
\[
\left(\frac{W'}{W}\right)'=-\frac 1W.
\]
Multiplying by $\frac{W'}W$ and integrating, one finds on $(0,2\pi)$,
\[
\frac 12 \left(\frac {W'}{W}\right)^2 = \frac 1 W + C,
\]
where $C$ is a constant. Note that the constant $C$ has to be negative since otherwise
$W'$ does not vanish on $(0,2\pi)$ and so $W(2\pi)=0$ is impossible.
We set $C=-\frac{a^2}2$ where $a>0$, so that the equation satisfied by $W$ on $[0,2\pi]$ becomes
\[
(W')^2 = 2 W - a^2 W^2  = \frac 1{a^2} \left[ 1  - \left(1-a^2 W\right)^2\right],
\]
and setting $U(\phi)=1-a^2 W(\phi/a)$, we find $(U')^2+U^2=1$.
Since $U(0)=1$, $U'(0)=0$ and $U''(0)=-1$, we obtain $U(\phi)=\cos \phi$, and thus
$W(\phi) = a^{-2} (1 - \cos (a \phi))$.
Since $\zeta_+=2\pi$, we have $a=1$ and so $W(\phi)=1-\cos \phi$, the sine-Gordon potential~\eqref{W:sineG}.
Since the kinks of the sine-Gordon equation are not asymptotically stable (in the sense of Theorem~\ref{th:2}, see references in the Introduction) this shows that the first condition
$V'\not \equiv 0$ in \eqref{on:V} is necessary.

\subsection{The $P(\phi)_2$ theory for low degrees}\label{S:5.3}

Following \cite{Lohe}  we consider two families of models with finitely many wells. 
Let $n\geq 1$ and $0\leq m_1< m_2<\cdots<m_n$ be  $n$ distinct numbers  representing  the locations of the wells
on the half line $(0,\infty)$. Set
\begin{align}
W_{4n}(\phi;m_1, \dots, m_n)&=\prod_{k=1}^n (\phi^2-m_k^2)^2,\label{def: w4n}\\
W_{4n+2}(\phi;m_1, \dots, m_n)&=\phi^2\prod_{k=1}^n (\phi^2-m_k^2)^2.\label{def: w4n2}
\end{align}
These two potentials represent respectively the $\phi^{4n}$ and $\phi^{4n+2}$ models in generalized form.
The usual $\phi^{4n}$ and $\phi^{4n+2}$ models correspond to specific choices of the wells, see~\cite{Lohe},
~\S\ref{S:5.3bis} and to a suitable normalization to approximate the sine-Gordon theory in the limit
$n\to \infty$. We discuss in this section specific results for $\phi^4$, $\phi^6$, $\phi^8$ and $\phi^{10}$, possibly depending on the location of the wells.

\subsubsection{Failure of the criterion for the $\phi^4$ model}\label{S:5.3.1}
For $n=1$ and $m_1=1$ in \eqref{def: w4n} (note that by \S\ref{S:5.1}, we may restrict ourselves to $m_1=1$ without loss of generality),
we obtain the $\phi^4$ model, with potential \eqref{W:phi4}.
The static kink of the $\phi^4$ model, corresponding to the potential~\eqref{W:phi4} is given explicitely by
$\HH_4=(H_4,0)$
where $H_{4}(x)=\tanh\left(\sqrt{2} {x}\right)$.
By direct computation, the corresponding transformed potential $V_4$ defined by \eqref{def:V} writes
\[V_4(\phi)=4(\phi^2+1) \quad \mbox{and so}\quad V_4'(\phi)=8\phi.\]
We observe that the criterion (\ref{on:V}) is not satisfied on $[-1,1]$, while the unique (increasing) kink for this model connects $\zeta_-=-1$ to $\zeta_+=1$.
This is related to the fact that the $\phi^4$ kink has an internal mode in addition to the zero eigenfunction~$H'$.

In particular, recall the result from \cite{KMM}.

\begin{theorem}[The static $\phi^4$ kink under odd perturbation~\cite{KMM}]\label{th:phi4}
There exists $\delta>0$ such that for any odd $\pp^{in}$ with
$\|\pp^{in}-\HH_4 \|_{H^1\times L^2}\leq \delta$,
the solution $\pp$ of the $\phi^4$ model with $\pp(0)=\pp^{in}$
satisfies, for any bounded interval $I$,
\[\lim_{t \to \pm\infty} \|\pp(t)-\HH_4\|_{(H^1\times L^2)(I)} =0.
\]
\end{theorem}
Recall that the $\phi^4$ model has an even resonance, which justifies the restriction to odd initial perturbation of the odd static kink. The main difficulty in~\cite{KMM} was to deal with the odd internal mode, which yields a weak decay rate of the solution.

\subsubsection{Asymptotic stability for all kinks of the $\phi^6$ model}
The $\phi^6$ model corresponds to \eqref{W:phi6}, which is the potential in \eqref{def: w4n2} with $n=1$ and $m=1$. By  \S\ref{S:5.1}, we may restrict ourselves to the case $m_1=1$. By direct computations, we obtain the corresponding transformed potential
\[
V_6(\phi)=2(\phi^2-1)^2+4\phi^2(\phi^2+1)
=6\phi^4 +2 \quad \mbox{and so}\quad  V_6'(\phi)=24 \phi^3.
\]
Since $(1-\phi) V'_6(\phi)\leq 0$ on $[0,1]$, the static kink of the $\phi^6$ model which connects $0$ to $1$ and is explicitly given by
\[
H_6(x)=\left(\frac{1+\tanh ({\sqrt 2}x)}{2}\right)^{1/2}
\]
satisfies the condition \eqref{on:V} and
is thus asymptotically stable. This is also the case of all moving kinks obtained from $H_6$ by the Lorentz boost.

\begin{theorem}[$\phi^6$ model]\label{th:phi6}
The kink $H_6$ of the $\phi^6$ model is asymptotically stable.
\end{theorem}

\subsubsection{The generalized $\phi^{8}$ model}\label{subsec v8}
We consider the generalized $\phi^8$ model with the potential
\begin{equation}\label{W:phi8}
W_8(\phi; m)=(\phi^2-1)^2(\phi^2-m^2)^2,
\end{equation}
where we assume $m_1=1$ and $m_2=m>1$ without loss of generality by \S\ref{S:5.1}.
Note that this model has four potential wells at $\pm 1$ and $\pm m$ and it has three static kinks: the unique odd kink 
$H_{8}{\{-1,1\}}$ connecting $-1$ to $+1$ and two non-symmetric kinks $H_8{\{-m,-1\}}$ and $H_8{\{1,m\}}$ 
(images of each other by the symmetry $x\mapsto -x$, $\phi\mapsto -\phi$), connecting respectively $-m$ to $-1$ and $1$ to $m$. 

We study the function $V_8$ and its derivative $V_8'$ in terms of the parameter~$m>1$. 
By direct computations
\begin{equation*}
V_8(\phi;m)=4\left[(\phi^2+1)(\phi^2-m^2)^2+(\phi^2+m^2)(\phi^2-1)^2\right]
\end{equation*}
and so
\begin{equation*}
V'_8(\phi; m)=8\phi U_8(\phi;m),
\end{equation*}
where
\begin{align*}
U_8(\phi;m)
&= (\phi^2-m^2)^2+2(\phi^2-m^2)(\phi^2+1)+2(\phi^2-1)(\phi^2+m^2)+(\phi^2-1)^2\\
&= m^4 - 2 m^2 (\phi^2+2) + 6\phi^4-2\phi^2+1.
\end{align*}
In particular, sign changes for $V'_8$ are related to the zeros of $U(\phi;m)$.
We introduce the functions
\begin{align*}
M^+_8(\phi) & = \sqrt{\phi^2 + 2 +\sqrt{-5\phi^4+ 6\phi^2+3}} ,\quad\mbox{for}\quad |\phi| \leq  \sqrt{\frac{3+2\sqrt{6}}{5}}\\
M^-_8(\phi) & = \sqrt{\phi^2 + 2 -\sqrt{-5\phi^4+ 6\phi^2+3}}, \quad\mbox{for}\quad \frac 13 \sqrt{\frac 72}\leq |\phi| \leq \sqrt{\frac{3+2\sqrt{6}}{5}}.
\end{align*}
We will use Figure \ref{figure phi8} to guide us through different cases. 
The curves correspond to portions of the graphs of $\phi  \mapsto M^\pm_8(\phi)$.
\begin{figure}[ht]
\vspace*{-20pt}
\scalebox{0.45}
{\includegraphics{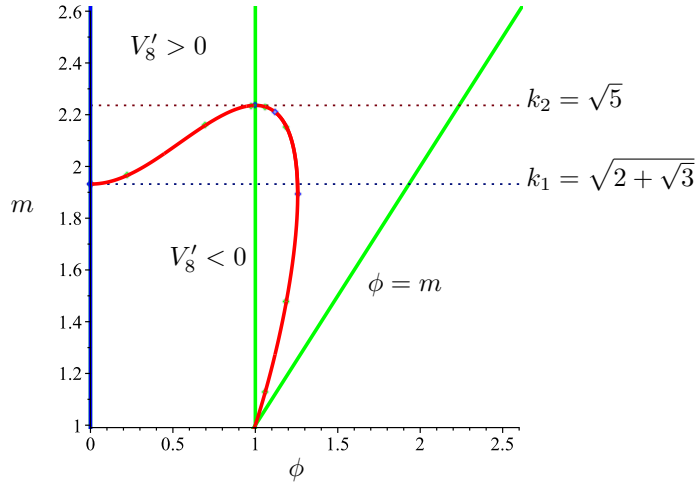}}
\put(-70,290){$k_2=\sqrt{5}$}
\put(-70,260){$k_1=\sqrt{2+\sqrt{3}}$}
\put(-130,220){$\phi=m$}
\put(-205,230){$V'_8<0$}
\put(-160,150){$\phi$}
\put(-220,310){$V'_8>0$}
\put(-265,250){$m$}
\vspace*{-150pt}
\caption{The zero level set of $V'_8(\phi;m)$ for $\phi \geq 0$ and $m\geq 1$.}
\label{figure phi8}
\end{figure}

We check the criterion \eqref{on:V} for the two kinks $H_8\{-1,1\}$ and $H_8\{1,m\}$, positive results for these kinks also holds for their Lorentz boosted versions. We see that there are $2$ threshold values $k_1=\sqrt{2+\sqrt{3}}$ and $k_2=\sqrt{5}$ of the parameter $m$ (vertical axis) at which changes for the criteron~\eqref{on:V} take place.
\begin{itemize}
\item[(i)] Case $1<m\leq k_1$: 
\begin{itemize}
\item[(a)] Since $\phi V'_8(\phi;m)=8\phi^2 U_8(\phi;m)\leq 0$ for all $\phi\in [-1,1]$, the criterion \eqref{on:V} is satisfied on $(-1,1)$ and so the kink $H_8\{-1, 1\}$ is asymptotically stable by Theorem~\ref{th:2}.
\item[(b)] For the kink $H\{1,m\}$ the criterion is inconclusive since in the range $(1,m)$ of this kink, the function $V'_8(\phi;m)$ is negative and then positive after changing sign on the red curve.
\end{itemize}
\item[(ii)] Case $k_1<m<k_2$: the criterion is inconclusive for both kinks.
Indeed, $\phi\in (-1,1)\mapsto V'_8(\phi;m)$ changes sign twice while $\phi\in (1,m)\mapsto V_8'(\phi;m)$
is negative and then positive.
\item[(iii)] Case $m\geq k_2$:
\begin{itemize}
\item[(a)] For $H\{-1,1\}$ the criterion is inconclusive since $\phi\in (-1,1)\mapsto V'_8(\phi;m)$ is negative and then positive.
\item[(b)] Since  $(\phi-1)V'_8(\phi;m)\leq 0$ for any  $\phi\in [1,m]$, the criterion~\eqref{on:V} is satisfied on $(1,m)$ and so the kink $H\{1, m\}$ is asymptotically stable.  By symmetry, the kink $H\{-m, -1\}$ is also asymptotically stable in this case.
\end{itemize}
\end{itemize}

\begin{theorem}[Generalized $\phi^8$ model]\label{th:phi8}
For  the $\phi^8$ model with potential~\eqref{W:phi8}, 
\begin{enumerate}
\item If $1< m\leq \sqrt{2+\sqrt{3}}$  the kink $H_8\{-1,1\}$ is asymptotically stable.
\item If $m\geq \sqrt{5}$  the kink $H_8\{1,m\}$ is asymptotically stable.
\end{enumerate}
\end{theorem}

\begin{remark}
After change of variable, the $\phi^8$ model considered in \cite{Lohe} (see also \S\ref{S:5.3.6})
corresponds to $m=3>\sqrt{5}$ and thus the kink $H_8\{1,3\}$ is asymptotically stable. 
\end{remark}
\begin{remark}\label{rk:perturb}
There exists $\tau>0$ small such that for any $m\in (k_1,k_1+\tau)$, the condition \eqref{on:V} for
$V$ does not holds on $(-1,1)$, however, by Corollary~\ref{cor:new}, the kink $H\{-1,1\}$ is asymptotically stable. Indeed, setting
\[
\alpha(\phi) = \frac{W_8(\phi;m)}{W_8(\phi,k_1)}-1
=\frac{(\phi^2-m^2)^2}{(\phi^2-k_1^2)^2}-1,
\]
we see that for $m\in (k_1,k_1+\tau)$ with $\tau>0$ small, for $k=0,1,2,3$, and any $\phi \in (-\frac 32,\frac 32)$,
\[
|\alpha^{(k)}(\phi)| \leq C |m-k_1| \leq C \tau,
\]
 which is sufficient to apply Corollary~\ref{cor:new}.
\end{remark}

\subsubsection{The generalized $\phi^{10}$ model}
The potential of the generalized $\phi^{10}$ model is 
\begin{equation}\label{W:phi10}
W_{10}(\phi;m)=\phi^2(\phi^2-1)^2(\phi^2-m^2)^2
\end{equation}
where $m>1$ is a parameter. For this model there are two types of kinks $H_{10}\{0,1\}$ and $H_{10}\{1,m\}$, and their symmetric counterparts. 

The transformed potential $V_{10}$ and then $V'_{10}$ are computed
\begin{align*}
V_{10}(\phi;m) &= 2(\phi^2-1)^2(\phi^2-m^2)^2
\\&\quad + 4 \phi^2 \left[ (\phi^2+1)(\phi^2-m^2)^2+(\phi^2-1)^2(\phi^2+m^2)\right]\\
V'_{10}(\phi;m)&= 24 \phi^3 U_{10}(\phi;m),
\end{align*}
where
\begin{equation*}
U_{10}(\phi;m)  =m^4 - 2  \left(\phi^2+\frac 23\right) m^2 
+\frac {10}3\phi^4 - 2\phi^2 +1.
\end{equation*}
We introduce
\begin{align*}
&M_{10}^+ (\phi)  = \sqrt{ \phi^2 + \frac 23 + \sqrt{-\frac 73 \phi^4 +\frac {10}3 \phi^2 -\frac 59}},\\
&M_{10}^- (\phi) = \sqrt{ \phi^2 + \frac 23 - \sqrt{-\frac 73 \phi^4 +\frac {10}3 \phi^2 -\frac 59}},\\
&\mbox{for} \quad\sqrt{\frac 57 - \sqrt{\frac{40}{147}}}= \phi_1  \leq \phi \leq \phi_2=\sqrt{\frac 57 +\sqrt{\frac{40}{147}}}.
\end{align*}
As in the previous case we use 
Figure \ref{figure phi10}. 
\begin{figure}[ht]
\vspace*{-25pt}
\scalebox{0.5}
{\includegraphics{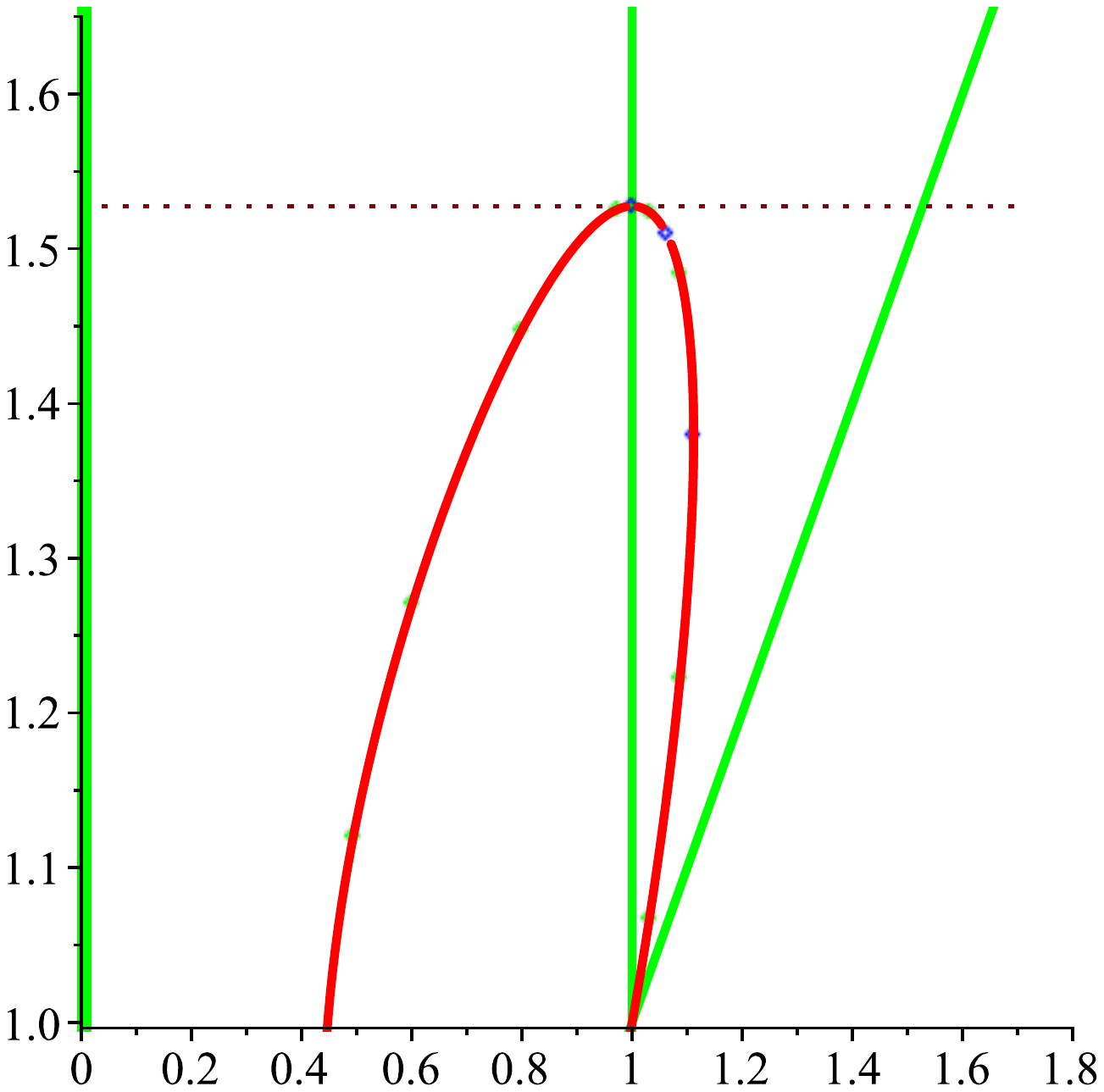}}
\put(-92,333){$k_1=\frac{\sqrt{21}}{3}$}
\put(-120,290){$\phi=m$}
\put(-160,175){$\phi$}
\put(-280,290){$m$}
\put(-200,230){$V'_{10}<0$}
\put(-200,350){$V'_{10}>0$}
\vspace*{-175pt}
\caption{The zero level set of $V'_{10}(\phi;m)$ for $\phi \geq 0$ and $m\geq 1$.}
\label{figure phi10}
\end{figure}
\begin{itemize}
\item[(i)] Case $1<m<k_1=\frac{\sqrt{21}}{3}$:
\begin{itemize}
\item[(a)] 
On $(0,1)$, the function $\phi\mapsto V'_{10}(\phi;m)$
is positive and then negative. The condition \eqref{on:V} holds on $(0,1)$ and so the kink $H_{10}\{0, 1\}$ is asymptotically stable.
\item[(b)] On $(1,m)$, the function $\phi\mapsto V'_{10}(\phi;m)$ 
is negative and then positive. Thus the criterion is inconclusive for the kink $H_{10}\{1,m\}$.
\end{itemize}
\item[(ii)] Case $m\geq k_1$:  $V'_{10}(\phi;m)\geq 0$ for any $\phi\geq 0$, and so both kinks $H_{10}\{0,1\}$ and $H_{10}\{1,m\}$ are asymptotically in this case stable.
\end{itemize}

\begin{theorem}[Generalized $\phi^{10}$ model]\label{th:phi10}
For the $\phi^{10}$ model with potential~\eqref{W:phi10}, 
\begin{enumerate}
\item If $m>1$ then the kink $H_{10}\{0,1\}$ is asymptotically stable.
\item If $m\geq \frac{\sqrt{21}}{3}$ then the kink $H_{10}\{1,m\}$ is asymptotically stable.
\end{enumerate}
\end{theorem}

\begin{remark}
After change of variable, the $\phi^{10}$ model considered in \cite{Lohe} (see also \S\ref{S:5.3.5}) corresponds to $m=2>\frac{\sqrt{21}}3$ for which both kinks are asymptotically stable.
\end{remark}

\subsection{Maximal number of asymptotically stable kinks in the $P(\phi)_2$ theory}
By the previous section, any kink of the generalized $\phi^8$ and $\phi^{10}$ models \eqref{W:phi10}, \eqref{W:phi10}
except the odd kink of the $\phi^8$ model, is asymptotically stable  for $m$ large enough
(see Theorems~\ref{th:phi8}, \ref{th:phi10}).
Now we prove a similar result for any $n\geq 2$: for $n$ wells sufficiently far away, any kink of the  generalized  $\phi^{4n}$ and $\phi^{4n+2}$ models, except possibly the odd kink of the $\phi^{4n}$ model, is asymptotically stable.
Recall that without loss of generality, one may assume that $m_1=1$ (see \S\ref{S:5.1}).

\begin{theorem}\label{th:lambdan}
There exists a sequence $(\lambda_n)_n$ with $\lambda_n>1$ such that 
for any $n\geq 2$, and  $(m_1,\ldots,m_n)$ satisfying
\begin{equation}\label{on:mm}
m_{j+1} \geq \lambda_j m_j>0\quad \mbox{for all $j\in \{1,\ldots,n-1\}$,}
\end{equation}
the following hold.
\begin{enumerate}
\item
Any kink of the $\phi^{4n}$ model with potential (\ref{def: w4n}), except possibly the unique odd kink, is asymptotically stable.
\item
Any kink of the $\phi^{4n+2}$ model with potential (\ref{def: w4n2}) is asymptotically stable.\end{enumerate}
\end{theorem}
Note that the condition on the wells is invariant by the change of variable~\S\ref{S:5.1}.

\begin{proof}
Let $W_{4n}(\phi;\mathbf{m})$ and $W_{4n+2}(\phi;\mathbf{m})$ where $\mathbf{m}=(m_1, \dots, m_n)$ 
be the potentials defined in (\ref{def: w4n}) and (\ref{def: w4n2}).
We start by proving (1).
By direct computations, we have
\begin{align}
V_{4n}(\phi; \mathbf{m})&=4\sum_{k=1}^n(\phi^2+m_k^2)\prod_{j\neq k}^n(\phi^2-m_j^2)^2,\label{def v4n}\\
\partial_\phi V_{4n}(\phi; \mathbf{m})&=8\phi U_{4n}(\phi;\mathbf m),\label{dv4n}
\end{align}
where
\[
U_{4n}(\phi;\mathbf m) =\sum_{k=1}^n \prod_{j\neq k}^n (\phi^2-m_j^2)^2+2\sum_{k=1,j\neq k}^n(\phi^2+m_k^2)(\phi^2-m_j^2)\prod_{l\neq k, l\neq j}^n(\phi^2-m_l^2)^2.
\]
Set
\begin{equation}\label{def:R4n}
R_{4n}(\phi;\mathbf m)= \sum_{k=1}^n \prod_{j\neq k}^n (\phi^2-m_j^2)^2.
\end{equation}
Note that $R_{4n}(\phi;\mathbf m) >0$ as soon as $m_j < m_{j+1}$ for some $j\in \{1,\ldots,n-1\}$.
We will prove the following statement by mathematical induction on $n\geq 2$:
there exist constants $c_n$ depending only on $n$ and $\lambda_1,\ldots,\lambda_{n-1}$ with $\lambda_j>1$ ,
such that if \eqref{on:mm} holds then, for any $\phi \geq 0$
\begin{equation}
\label{ind 2}
U_{4n}(\phi;\mathbf m)\geq c_n R_{4n}(\phi;\mathbf m).
\end{equation}
Suppose that this statement is true.
By \eqref{dv4n}, it implies that $\partial_\phi V_{4n}(\phi;\mathbf{m})>0$ for all $\phi\in (0,\infty)$, 
and in particular, the condition \eqref{on:V} is satisfied for the kink connecting the wells
$m_j$ and $m_{j+1}$, for any $j\in \{1,\ldots,n-1\}$. It also implies that $0$ is the minimum of $V_{4n}$, so condition \eqref{on:V} is not satisfied for the odd kink connecting $-m_1$ to $+m_1$.

First, we check that this property is satisfied for $n=2$. Taking $m_1=1$ and $m_2=m>1$ without loss of generality, 
we recall from \S\ref{subsec v8}
\begin{equation*}
6 U_{8}(\phi;\mathbf m)
 =  [6\phi^2 - (1+m^2)]^2 + (m^2-5)(5m^2-1).
\end{equation*}
Thus, fixing any $\lambda_1>\sqrt{5}$ (which is consistent with 
(2) of Theorem~\ref{th:phi8}), the property $U_{8}(\phi;\mathbf m)\geq c_2 \left((\phi^2-1)^2+ (\phi^2-m^2)^2\right)$  holds true for any $m\geq \lambda_1$, with some $c_2>0$ depending on the choice for $\lambda_1$.

Now we suppose that the property holds true up to a given $n\geq 2$ and we consider the $\phi^{4n}$ model with a configuration of wells $\mathbf m=(m_1,\ldots,m_n)$ satisfying the condition \eqref{on:mm} and thus~\eqref{ind 2}.
We add another well to this model, located at $m_{n+1}=\epsilon>m_n$, which means that we consider the potential
$W_{4(n+1)}(\phi; \mathbf{m},\epsilon)$. 
We check that the corresponding transformed potential writes
\begin{align*}
V_{4(n+1)}(\phi; \mathbf{m},\epsilon)
&=V_{4n}(\phi;\mathbf m)(\phi^2-\epsilon^2)^2+4W_{4n}(\phi;\mathbf m)(\phi^2+\epsilon^2)
\end{align*}
where we have used \eqref{def v4n}. 
Differentiating with respect to $\phi$, using the notation~\eqref{dv4n} and the induction hypothesis, we obtain
\begin{align*}
    {2} U_{4(n+1)}(\phi;\mathbf m,\epsilon)
    &=
    (\phi^2 - \epsilon^2)V_{4n}(\phi;\mathbf m) +{2}(\phi^2-\epsilon^2)^2 U_{4n}(\phi;\mathbf m) \\
    & \quad +2 W_{4n}(\phi;\mathbf m) +  (\phi^2 +\epsilon^2)\frac{\partial_\phi  W_{4n}(\phi;\mathbf m)}{\phi} \\
    & \ge  2 c_n (\phi^2 - \epsilon^2)^2R_{4n}(\phi;\mathbf m) +    2 W_{4n}(\phi;\mathbf m) \\
    &\quad +   (\phi^2 - \epsilon^2)V_{4n}(\phi;\mathbf m) +   (\phi^2 +\epsilon^2)\frac{\partial_\phi  W_{4n}(\phi;\mathbf m)}{\phi}.
\end{align*} 
For $\phi^2 \ge \epsilon^2$, both terms on the second line are non-negative and since
\[
R_{4(n+1)}(\phi;\mathbf m, \epsilon) = (\phi^2-\epsilon^2)R_{4n}(\phi;\mathbf m) + W_{4n}(\phi;\mathbf m),
\]
the property is satisfied for $\phi^2 \ge \epsilon^2$.

We now restrict our attention to the case $\phi^2 \in [0,\epsilon^2)$.
We bound
\begin{align*}
    \frac{\partial_\phi  W_{4n}(\phi;\mathbf m)}{\phi}
    = 4 \sum_{k=1}^n (\phi^2 - m_k^2)\prod_{j\neq k}^n (\phi^2-m_j^2)^2 
    \ge 4 (\phi^2 - m_n^2) R_{4n}(\phi;\mathbf m)
\end{align*}
and 
\begin{align*}\label{eq:upper_V}
    V_{4n}(\phi;\mathbf m) = 4 \sum_{k=1}^n (\phi^2 + m_k^2)\prod_{j\neq k}^n (\phi^2-m_j^2)^2 
    \le 4 (\phi^2 + m_n^2) R_{4n}(\phi;\mathbf m).
\end{align*}
Therefore
\begin{align*} 
 {2}   U_{4(n+1)}(\phi;\mathbf m,\epsilon)
    \ge \min\left(2, {c_n}\right)  R_{4(n+1)}(\phi;\mathbf m,\epsilon)
    +  P_{4n}(\phi; m_n,\epsilon)R_{4n}(\phi;\mathbf m) ,
    \end{align*} 
where we have defined 
\[
 P_{4n}(\phi ;m_n, \epsilon) = {c_n}(\phi^2-\epsilon^2)^2 + 4(\phi^2 -\epsilon^2)(\phi^2 + m_n^2) + 4 (\phi^2 + \epsilon^2)(\phi^2 -m_n^2).
\]
Expanding and completing the square gives
\begin{align*}
  P_{4n}(\phi;m_n,\epsilon) &=  {c_n}(\phi^2 - \epsilon^2)^2 + 8 (\phi^4 -\epsilon^2 m_n^2)\\
    &= (c_n + 8 )\left(\phi^2 - \frac{c_n \epsilon^2}{ c_n+8}\right)^2 - \frac{c_n^2 \epsilon^4}{c_n+8} + \epsilon^4 c_n - 8 \epsilon^2 m_n^2\\
    &\ge 8 \epsilon^4 \left(\frac{ c_n}{c_n + 8} -  \frac{m_n^2}{\epsilon^2} \right).
\end{align*}
Thus, if $m_{n+1}=\epsilon \ge \lambda_n m_n$ with $\lambda_n:= \sqrt{\frac{c_n + 8}{c_n}}>1$, then
\[
 P_{4n}(\phi ; m_n,\epsilon) \ge 0 \text{ for all } \phi \in \R,
\]
and we conclude that, with $c_{n+1} = { \min (1, c_n/2)}$,
\begin{align*}
    U_{4(n+1)}(\phi;\mathbf m,\epsilon) \ge c_{n+1} R_{4(n+1)}(\phi;\mathbf m,\epsilon).
    \end{align*}
To prove (2), observe that 
$$
V_{4n+2}(\phi;\mathbf m) = 2 W_{4n}(\phi;\mathbf m) + \phi^2 W_{4n}(\phi;\mathbf m)
$$
and by using the explicit expression \eqref{def v4n},
$$
\partial_\phi V_{4n+2}(\phi;\mathbf m) =8 \phi^3 \left( U_{4n}(\phi;\mathbf m) +  R_{4n}(\phi;\mathbf m)\right),
$$
which is non-negative if $\mathbf{m}$ satisfies \eqref{on:mm}.
\end{proof}

\subsection{The $P(\phi)_2$ theory as approximation of the sine-Gordon theory}\label{S:5.3bis}
Following~\cite{Lohe}, we recall that
the potentials $W_{4n}$ and $W_{4n+2}$ with suitable choices of wells $\{m_k\}_{k=1,\ldots,n}$
and after change of variables are approximations of the sine-Gordon potential for $n$ large.
We study the asymptotic stability of the kinks of these models in the sine-Gordon limit.

\subsubsection{The $\tilde W_{4n+2}$ potentials as approximation of sine-Gordon potential}\label{S:5.3.5}
Recall the formula
\begin{equation*}
\frac{\sin(\pi \phi)}{\pi \phi }  = \prod_{k=1}^\infty \left( 1 - \frac{\phi^2}{k^2} \right).
\end{equation*}
Thus, the sine-Gordon potential \eqref{W:sineG} can be written as 
\begin{align*}
W_\SG(\phi)=1-\cos \phi = 2 \sin^2\left( \frac \phi 2 \right)
=\frac 12 \phi^2 \prod_{k=1}^\infty \left(1-\frac{\phi^2}{(2\pi k)^2}\right)^2.
\end{align*}
For $n\geq 1$, setting
\begin{equation}\label{Wt:4n2}
\tilde W_{4n+2} (\phi )= \frac 12 \phi^2 \prod_{k=1}^n \left(1-\frac{\phi^2}{(2\pi k)^2}\right)^2 ,
\end{equation}
we obtain an approximation of the sine-Gordon potential, with $2n$ kinks denoted by
$\tilde H_{4n+2}\{2\pi j,2\pi (j+1)\}$, for $j=-n,\ldots,n-1$.
To study this perturbation, we use the formulation
\[
\tilde W_{4n+2}(\phi) = \left(1-\cos \phi\right) P_n(\phi)
\quad\mbox{where}\quad
P_n(\phi)=\prod_{k=n+1}^\infty \left(1-\frac{\phi^2}{(2\pi k)^2}\right)^{-2}.
\]
Then,
\[
-\frac{\tilde W_{4n+2}'}{\tilde W_{4n+2}}
= - \frac{\sin \phi}{1-\cos \phi}
- \frac 1{\pi^2}\sum_{k=n+1}^\infty \frac 1{k^2}\phi \left(1-\frac{\phi^2}{(2\pi k)^2}\right)^{-1} 
\]
and
\begin{equation*}
\tilde V_{4n+2}
=-\tilde W_{4n+2} \left(\frac{\tilde W_{4n+2}'}{\tilde W_{4n+2}}\right)'
=P_n  \left[1 -  A_n (1-\cos \phi)\right],
\end{equation*}
where
\[
A_n
=\left(\frac{P_n'}{P_n}\right)'
=\frac 1{\pi^2} \sum_{k=n+1}^\infty \frac 1{k^2}\left(1+\frac{\phi^2}{(2\pi k)^2}\right) \left(1-\frac{\phi^2}{(2\pi k)^2}\right)^{-2}.
\]
Thus,
\begin{align*}
\tilde V_{4n+2}'
&=P_n'\left[1 - A_n (1-\cos \phi) \right]
- P_n A_n \sin \phi - P_n A_n' (1-\cos \phi).
\end{align*}
The value of $\phi$ being fixed, for $n$ large, the following estimates hold
\[
P_n \sim 1,\quad P_n' \sim \frac \phi{\pi^2 n},\quad A_n\sim \frac 1{\pi^2 n},\quad
|A_n'|\leq \frac C{n^3},
\]
and so
\[
\tilde V_{4n+2}'(\phi) \sim  \frac 1{\pi^2 n} \left( \phi - \sin \phi \right).
\]
For any $J>0$, $\tilde V_{4n+2}'(\phi)$ has constant sign on $(-J,0)$ and $(0,J)$, for $n$ large enough, depending on $J$. Thus, the kinks of $\tilde W_{4n+2}$ with range in $(-J,J)$ are asymptotically stable for such $n$.

\begin{theorem}[The sine-Gordon limit for the $\phi^{4n+2}$ models]\label{th:Wt4n2}
For any $J\geq 1$, there exists $n(J)\geq 1$ such that for $n\geq n(J)$ and any $j\in \Z$ with $|j|\leq J$,
the kink $\tilde H_{4n+2}\{2\pi j,2\pi(j+1)\}$ of the $\phi^{4n+2}$ 
model with potential \eqref{Wt:4n2} is asymptotically stable.
\end{theorem}

\subsubsection{The $\tilde W_{4n}$ potentials as approximation of the sine-Gordon potential }\label{S:5.3.6}
Consider the shifted sine-Gordon potential
$ W_{\SSG}(\phi) = 1-\cos (\phi + \pi) = 1+\cos (\phi)$. For this model, there is an odd kink
connecting $-\pi$ to $\pi$ and infinitely many other identical translated kinks.
Recall 
\[
\cos(\pi \phi) = \frac 1 2 \frac{\sin(2\pi \phi)}{\sin (\pi \phi)} 
= \prod_{k=1}^\infty \left( 1 - \frac{\phi^2}{(k - \tfrac 12)^2} \right).
\]
Thus, the shifted sine-Gordon potential has the infinite product expansion
\begin{align*}
W_\SSG = 1+\cos \phi = 2 \cos^2\left( \frac \phi 2 \right)
= 2 \prod_{k=1}^\infty \left(1-\frac{\phi^2}{\pi^2(2k-1)^2}\right)^2.
\end{align*}
For $n\geq 1$, setting
\begin{equation}\label{Wt:4n}
\tilde W_{4n}(\phi )= 2 \prod_{k=1}^n \left(1-\frac{\phi^2}{\pi^2(2k-1)^2}\right)^2 ,
\end{equation}
we obtain another approximation of the sine-Gordon potential.
This model has an odd kink connecting $-\pi$ to $\pi$, denoted by $\tilde H_{4n}\{-\pi,\pi\}$,
$n-1$ positive kinks denoted by $\tilde H_{4n}\{\pi (2j-1),\pi (2j+1)\}$, for $j=1,\ldots,n-1$,
and $n-1$ negative kinks $\tilde H_{4n}\{-\pi (2j+1),-\pi (2j-1)\}$.
Write
\begin{equation*}
\tilde W_{4n}(\phi) = \left(1+\cos \phi\right) Q_n(\phi)
\quad\mbox{where}\quad
Q_n(\phi)=\prod_{k=n+1}^\infty \left(1-\frac{\phi^2}{\pi^2(2k-1)^2}\right)^{-2}.
\end{equation*}
Then,
\[
-\frac{\tilde W_{4n}'}{\tilde W_{4n}}
= \frac{\sin \phi}{1+\cos \phi}
- \frac 4{\pi^2}\sum_{k=n+1}^\infty \frac 1{(2k-1)^2}\phi \left(1-\frac{\phi^2}{\pi^2(2k-1)^2}\right)^{-1} 
\]
and
\begin{equation*}
\tilde V_{4n}
=-\tilde W_{4n} \left(\frac{\tilde W_{4n}'}{\tilde W_{4n}}\right)'
=Q_n  \left[1 -  B_n (1+\cos \phi)\right],
\end{equation*}
where
\[
B_n=\frac 4{\pi^2} \sum_{k=n+1}^\infty \frac 1{(2k-1)^2}\left(1+\frac{\phi^2}{\pi^2(2k-1)^2}\right) \left(1-\frac{\phi^2}{\pi^2(2k-1)^2}\right)^{-2}.
\]
Thus,
\begin{align*}
\tilde V_{4n}'
&=Q_n'\left[1 - B_n (1+\cos \phi) \right]
+Q_n B_n \sin \phi- Q_n B_n' (1+\cos \phi).
\end{align*}
The value of $\phi$ being fixed, for $n$ large, the following estimates hold
\[
Q_n \sim 1,\quad Q_n' \sim \frac \phi{\pi^2 n},\quad B_n\sim \frac 1{\pi^2 n},\quad
|B_n'|\leq \frac C{n^3},
\]
and so
\[
\tilde V_{4n}'(\phi) \sim  \frac 1{\pi^2 n} \left( \phi + \sin \phi \right).
\]
We observe that the criterion related to the condition~\eqref{on:V} is inconclusive for the odd kink
$\tilde H_{4n}\{-\pi,\pi\}$ since on the interval $(-\pi,\pi)$ the function $\tilde V_{4n}'$
is negative and then positive. For all the other kinks, the condition~\eqref{on:V} holds since the sign of
$\tilde V_{4n}'$ is constant on $(0,\infty)$ and on $(-\infty,0)$.

\begin{theorem}[The sine-Gordon limit for the $\phi^{4n}$ models]\label{th:Wt4n}
For any $J\geq 1$, there exists $n(J)\geq 1$ 
such that for $n\geq n(J)$ and any $j\in \N$ with $1\leq j\leq J$,
the kinks $\tilde H_{4n}\{\pi (2j-1),\pi(2j+1)\}$ and $\tilde H_{4n}\{-\pi (2j+1),-\pi(2j-1)\}$ of the $\phi^{4n}$
model with potential \eqref{Wt:4n} are asymptotically stable.
\end{theorem}

\subsection{The double sine-Gordon model}\label{S:5.4}
Last, following \cite{Campbell}, we consider a model with infinitely many wells which is related to the sine-Gordon equation.

First, for
$\eta>-\frac 14$, we consider the potential
\begin{align}
W_\DSG(\phi;\eta)
&=\frac{4}{1+4|\eta|}\left[\eta (1-\cos\phi)+1+\cos\left(\frac{\phi}{2}\right)\right]\label{W:DSG1}\\
&=\frac{4}{1+4|\eta|}\left[1+\cos\left(\frac{\phi}{2}\right)\right]
\left[1+2\eta \left(1-\cos\left(\frac{\phi}{2}\right)\right)\right].\nonumber
\end{align}
For $\eta\geq 0$, this potential interpolates between $1+\cos\left(\frac{\phi}{2}\right)$ for $\eta=0$ and $1-\cos \phi$ as $\eta\to \infty$.
Since the potential is periodic of period $4\pi$, it is sufficient to check the criterion for the kink $H_\DSG$ connecting $-2\pi$ with $2\pi$.
We compute
\[
-\frac{W_\DSG'}{W_\DSG}
= \frac 12 \frac{\sin\left(\frac{\phi}{2}\right)}{1+\cos\left(\frac{\phi}{2}\right)}
-\frac{\eta \sin\left(\frac{\phi}{2}\right)}{1+2\eta \left(1-\cos\left(\frac{\phi}{2}\right)\right)},
\]
and
\begin{align*}
V_\DSG
&=-W_\DSG\left(\frac{W_\DSG'}{W_\DSG}\right)'=\frac{1+4\eta}{1+4|\eta|}
-\frac{2\eta}{1+4|\eta|}
\frac{\left(1+\cos\left(\frac{\phi}2\right)\right)^2}{1+2\eta \left(1-\cos\left(\frac{\phi}{2}\right)\right)}.
\end{align*}
Thus,
\[
V_\DSG'(\phi)
= \frac{2\eta}{1+4|\eta|}
\left(1+\cos\left(\frac{\phi}2\right)\right)\sin\left(\frac{\phi}2\right)\frac{ 
 1+3\eta -\eta \cos\left(\frac{\phi}{2}\right)}
{\left[1+2\eta \left(1-\cos\left(\frac{\phi}{2}\right)\right)\right]^2}.
\]
We see that for $-\frac 14 <\eta<0$, it holds $V_\DSG'\not \equiv 0$ and $\phi V_\DSG'(\phi)\leq 0$ on $[-2\pi, 2\pi]$, thus the kink $H_\DSG$ is asymptotically stable by Theorem~\ref{th:2}.
For $\eta>0$, the criterion is inconclusive.

\begin{theorem}[Double sine-Gordon model I]\label{th:DSG1}
For any $-\frac 14<\eta <0$, 
the kink of the double sine-Gordon model with potential~\eqref{W:DSG1} is asymptotically stable.
\end{theorem}

Second, we consider the other case of double sine-Gordon models, designated as ``Region I'' in \cite[Fig.~I]{Campbell}. For $\eta<-\frac 14$, we set
\begin{align}
W_\DSG(\phi;\eta)
&=\frac{4}{1+4|\eta|}\left[- \eta \cos\phi +\cos\left(\frac{\phi}{2}\right) -\eta - \frac 1{8\eta}\right]
\label{W:DSG2}\\
&=\frac{8 |\eta|}{1+4|\eta|}\left[\cos\left(\frac{\phi}{2}\right)-\frac 1{4\eta}\right]^2.\nonumber
\end{align}
We recall from \cite{Campbell} that for this potential there exist two types of kinks.
Let $\zeta_\eta = 2 \arccos \frac 1{4\eta}\in (\pi,2\pi)$.
We denote by $H_\DSG\{-\zeta_\eta,\zeta_\eta\}$ the odd kink connecting $-\zeta_\eta$ to $\zeta_\eta$
and by $H_\DSG\{\zeta_\eta,4\pi - \zeta_\eta\}$ the kink connecting $\zeta_\eta$ to $4\pi - \zeta_\eta\in (2\pi,3\pi)$.
Other kinks for this model are deduced from these two kinks by translation and symmetry.
We compute
\[
-\frac{W_\DSG'}{W_\DSG}
=\frac {\sin\left(\frac{\phi}{2}\right)}{\cos\left(\frac{\phi}{2}\right)-\frac 1{4\eta}}
\]
and
\begin{align*}
V_\DSG
&=-W_\DSG\left(\frac{W_\DSG'}{W_\DSG}\right)'=\frac{4|\eta|}{1+4|\eta|}
\left[1-\frac 1{4\eta}\cos\left(\frac{\phi}2\right)\right].
\end{align*}
Thus,
\[
V_\DSG' = - \frac{1}{2(1+4|\eta|)}\sin\left(\frac{\phi}2\right).
\]
For the kink $H_\DSG\{-\zeta_\eta,\zeta_\eta\}$, the criterion \eqref{on:V} applies since $\phi V_\DSG'\leq 0$ on $[-\zeta_\eta,\zeta_\eta]$  and so $H_\DSG\{-\zeta_\eta,\zeta_\eta\}$ is asymptotically stable from Theorem~\ref{th:2}. For the kink $H_\DSG\{\zeta_\eta,4\pi - \zeta_\eta\}$, the criterion is inconclusive
since $V_\DSG'$ is negative on $(\zeta_\eta,2\pi)$ and positive on $(2\pi,4\pi-\zeta_\eta)$.

\begin{theorem}[Double sine-Gordon model II]\label{th:DSG2}
For any $\eta<-\frac 14$, 
the odd kink $H_\DSG\{-\zeta_\eta,\zeta_\eta\}$ 
of the double sine-Gordon model with potential~\eqref{W:DSG2}
is asymptotically stable.
\end{theorem}

\begin{remark}
We have not considered the case $\eta=-\frac 14$ since it corresponds to a degenerate potential, not entering the
general assumptions~\eqref{on:W}. See also (5) of Remark~\ref{rk:RS}.
\end{remark}

\end{document}